\documentclass[sort&compress,times,preprint,11pt]{elsarticle}

\usepackage{url}
\usepackage{xcolor}
\usepackage{graphicx,scalerel}
\newcommand\sbullet[1][.7]{\mathbin{\ThisStyle{\vcenter{\hbox{%
					\scalebox{#1}{$\SavedStyle\bullet$}}}}}%
}
\definecolor{gold}{rgb}{0.85, 0.65, 0.13}
\global\long\def\red#1{\textcolor{red}{#1}}%

\usepackage{enumitem}
\usepackage[hidelinks, colorlinks]{hyperref}
\usepackage{bookmark}
\bookmarksetup{
	numbered,
	addtohook={%
		\ifnum\bookmarkget{level}>1 %
		\bookmarksetup{numbered=false}%
		\fi
	},
}

\usepackage{framed,multirow}
\usepackage[margin=25mm]{geometry}
\global\bibsep=0pt

\usepackage{amssymb}

\usepackage{amsfonts}
\usepackage{mathrsfs}
\usepackage{upgreek}

\usepackage{amsmath, amsthm}
\usepackage{xspace}
\usepackage{bm}
\usepackage{physics}
\usepackage{stmaryrd}
\usepackage{cleveref}
\usepackage[labelfont=bf,textfont=normalfont]{caption}
\usepackage{subcaption}
\usepackage{mathtools}
\usepackage{booktabs}
\usepackage{multirow, multicol, makecell}
\usepackage{longtable}
\usepackage{threeparttable}

\newtheorem{theorem}{Theorem}

\newtheorem{assumption}{Assumption}

\newtheorem{definition}{Definition}

\crefname{assumption}{Assumption}{Assumptions}
\crefname{definition}{Definition}{Definitions}
\crefname{lemma}{Lemma}{Lemmas}
\crefname{remark}{Remark}{Remarks}
\crefname{theorem}{Theorem}{Theorems}
\crefname{proposition}{Proposition}{Propositions}
\crefname{section}{Section}{Sections}
\crefname{figure}{Fig.}{Figs.}
\crefname{equation}{}{}
\crefname{table}{Table}{Tables}
\crefname{appendix}{}{}

\newcommand{\fnc}[1]{\ensuremath{\mathcal{#1}}}
\newcommand{\vecfnc}[1]{\ensuremath{\boldsymbol{\mathcal{#1}}}} 

\newcommand{\uk}[0]{\ensuremath{\bm{{u}}_{k}}}

\renewcommand{\H}[0]{\mathsf{H}}
\newcommand{\D}[0]{\mathsf{D}}
\newcommand{\Q}[0]{\mathsf{Q}}
\newcommand{\E}[0]{\mathsf{E}}

\newcommand{\I}[0]{\mathsf{I}}

\renewcommand{\S}[0]{\mathsf{S}}
\newcommand{\J}[0]{\mathsf{J}}

\newcommand{\Z}[0]{\mathsf{Z}}

\newcommand{\B}[0]{\mathsf{B}}
\newcommand{\R}[0]{\mathsf{R}}
\newcommand{\N}[0]{\mathsf{N}}
\newcommand{\V}[0]{\mathsf{V}}

\newcommand{\Dxi}[0]{{\mathsf{D}}_{{x}_{i}}}
\newcommand{\Dxik}[0]{{\mathsf{D}}_{{x}_{i} k}}

\newcommand{\Sxi}[0]{{\mathsf{S}}_{{x}_{i}}}
\newcommand{\Sxik}[0]{{\mathsf{S}}_{{x}_{i}k}}

\newcommand{\Qxi}[0]{{\mathsf{Q}}_{{x}_{i}}}
\newcommand{\Qxik}[0]{{\mathsf{Q}}_{{x}_{i}k}}

\newcommand{\Exi}[0]{{\mathsf{E}}_{{x}_{i}}}
\newcommand{\Exik}[0]{{\mathsf{E}}_{{x}_{i}k}}

\newcommand{\Nxig}[0]{{\mathsf{N}}_{{x}_{i}\gamma}}

\newcommand{\Bg}[0]{\mathsf{B}_{\gamma}}

\newcommand{\Rgk}[0]{\mathsf{R}_{\gamma k}}

\newcommand{\etal}[0]{{et~al.\@}\xspace}
\newcommand{\eg}[0]{{e.g.\@}\xspace}
\newcommand{\ie}[0]{{i.e.\@}\xspace}

\newcommand{\ignore}[1]{} 

\newcommand{\polyref}[1]{\ensuremath{\mathbb{P}^{#1}}(\hat{\Omega})}

\newcommand{\contref}[1]{\ensuremath{\mathcal{C}^{#1}}(\hat{\Omega})}

\newcommand{\poly}[1]{\ensuremath{\mathbb{P}^{#1}}({\Omega}_k)}

\newcommand{\cont}[1]{\ensuremath{\mathcal{C}^{#1}}({\Omega}_k)}

\newcommand{\vecLtwo}[1]{[\ensuremath{L^2}({\Omega}_k)]^{#1}}

\newcommand{\IR}[1]{\mathbb{R}^{#1}}%
\newcommand{\IRtwo}[2]{\mathbb{R}^{{#1}\times{#2}}}%

\newcommand{\pder[2]}[]{\dfrac{\partial#1}{\partial#2}}%
\newcommand{\der[2]}[]{\dfrac{{\rm d}#1}{{\rm d}#2}}%
\global\long\def\norm#1{\left\vert \left\vert #1\right\vert \right\vert }%
\global\long\def\fn#1{\mathcal{#1}}%
\global\long\def\mathds#1{\mathds{#1}}%

\makeatletter
\let\overlinewithoriginalheight\overline
\newcommand*\overlinewithlessheight[1]{{\mathpalette\overline@aux{#1}}}
\newcommand*\overline@aux[2]{
	\begingroup
	\count0=\fam 
	\setbox0=\hbox{$\m@th #1\fam=\count0 #2$}
	\@tempdima=.4\ht0
	\setbox0=\hbox{$\m@th #1\fam=\count0\overlinewithoriginalheight{#2}$}%
	\advance\@tempdima by .6\ht0
	\ht0=\@tempdima 
	\usebox0
	\endgroup
}
\let\overline\overlinewithlessheight

\begin{document}
	\begin{frontmatter}
		
	\title{Tensor-Product Split-Simplex Summation-By-Parts Operators} 
	
	\author[1]{Zelalem Arega {Worku}\corref{cor1}}
	\cortext[cor1]{Corresponding author: } 
	\ead{zelalem.worku@mail.utoronto.ca}

	\author[2]{Jason E. {Hicken}}
	\ead{dwz@oddjob.utias.utoronto.ca}

	\author[1]{David W. {Zingg}}
	\ead{dwz@oddjob.utias.utoronto.ca}

	\address[1]{Institute for Aerospace Studies, University of Toronto, 
					4925 Dufferin St, Toronto, ON, 
					M3H 5T6,
					Canada}

	\address[2]{Department of Mechanical, Aerospace, and Nuclear Engineering, Rensselaer Polytechnic Institute,
					110 8th St, Troy, NY,
					2180-3590, 
					USA}
	
	\addcontentsline{toc}{section}{Abstract}
	\begin{abstract}
		We present an approach to construct efficient sparse summation-by-parts (SBP) operators on triangles and tetrahedra with a tensor-product structure. The operators are constructed by splitting the simplices into quadrilateral or hexahedral subdomains, mapping tensor-product SBP operators onto the subdomains, and assembling back using a continuous-Galerkin-type procedure. These tensor-product split-simplex operators do not have repeated degrees of freedom at the interior interfaces between the split subdomains. Furthermore, they satisfy the SBP property by construction, leading to stable discretizations. The accuracy and sparsity of the operators substantially enhance the efficiency of SBP discretizations on simplicial meshes. The sparsity is particularly important for entropy-stable discretizations based on two-point flux functions, as it reduces the number of two-point flux computations. We demonstrate through numerical experiments that the operators exhibit efficiency surpassing that of the existing dense multidimensional SBP operators by more than an order of magnitude in many cases. This superiority is evident in both accuracy per degree of freedom and computational time required to achieve a specified error threshold.
	\end{abstract}
	
	
\end{frontmatter}

\section{Introduction}
High-order methods can attain comparable accuracy as low-order methods on substantially coarser grids \cite{wang2013high}, and they are well-suited for modern computer architectures due to their high arithmetic intensity and data locality. However, despite their promising efficiency, high-order methods often suffer from robustness issues. A successful approach to mitigate this issue is to construct high-order methods such that they satisfy a property known as summation-by-parts (SBP). The SBP property is a discrete analogue of integration by parts and plays a crucial role in the development of provably stable high-order discretizations of partial differential equations (PDEs). The development of SBP operators and simultaneous approximation terms (SATs), which are weak interface coupling and boundary condition enforcement terms, is described in the review papers \cite{fernandez2014review,svard2014review}.  Furthermore, the development of various types of SBP operator on simplicial meshes can be found in \cite{hicken2016multidimensional,fernandez2018simultaneous,chen2017entropy,crean2018entropy}. 

While several factors affect the efficiency of a high-order method of a given order, in general, accuracy and sparsity are among the most important. Two methods of the same order can have orders of magnitude differences in the actual solution error values they produce. Furthermore, a sparse operator with a large number of degrees of freedom may require substantially fewer floating-point operations than dense operators with fewer degrees of freedom. Tensor-product SBP operators offer higher sparsity and accuracy compared to multidimensional SBP operators. Consequently, they are the predominant choice for implementing high-order methods, particularly for entropy-stable schemes, which are computationally demanding. However, simplicial meshes are favored when geometric flexibility is important. As a result, there has been some effort to develop efficient high-order SBP operators on simplicial elements. One approach is to derive quadrature rules with fewer nodes, \eg, \cite{worku2023quadrature}. An alternative is to develop SBP operators with tensor-product structure by applying a collapsed coordinate transformation \cite{sherwin1995triangular,sherwin1996tetrahedralhpfinite}, an approach that has recently been extended to incorporate the SBP property \cite{montoya2023efficient} and provable entropy stability \cite{montoya2024efficient}. While this is a promising approach toward improving the efficiency of entropy-stable schemes on simplices, we explore here an alternative approach toward the same goal. Another method to applying tensor-product operators on a given simplicial mesh involves splitting the simplicial mesh into quadrilateral and hexahedral subdomains and applying tensor-product operators on these split subdomains, \eg, \cite{durufle2009influence}. This approach is generally avoided for finite-element methods, as the resulting mesh can be of poor quality \cite{owen1998survey}. Additionally, discretizations on the split subdomains yield a reduced order of accuracy compared to those on affinely mapped quadrilateral or hexahedral meshes \cite{durufle2009influence,castel2009application}. Despite these drawbacks, the approach remains in use as it provides a relatively simple means of obtaining the benefits of tensor-product operators on simplicial meshes \cite{dalcin2019conservative,al2023evaluation}.

In this article, we present a novel approach to construct sparse SBP operators on simplicial elements by combining the idea of splitting simplicial meshes into quadrilateral or hexahedral subdomains with the idea of assembling SBP operators in a continuous-Galerkin-type approach \cite{hicken2016multidimensional}. The splitting is used only for the construction of the reference SBP operators on the simplex, which are referred to as tensor-product split-simplex (TPSS) operators; thus the approach does not require actual splitting of the simplicial mesh generated for a given computational domain. Unlike the method of splitting the physical simplex elements, there are no repeated degrees of freedom inside the reference TPSS element; consequently, there is no need to use numerical fluxes or SATs at the shared interfaces of the split subdomains. Hence, once constructed, TPSS operators are used in the same manner as any other multidimensional SBP operator. Finally, as a consequence of careful application of the mapping metric terms and assembly procedure, the TPSS operators satisfy the SBP property and thus lead to discretely stable discretizations, a property that is not shared by all methods applying the simplex splitting approach. 

The rest of the paper is organized as follows. \cref{sec:preliminaries} presents some notation and definitions. \cref{sec:TPSS construction} describes the construction procedures of the TPSS operators. The accuracy and sparsity analysis as well as the formal definition of the TPSS operators are presented in \cref{sec:acc and sparsity of TPSS}. Finally, the efficiency and stability of the TPSS operators are verified via numerical results in \cref{sec:numerical results}, and conclusions are presented in \cref{sec:conclusions}.

\section{Preliminaries}\label{sec:preliminaries}
\subsection{Notation and definitions}\label{subsec:notation}
In this paper, spatial discretizations are handled using element-type SBP operators. The spatial domain, $ \Omega $, is assumed to be compact and connected, and it is tessellated into $ n_e $ elements, $ {\mathcal T}_h \coloneqq \{\{ \Omega_k\}_{k=1}^{n_e}: \Omega=\cup_{k=1}^{n_e} {\Omega}_k\}$. The boundaries of each element are assumed to be piecewise smooth and will be referred to as facets or interfaces, and the union of the facets of element $ \Omega_{k} $ is denoted by $ \Gamma_k \coloneqq \partial\Omega_{k} $. The set of $ n_p $ volume nodes in element $ \Omega_k $ is represented by $ S_{\Omega_{k}}=\{\bm{x}^{(j)}\}_{j=1}^{n_p} $, where $\bm{x}=[x_{1},\dots,x_{d}]^T$ denotes the Cartesian coordinates on the physical domain, $\bm{x}^{(j)}$ is the tuple of the coordinates of the $j$-th node, and $d$ is the spatial dimension. The number of nodes on facet $ \gamma \in \Gamma_{k}$ is denoted by $ n_f $, and the set of nodes on $\gamma$ are represented by $S_{\gamma}$. Scalar functions over element $\Omega_k$ are written in uppercase script type, \eg, $\fnc{U}_k \in \cont{\infty}$, and vector-valued functions of dimension $ n $ are represented by boldface uppercase script letters, \eg, $\vecfnc{W}_k \in \vecLtwo{n}$. The space of polynomials of total degree $ p $ is denoted by $\poly{p} $. Vectors containing grid function values are denoted by bold letters, \eg, $ \uk \in \IR{n_p}$. We define $h \coloneqq \max_{a, b \in S_{\Omega_{k}}} \norm{a - b}_2$ as the nominal element size. Matrices are denoted by sans-serif uppercase letters, \eg, $\V \in \IRtwo{n_p}{n_p}$; $ \bm{1} $ denotes a vector consisting of all ones, $ \bm {0} $ denotes a vector or matrix consisting of all zeros. The sizes of  $ \bm{1} $ and $ \bm {0} $ should be clear from context. The identity matrix of size $ n \times n $ is denoted by $ \I_{n} $. 

The line, triangle, and tetrahedron reference elements are defined, respectively, as 
\begin{align} \label{eq:ref elems}
	&\widehat{\Omega}_{\text{line}}=\{\xi_{1} \mid -1\le \xi_{1}\le 1\},\\
	&\widehat{\Omega}_{\text{tri}}=\{\left(\xi_{1},\xi_{2}\right)\mid \xi_{1},\xi_{2}\ge-1; \; \xi_{1}+\xi_{2}\le0\},\\
	&\widehat{\Omega}_{\text{tet}}=\{\left(\xi_{1},\xi_{2},\xi_{3}\right)\mid \xi_{1},\xi_{2},\xi_{3}\ge-1; \; \xi_{1}+\xi_{2}+\xi_{3}\le-1\}, 
\end{align}
where $\bm{\xi}=[\xi_{1},\dots,\xi_{d}]^T$ denotes the coordinate system on the reference element. Equilateral reference triangle and tetrahedron are used for illustrations; there is a bijective mapping between the standard and equilateral reference elements. The boundary of the reference element is denoted by $\widehat{\Gamma}$. 

An SBP operator on a reference element, $\widehat{\Omega}$, is defined as follows \cite{hicken2016multidimensional}.
\begin{definition} \label{def:sbp}
	The matrix $\Dxi\in \IRtwo{n_p}{n_p}$ is a degree $ p $ SBP operator approximating the first derivative $ \pdv{{\xi}_i} $ on the set of nodes $ {S}_{\widehat{\Omega}}=\{\bm{\xi}^{(j)}\}_{j=1}^{n_p} $ if
	\begin{enumerate}[label={\arabic*.}]
		\setlength\itemsep{0.05em}
		\item $[ \Dxi \bm{p}]_j = {\pdv{\fnc{P}}{\xi_{i}}}({\bm{\xi}^{(j)}}) $  for all $\fnc{P} \in \polyref{p} $
		\item $ \Dxi={\H}^{-1} \Qxi $ where $ {\H}$ is a symmetric positive definite (SPD) matrix, and 
		\item $ \Qxi = \Sxi + \frac{1}{2} \Exi$ where $ \Sxi= - \Sxi^T $, $\bm{p}^T \Exi  \bm {q} = \sum_{\widehat{\gamma}\in\widehat{\Gamma}}\int_{\widehat{\gamma}} \fnc{P}\fnc{Q} \;{n}_{\xi_i} \dd{{\Gamma}}$
		$\;\forall \fnc{P},\fnc{Q} \in \polyref{r\ge  p} $, and $ {n}_{\xi_{i}} $ is the $ \xi_{i}$-component of the outward pointing unit normal vector on facet $\widehat{\gamma}$.
	\end{enumerate}  
\end{definition} 

The third property in \cref{def:sbp} implies that 
\begin{equation}\label{eq:sbp property}
	\Qxi + \Qxi^T = \Exi,
\end{equation}
which is referred to as the SBP property --- the most fundamental property of SBP operators that extends integration by parts to a discrete setting.

We exclusively consider SBP operators with diagonal norm matrices with extension to dense-norm operators deferred to future work. A diagonal norm matrix contains positive weights of a quadrature rule of degree at least $ 2p-1 $ on its diagonal. Hence, it approximate an inner product,
\begin{equation}
	\bm{p}^T \H \bm{q} = \int_{\widehat{\Omega}} \fn{P}\fn{Q} \dd{\Omega}, \quad \forall \;\fn{P}\fn{Q} \in \polyref{r}, \quad r\le 2p-1,
\end{equation}
and is used to define the $ \H $-norm, $ \norm{\bm{u}}_{\H}^2 = \bm{u}^T \H \bm{u}$. 

\subsection{SBP operators on curved elements}
Assume that we have SBP operators on a certain reference element; for example, tensor-product SBP operators. The construction of element-type SBP operators can be found in the literature, \eg, \cite{fernandez2014generalized,gassner2013skew,hicken2016multidimensional,fernandez2018simultaneous,crean2018entropy,chen2017entropy,glaubitz2023multi}. Operators are usually constructed on a reference element and are mapped to physical elements in the computational domain, which may be curved. We make the following assumption regarding the geometric mapping from the reference to the physical elements. 
\begin{assumption}\label{ass:mapping}
	We assume that there is a bijective polynomial mapping of degree $ p_{\text{geom}}$ from the reference to the physical elements. Furthermore, we assume that $ p_{\text{geom}} \le p+1 $ and $ p_{\text{geom}}\le \lfloor p/2 \rfloor +1 $ in two and three dimensions, respectively, where $ p $ is the degree of the SBP operator and $\lfloor \cdot \rfloor$ is the floor operator.
\end{assumption}

The limits on the polynomial degree of the mapping function are imposed to simplify our discussion for a reason clarified below; in general, SBP operators can be implemented with higher-order geometric mappings. Since the geometric mapping is polynomial by \cref{ass:mapping}, we can compute the exact mapping Jacobian matrix at point $\bm{\xi}^{(m)}$ as
\begin{equation}
	\begin{aligned}
		[\fn{J}^{(m)}_{k}]_{ij} &= \pder[x_{i}]{\xi_{j}}(\bm{\xi}^{(m)}), \quad \forall \bm{\xi}^{(m)}\in S_{\widehat{\Omega}},  \;\; i \in \{1,\dots,d\}, \; \; j\in\{1,\dots,d\},			
	\end{aligned}
\end{equation}
and its determinant is denoted by $|\fn{J}^{(m)}_{k}|$. The determinant of the Jacobian is assumed to be positive at each node. SBP operators on the physical element, $\Omega_{k}$, can now be constructed as \cite{crean2018entropy}
\begin{align}
		[\H_{k}]_{mm} &= [\H]_{mm}|\fn{J}^{(m)}_{k}|, \label{eq:norm matrix}
		\\
		[\Nxig]_{mm} &= \sum_{j=1}^{d}n_{\xi_{j}}|\fn{J}^{(m)}_{k}|\pder[\xi_{j}]{x_i}(\bm{\xi}^{(m)}),\label{eq:normal}
		\\
		\Exik &= \sum_{\gamma\in\Gamma_{k}} \E_{x_{i}}^{kk}, \quad \E_{x_{i}}^{k k}=\Rgk^T \Bg \Nxig \Rgk, \label{eq: Exik} 
		\\
		[\Sxik]_{mn} & = \frac{1}{2}\sum_{j=1}^{d}\left(|\fn{J}^{(m)}_{k}|\pder[\xi_{j}]{x_{i}}(\bm{\xi}^{(m)})[\Q_{\xi_j}]_{mn}
		- |\fn{J}^{(n)}_{k}|\pder[\xi_{j}]{x_{i}}(\bm{\xi}^{(n)})[\Q_{\xi_j}]_{nm} \right),\label{eq:Sxik}
		\\
		\Qxik &= \Sxik + \frac{1}{2}\Exik,\label{eq:Qxik}
		\\
		\Dxik & = \H_{k}^{-1}\Qxik,\label{eq:Dxik}
\end{align}
for all $\bm{\xi}^{(m)}, \bm{\xi}^{(n)} \in S_{\widehat{\Omega}} $ and $i=\{1,\dots,d\}$.
The extrapolation matrix, $\R_{\gamma k} \in \IR{n_f\times n_p}$, is exact for constant functions in the physical element, $ \Omega_k $, particularly $ \Rgk \bm{1} = \bm{1} $. For operators with collocated volume and facet nodes, referred to as SBP diagonal-$\E$ operators \cite{chen2017entropy}, $\R_{\gamma k}$ contains unity at row (facet) and column (volume) indices corresponding to each collocated node, and all other entries are zero. Polynomials in $ \Omega_k $ are not necessarily polynomials in the reference element, $ \widehat{\Omega} $; thus, SBP operators in the physical domain are not exact for polynomials in $ \Omega_k $. However, under \cref{ass:mapping} (or a careful application of high-order geometric mappings \cite{crean2018entropy}) the accuracy of the derivative operators on the physical elements is not compromised \cite{crean2018entropy}. We state, without proof, Theorem 9 in \cite{crean2018entropy} that establishes the accuracy of the SBP derivative operator on a physical element. 
\begin{theorem}\label{thm:Accuracy of Dx}
	Let \cref{ass:mapping} hold and the metric terms be computed exactly, then for $ \bm{u}_k\in\IR{n_p} $ holding the values of $ \fnc{U}\in\contref{p+1} $ at the nodes $ S_{\Omega_{k}} $, the derivative operator given by \cref{eq:Dxik} is order $ p $ accurate, \ie,
	\begin{equation*}
		[\Dxik\bm{u}_k]_{m} = {\pdv{\fnc{U}}{x_{i}}} (\bm{\xi}^{(m)}) +\order{h^{p}}, \quad i\in {1,\dots,d}.
	\end{equation*}
\end{theorem}

Freestream preservation and entropy conservation require that $\Dxik \bm{1} = \bm{0}$, which is equivalent to satisfying the discrete form of the metric invariants \cite{vinokur1974conservation,thomas1979geometric,kopriva2006metric,crean2018entropy,chan2019discretely,shadpey2020entropy}, 
\begin{equation}
	\sum_{i=1}^{d}\pder[]{\xi_{i}}\left(|\fn{J}| \pder[\xi_{i}]{x_{j}}\right) = 0. 
\end{equation}
By limiting the degree of the mapping, \cref{ass:mapping} ensures that $|\fn{J}| \pder[\xi_{i}]{x_{j}}$ is a degree $p$ polynomial in the reference space allowing the discrete metric invariants to be satisfied on curved elements \cite{crean2018entropy,shadpey2020entropy}. If a higher-degree polynomial mapping is needed to represent a three-dimensional geometry, then one may employ the strategy outlined in \cite{crean2018entropy}, which involves solving elementwise quadratic optimization problems. 

\section{Construction of TPSS operators}\label{sec:TPSS construction}
TPSS operators are constructed by mapping tensor-product operators onto quadrilateral or hexahedral subdomains of the reference triangle or tetrahedron, respectively, and patching together the tensor-product SBP operators on each subdomain using a continuous-Galerkin-type assembly. The splitting of the triangle into three quadrilaterals is achieved by connecting the centroid to the mid-points of each edge, while for the tetrahedron, lines connecting the mid-points of each edge to the centroids of the facets sharing the edge and the lines connecting the facet centroids to the centroid of the tetrahedron create planes partitioning the tetrahedron into four hexahedra. These splittings are illustrated in \cref{fig:simplex splitting} and each subdomain will be denoted by $\widehat{\Omega}^{(\ell)}$, where $\ell\in\{1,\dots,L\}$ and $L=d+1$ denotes the number of split subdomains. 
\begin{figure}[t]
	\centering
	\begin{subfigure}{0.5\textwidth}
		\centering
		\includegraphics[scale=0.28]{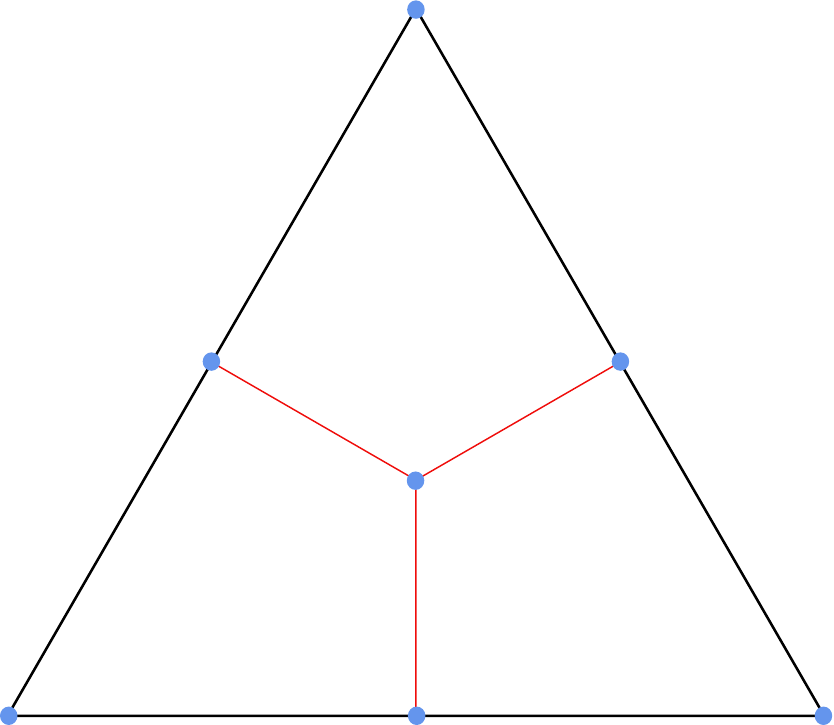}
	\end{subfigure}\hfill
	\begin{subfigure}{0.5\textwidth}
		\centering
		\includegraphics[scale=0.018]{tet_split.pdf}
	\end{subfigure} 
	\caption{\label{fig:simplex splitting} Splitting the triangle and tetrahedron into three quadrilaterals and four hexahedra, respectively.}
\end{figure}

The mapping of any point in the reference quadrilateral or hexahedral element, shown in \cref{fig:quad hex nod numbering}, to the quadrilateral or hexahedral elements in the split simplex elements is bilinear or trilinear, respectively. 
\begin{figure}[t]
	\centering
	\begin{subfigure}{0.5\textwidth}
		\centering
		\includegraphics[scale=0.28]{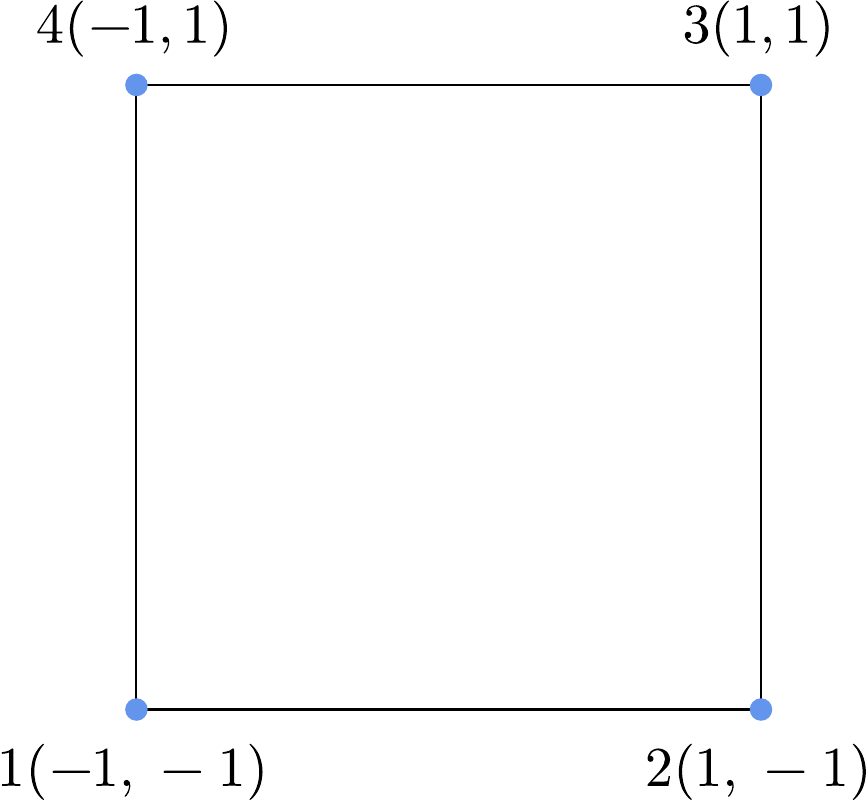}
	\end{subfigure}\hfill
	\begin{subfigure}{0.5\textwidth}
		\centering
		\includegraphics[scale=0.38]{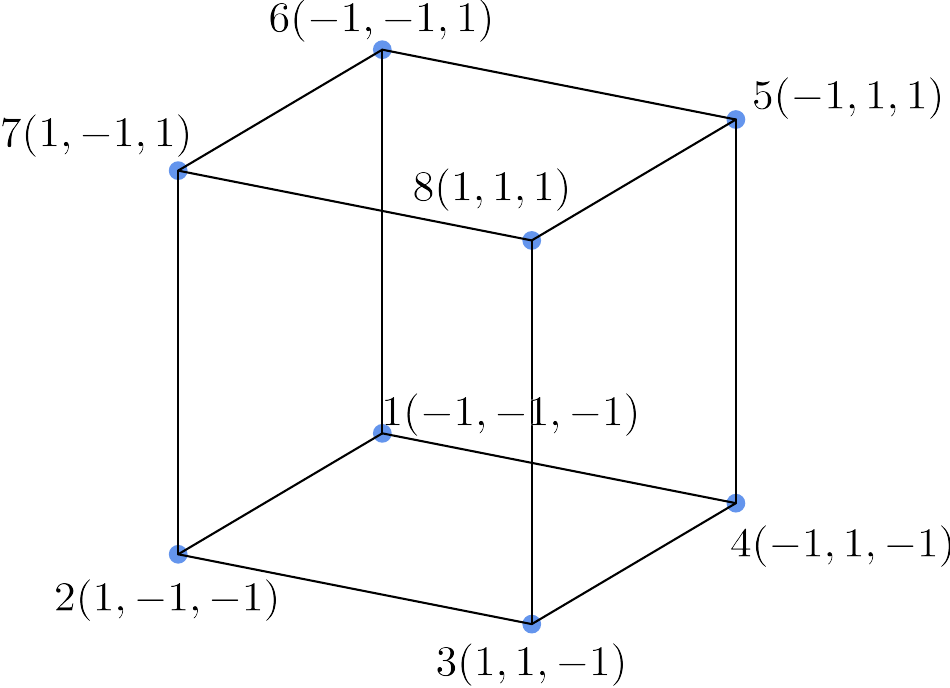}
	\end{subfigure} 
	\caption{\label{fig:quad hex nod numbering} Quadrilateral and hexahedral reference elements and their node numbering.}
\end{figure}
We denote the coordinates on the reference quadrilateral and hexahedral elements by $\bm{\eta}=\{\eta_1,\dots,\eta_{d}\}$, while keeping $\bm{\xi}$ and $\bm{x}$ for the coordinates on the reference simplices and the physical elements, respectively. The bilinear and trilinear mappings of any point in the reference quadrilateral or hexahedron to the quadrilaterals or hexahedra in the split simplex can be written as 
\begin{equation}\label{eq:bilinear trilinear mapping}
	\bm{\xi}^{(\ell)}=(\xi_1^{(\ell)},\dots,\xi_d^{(\ell)}) = \sum_{\alpha=1}^{2^d}P^{(\ell)}_{\alpha}\Psi_{\alpha}\left(\eta_{1},\dots,\eta_{d}\right), \qquad \ell \in\{1,\dots,L\},
\end{equation}
where $P^{(\ell)}_{\alpha}$ is the $\alpha$-th vertex coordinate tuple of the $\widehat{\Omega}^{\ell}$ subdomain in the split simplex, and $\Psi_{\alpha}$ is the $\alpha$-th vertex bilinear or trilinear shape function. For the node numbering shown in \cref{fig:quad hex nod numbering}, the bilinear shape functions on the reference quadrilateral are given by
\begin{equation}\label{eq:bilinear Psi}
	\medmuskip=-0.5mu
	\Psi_{1}=\frac{1}{4}\left(1-\eta_{1}\right)\left(1-\eta_{2}\right),\quad\Psi_{2}=\frac{1}{4}\left(1+\eta_{1}\right)\left(1-\eta_{2}\right),\quad\Psi_{3}=\frac{1}{4}\left(1+\eta_{1}\right)\left(1+\eta_{2}\right),\quad\Psi_{4}=\frac{1}{4}\left(1-\eta_{1}\right)\left(1+\eta_{2}\right).
\end{equation}  
Similarly, the trilinear shape functions on the reference hexahedron are 
\begin{equation}\label{eq:trilinear Psi}
	\medmuskip=-0mu
	\begin{aligned}
		\Psi_{1}&=\frac{1}{8}\left(1-\eta_{1}\right)\left(1-\eta_{2}\right)\left(1-\eta_{3}\right),&\Psi_{5}&=\frac{1}{8}\left(1-\eta_{1}\right)\left(1+\eta_{2}\right)\left(1+\eta_{3}\right),
		\\
		\Psi_{2}&=\frac{1}{8}\left(1+\eta_{1}\right)\left(1-\eta_{2}\right)\left(1-\eta_{3}\right),&\Psi_{6}&=\frac{1}{8}\left(1-\eta_{1}\right)\left(1-\eta_{2}\right)\left(1+\eta_{3}\right),
		\\
		\Psi_{3}&=\frac{1}{8}\left(1+\eta_{1}\right)\left(1+\eta_{2}\right)\left(1-\eta_{3}\right),&\Psi_{7}&=\frac{1}{8}\left(1+\eta_{1}\right)\left(1-\eta_{2}\right)\left(1+\eta_{3}\right),
		\\
		\Psi_{4}&=\frac{1}{8}\left(1-\eta_{1}\right)\left(1+\eta_{2}\right)\left(1-\eta_{3}\right),&\Psi_{8}&=\frac{1}{8}\left(1+\eta_{1}\right)\left(1+\eta_{2}\right)\left(1+\eta_{3}\right).
	\end{aligned}
\end{equation}
The bilinear and trilinear shape functions are polynomial functions; hence, the mapping from the reference quadrilateral and hexahedron to the quadrilateral and hexahedral subdomains in the split simplices is a polynomial of total degree two and three, respectively. This allows straightforward computation of the metric terms and entries of the mapping Jacobian matrix,
\begin{equation}
	\fn J_{ij}^{(\ell)}=\pder[\xi_{i}^{(\ell)}]{\eta_{j}}=\sum_{\alpha=1}^{2^{d}}P^{(\ell)}_{\alpha,i}\pder[\Psi_{\alpha}]{\eta_{j}},  \qquad \ell \in\{1,\dots,d+1\},
\end{equation} 
where $P^{(\ell)}_{\alpha,i} $ denotes the value of the $i$-th coordinate of the $\alpha$-th vertex on the $\ell$-th subdomain. The fact that the mapping in \cref{eq:bilinear trilinear mapping} is not affine has detrimental consequences on the accuracy of the TPSS operators, as discussed below in \cref{sec:acc and sparsity of TPSS}. 

Given a one-dimensional SBP operator, it is straightforward to generate tensor-product operators on the reference quadrilateral and hexahedral elements. For convenience, we will make the assumption that the tensor-product operator has the same number of nodes in each direction and that the operator belongs to the SBP diagonal-$\E$ family, but these assumptions can be relaxed if necessary. Let the one-dimensional SBP operators be represented by a subscript $\rm 1D$, \eg, $\H_{\rm 1D}$; then the tensor-product SBP operators on reference quadrilateral are given by 
\begin{equation}
	\medmuskip=-0mu
	\H=\H_{{\rm 1D}}\otimes\H_{{\rm 1D}},\quad\Q_{\eta_{1}}=\H_{{\rm 1D}}\otimes\Q_{{\rm 1D}},\quad\Q_{\eta_{2}}=\Q_{{\rm 1D}}\otimes\H_{{\rm 1D}},\quad\R_{\gamma_{1}}=\I_{n_1}\otimes\bm{t}_{L}^T,\quad\R_{\gamma_{4}}=\bm{t}_{R}^T\otimes\I_{n_1},
\end{equation}
where $\bm{t}_{L}$ and $\bm{t}_{R}$ are one-dimensional extrapolation column vectors of size $n_1$ with a value of unity at the first and last entries, respectively, $\gamma_{1}$ is the facet connecting nodes $ 4$ and $1$ in \cref{fig:quad hex nod numbering}, while $\gamma_{4}$ is the facet connecting nodes $3$ and $4$. The extrapolation operators for the remaining facets are obtained in a similar manner. Furthermore, the remaining operators, $\D_{\eta_{i}}$, $\E_{\eta_i}$, and $\S_{\eta_i}$ can be derived using the relations in \cref{def:sbp}. The tensor-product SBP operators on the hexahedron are derived in the same manner, \eg, 
\begin{equation}
	\H=\H_{{\rm 1D}}\otimes\H_{{\rm 1D}}\otimes\H_{{\rm 1D}},\quad\Q_{\eta_{1}}=\H_{{\rm 1D}}\otimes\H_{{\rm 1D}}\otimes\Q_{{\rm 1D}},\quad\R_{\gamma_{1}}=\I_{n_1}\otimes\I_{n_1}\otimes\bm{t}_{L}^T,
\end{equation}
where $\gamma_{1}$ is the facet defined by the vertices $1,2,7$ and $6$ in \cref{fig:quad hex nod numbering}. With the tensor-product SBP operators and metric terms computed, we can now construct the SBP operators on each of the quadrilateral and hexahedral subdomains in the split simplices using the relations given in \cref{eq:norm matrix,eq:normal,eq: Exik,eq:Sxik,eq:Qxik,eq:Dxik}.

TPSS-SBP operators are constructed by combining the respective tensor-product SBP operators on the subdomains of the reference simplices in such a way that the SBP property and accuracy are preserved. This is accomplished using the continuous-Galerkin-type assembly of Hicken \etal, \cite{hicken2016multidimensional,hicken2020entropy,hicken2021entropy}. One can also construct TPSS operators by splitting each physical element in a mesh, mapping tensor-product operators to the split subdomains, and assembling the resulting operators on the physical elements according to the procedure in \cite{hicken2016multidimensional}; however, this approach is not pursued in this work. \ignore{However, this approach would be more expensive compared to constructing TPSS operators on a single reference element and mapping them to the physical elements.  }

For completeness, we describe the operator assembly procedure of \cite{hicken2016multidimensional}. Before proceeding, however, we need to introduce some notation and operators. Let the set of nodes on the subdomain $\widehat{\Omega}^{(\ell)}$ be denoted by $S^{(\ell)} = \{\bm{\xi}^{(\ell,j)}\}_{j=1}^{n_1^d}$ and the set of nodes on the simplex be $S=\cup S^{(\ell)}$. Note that $S$ contains repeated nodes at the shared interfaces of the subdomains. We assign a global node index to each unique node and denote the total number of unique nodes by $n_p$ and the set containing them by $S_{\widehat{\Omega}}$. Let $\tilde{\imath}$ and $\tilde{\jmath}$ be the global node indices corresponding to the local node indices $i$ and $j$ on $\widehat{\Omega}^{(\ell)}$. Define the matrix $\Z^{(\ell)}(i,j) $ to be an $n_{p}\times n_{p}$ matrix with a value of unity at entry $ (\tilde{\imath},\tilde{\jmath})$ and zero elsewhere. The assembly procedure and preservation of the SBP property and accuracy are stated in the following theorem, the proof of which can be found in \cite{hicken2016multidimensional}. 
\begin{theorem}\label{thm:assembly}
	Let $\D^{(\ell)}_{\xi_{l}} =[ \H^{(\ell)}]^{-1}\Q_{\xi_{l}}^{(\ell)}$ be a degree $p$ SBP operator approximating the first derivative $\pder[]{\xi_{l}}$ on the nodal set $S^{(\ell)}$.  If we define 
	{
		\setlength{\abovedisplayskip}{0pt}
		\begin{align*}
			\H &\coloneqq \sum_{\ell=1}^{L} \sum_{i=1}^{n_1^d} \sum_{j=1}^{n_1^d}\H_{ij}^{(\ell)} \Z^{(\ell)}(i,j), \\
			\Q_{\xi_{l}} &\coloneqq \sum_{\ell=1}^{L} \sum_{i=1}^{n_1^d} \sum_{j=1}^{n_1^d}(\Q_{\xi_{l}}^{(\ell)})_{ij} \Z^{(\ell)}(i,j),
		\end{align*}
	}
	then $\D_{\xi_{l}}=\H^{-1}\Q_{\xi_{l}} $ is a degree $p$ SBP operator on the global nodal set $S_{\widehat{\Omega}}$.
\end{theorem}

Simply stated, \cref{thm:assembly} implies that shuffling the SBP operators on each subdomain, $\widehat{\Omega}^{(\ell)}$, in a manner consistent with the global node ordering and adding the resulting matrices will produce an accurate SBP operator on the reference simplex element, $\widehat{\Omega}$. The $\Z^{(\ell)}(i,j)$ matrix is a binary matrix that rearranges the entries of the SBP operators such that they can be applied to the global nodes. To illustrate the procedure with an example, consider the construction of the TPSS-SBP operator on the reference triangle using the degree $p=1$ one-dimensional Legendre-Gauss-Lobatto (LGL) operator.
\begin{figure}[t]
	\centering
	\begin{subfigure}{0.5\textwidth}
		\centering
		\includegraphics[scale=0.28]{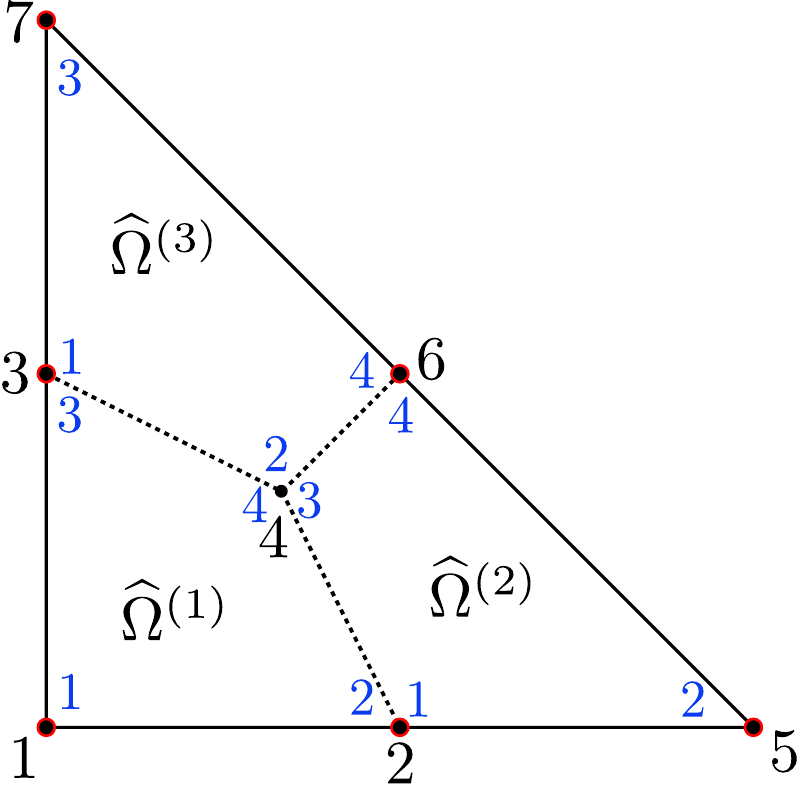}
	\end{subfigure}
	\caption{\label{fig:node ordering} Local and global node ordering for construction of the TPSS-SBP operator using the degree $p=1$ LGL one-dimensional operator. The global node ordering is shown in larger font size.}
\end{figure}
As an example, we show the $\Q_{\xi_{1}}^{(2)}$ matrix on the quadrilateral $\widehat{\Omega}^{(2)}$ in \cref{fig:node ordering} and its reordering after applying the $\Z^{(2)}(i,j)$ matrix, accurate to four decimal places, 
\begin{equation}
	\small
	\setlength{\arraycolsep}{3pt}
	\begin{aligned}
		\Q_{\xi_{1}}^{(2)} = 
		\begin{bmatrix*}[l]
			-0.1667 & 0.2083 & -0.0417 & 0 \\
			-0.2083 & 0.25 & 0 & -0.0417 \\
			0.0417 & 0 & -0.25 & 0.2083 \\
			0 & 0.0417 & -0.2083 & 0.1667 \\
		\end{bmatrix*},
		\qquad
		\Q_{g}^{(2)}=
		\begin{bmatrix*}[l]
			0 & 0 & 0 & 0 & 0 & 0 & 0 \\
			0 & -0.1667 & 0 & -0.0417 & 0.2083 & 0 & 0 \\
			0 & 0 & 0 & 0 & 0 & 0 & 0 \\
			0 & 0.0417 & 0 & -0.25 & 0 & 0.2083 & 0 \\
			0 & -0.2083 & 0 & 0 & 0.25 & -0.0417 & 0 \\
			0 & 0 & 0 & -0.2083 & 0.0417 & 0.1667 & 0\\
			0 & 0 & 0 & 0 & 0 & 0 & 0 \\
		\end{bmatrix*},
	\end{aligned}
\end{equation}
where $\Q_{g}^{(2)} \coloneqq \sum_{i=1}^{4} \sum_{j=1}^{4}(\Q_{\xi_{1}}^{(2)})_{ij} \Z^{(2)}(i,j)$. We note that $[\Q_{\xi_{1}}^{(2)}]_{1,1}$ is placed at $ [\Q_{g}^{(2)}]_{2,2}$, $[\Q_{\xi_{1}}^{(2)}]_{1,2}$ is placed at $ [\Q_{g}^{(2)}]_{2,5}$ and so on, such that when $\Q_{g}^{(2)}$ is applied to a grid function defined at the global nodes, it produces the derivative of the vector at the nodes of $\widehat{\Omega}^{(2)}$, \eg, see the node local and global node numbering in \cref{fig:node ordering}. 

\cref{thm:assembly} provides enough information to construct the remaining SBP matrices, $\E_{\xi_i}= \Q_{\xi_{i}}+\Q_{\xi_{i}}^T$ and $\S_{\xi_i}=\Q_{\xi_{i}}-\frac{1}{2}\E_{\xi_i}$. However, it does not specify whether the $\E_{\xi_i}$ matrix can be decomposed as in \cref{eq: Exik}. Nonetheless, the decomposition is possible due to the presence of nodes at the facets of each split subdomain and the exactness of the $\R_{\gamma}^{(\ell)}$ matrix at those facets, which produces diagonal $\E_{\xi_{i}}$ matrices. That is, since $\R_{\gamma}^{(\ell)}$ is exact and the nodes at the shared interfaces are conforming, the normal vectors at the shared nodes are equal but of opposite signs, and thus $(\R_{\gamma}^{(\ell)})^T\N_{\xi_{i}\gamma}^{(\ell)}\B_\gamma^{(\ell)}\R_{\gamma}^{(\ell)}$ cancels exactly at the interior interfaces, leaving a diagonal $\E_{\xi_i}$ matrix with nonzero values only at the diagonal entries corresponding to the nodes located at exterior facets of the simplices. Therefore, the $\R_\gamma\in\IR{n_{f}\times n_{p}}$ matrix for each facet of the simplices can be constructed such that it contains an entry equal to unity at the row and column indices corresponding to the facet and global volume node numbering, respectively, and all its remaining entries are zero. Furthermore, the $\B_{\gamma}$ matrix at the facets of the simplices is constructed the same way as the $\H$ matrix, except restricted to the facet partitioning, 
\begin{equation}
	\B_{\gamma} \coloneqq \sum_{\ell=1}^{d} \sum_{i=1}^{n_1^{d-1}} \sum_{j=1}^{n_1^{d-1}}(\B_{\gamma}^{(\ell)})_{ij} \Z_{\gamma}^{(\ell)}(i,j),
\end{equation} 
where in this case $\ell$ refers to index of the line and quadrilateral subdomains on the facets of the triangle and tetrahedron, respectively. The matrix $\Z_{\gamma}^{(\ell)}(i,j)$ is constructed in a similar manner as $\Z^{(\ell)}(i,j)$ but using the local and global facet node numbering. Finally,  $\N_{\xi_{i}\gamma} \in \IR{n_f\times n_f}$, which is used in the construction of the $\Exi$ matrix, as in \cref{eq: Exik}, is a diagonal matrix containing along its diagonal the value of the outward-pointing unit normal vector on facet $\gamma$ of the reference simplex element. 

\section{Types of TPSS operators}

The application of one-dimensional operators on simplicial elements presents multiple opportunities to enhance the efficiency of multidimensional SBP operators. For example, TPSS operators can be derived using various well-studied one-dimensional SBP operators, such as the optimized operators of Mattsson \etal, \cite{mattson2014optimal} and Diener \etal \cite{diener2007optimized}, or the recent operators of Glaubitz \etal \cite{glaubitz2023summation} developed using nonpolynomial basis functions. The TPSS operators constructed in this work belong to the SBP diagonal-$\E$ family \cite{chen2017entropy}, which is characterized by having collocated facet and volume nodes\footnote{In this work, volume nodes, solution nodes, and quadrature nodes are used interchangeably, as they are collocated.}. Construction of TPSS operators that do not have volume nodes on the element facets, such as the Legendre-Gauss (LG) operators which belong to the SBP-$\Omega$ family of operators \cite{fernandez2018simultaneous}, is also possible, but application of SATs with such operators is not straightforward. This is because simultaneously satisfying the accuracy condition, $\S_{\xi_i}\V = \H\V_{\xi_i^\prime} -\frac{1}{2}\E_{\xi_i}\V$, and the decomposition of $\E_{\xi_i}$, given in \cref{eq: Exik}, on the simplicial element is not straightforward. One can possibly use similar strategies for the decomposition of the $\E_{\xi_i}$ matrix as in \cite{hicken2016multidimensional}, but this is not pursued in this work. We will primarily focus on TPSS operators that are derived using the one-dimensional LGL operators, but we have also derived TPSS operators based on classical SBP (CSBP) operators \cite{kreiss1974finite,strand1994summation,fernandez2014review,svard2014review}. Examples of the LGL-TPSS and CSBP-TPSS operators on triangles and tetrahedra are shown in \cref{fig:TPSS oper}.

In addition to enabling the use of various types of one-dimensional SBP operator on simplices, TPSS operators can also provide an alternative approach for $h$-refinement, see \cref{fig:csbp p1,fig:csbp p1 refined}. CSBP type finite-difference methods allow mesh refinements by adding interior nodes without increasing the degree of the operator, and TPSS operators inherit this property, enabling a similar type of mesh refinement on simplices. Although they produce sparse matrices, the CSBP-TPSS operators have a very large number of nodes, especially in three dimensions. While these operators can be better suited for use on meshes with large simplicial elements resembling multi-block meshes typically used in traditional finite-difference methods, further research is required to quantify their efficiency and other potential benefits. 

\begin{figure}[t]
	\centering
	\begin{subfigure}{0.33\textwidth}
		\centering
		\includegraphics[scale=0.28]{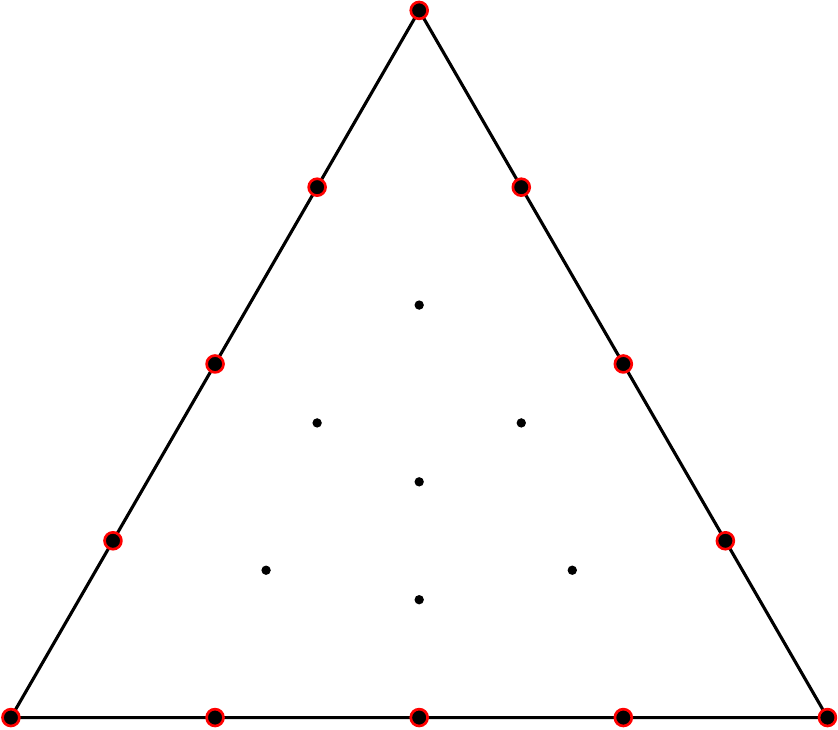}
		\caption{$p=1$, $n_p=19$, LGL-TPSS}
	\end{subfigure}\hfill
	\begin{subfigure}{0.33\textwidth}
		\centering
		\includegraphics[scale=0.28]{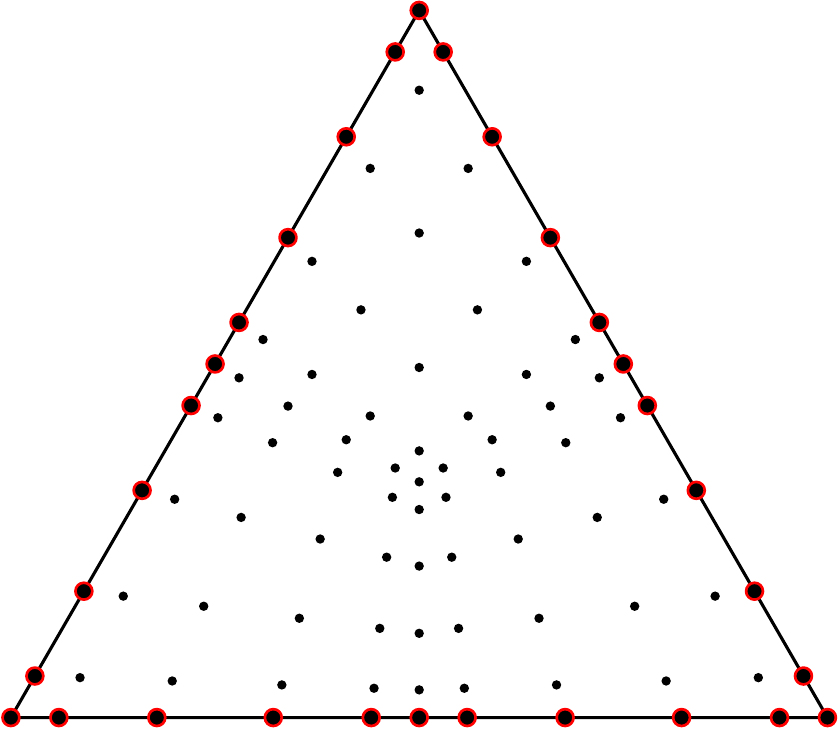}
		\caption{$p=4$, $n_p=91$ LGL-TPSS}
	\end{subfigure}\hfill
	\begin{subfigure}{0.33\textwidth}
		\centering
		\includegraphics[scale=0.091]{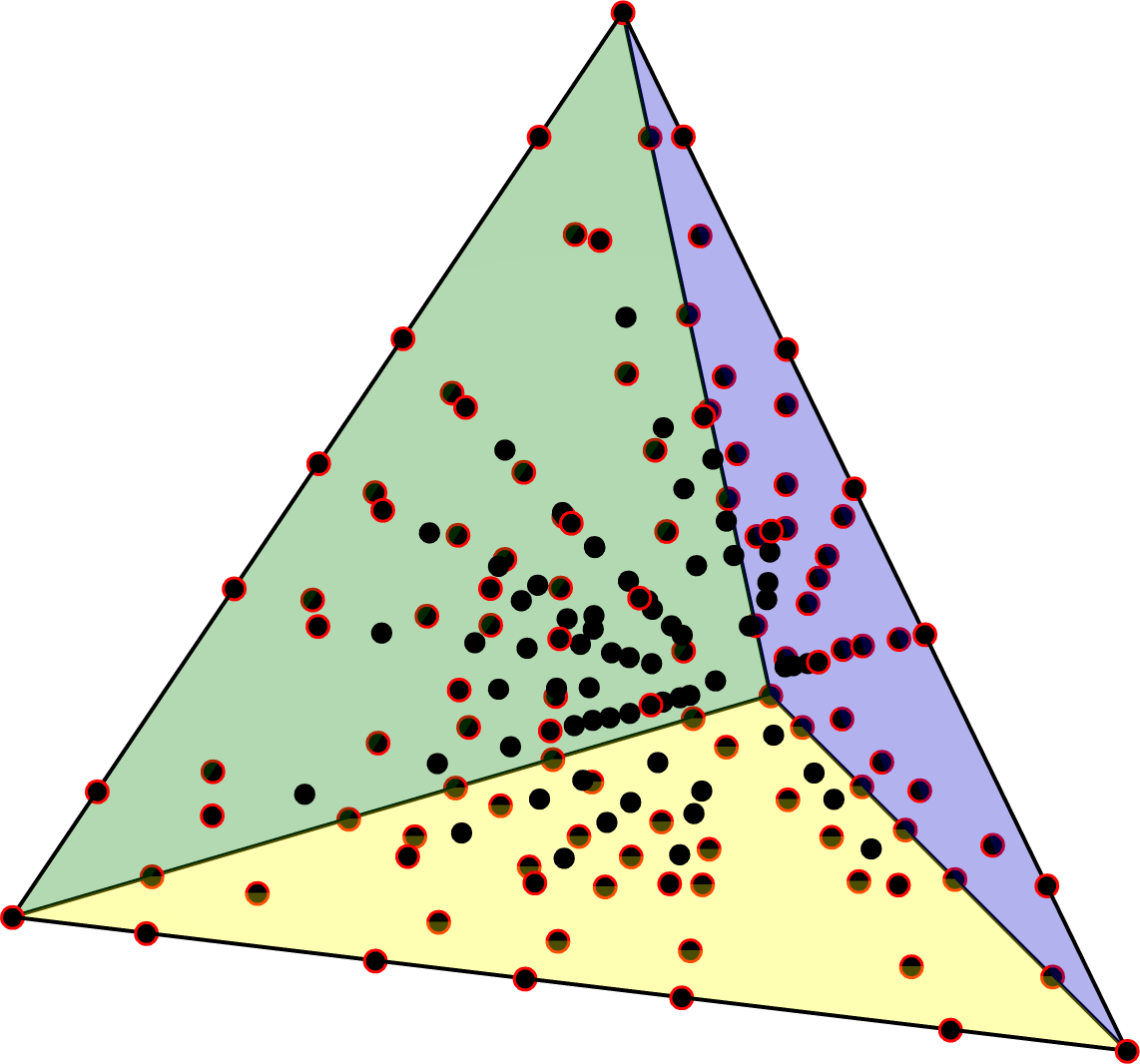}
		\caption{$p=1$, $n_p=175$ LGL-TPSS}
	\end{subfigure}
	\\
	\begin{subfigure}{0.33\textwidth}
		\centering
		\includegraphics[scale=0.28]{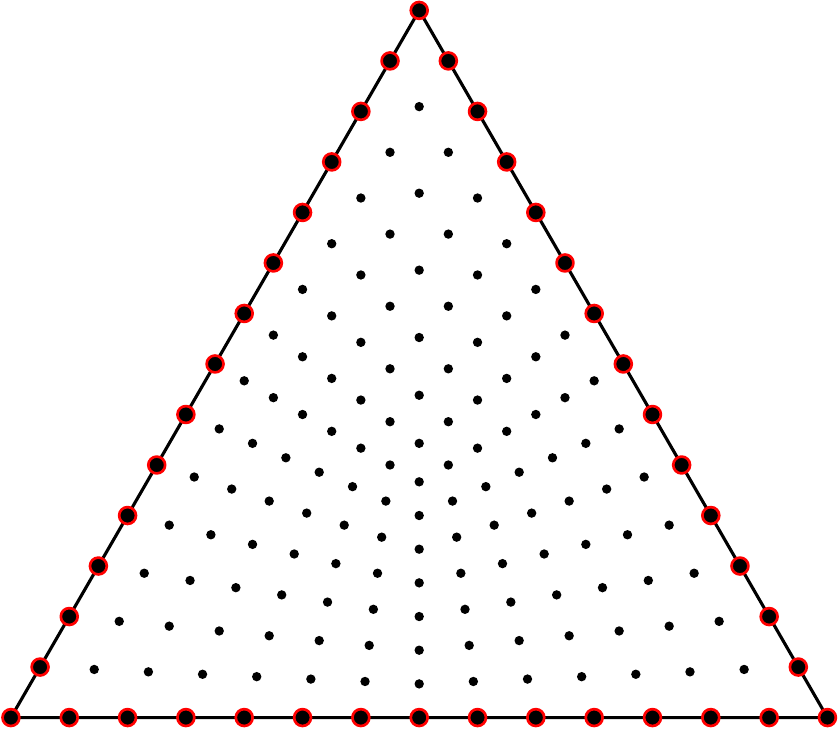}
		\caption{\label{fig:csbp p1}$p=1$, $n_p=169$ CSBP-TPSS}
	\end{subfigure}\hfill
	\begin{subfigure}{0.33\textwidth}
		\centering
		\includegraphics[scale=0.28]{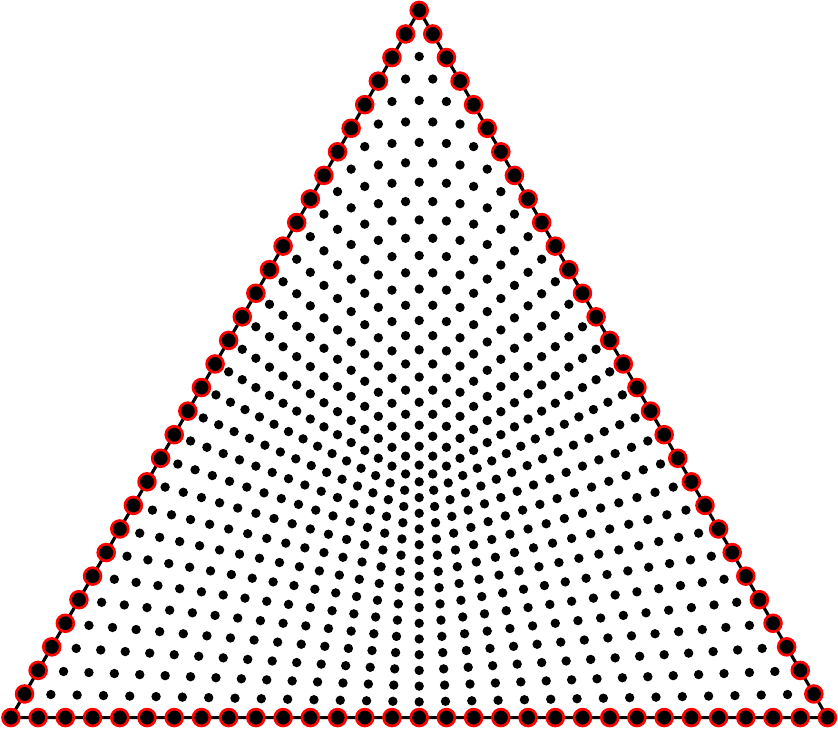}
		\caption{\label{fig:csbp p1 refined}$p=1$, $n_p=721$ CSBP-TPSS (refined)}
	\end{subfigure}\hfill
	\begin{subfigure}{0.33\textwidth}
		\centering
		\includegraphics[scale=0.04]{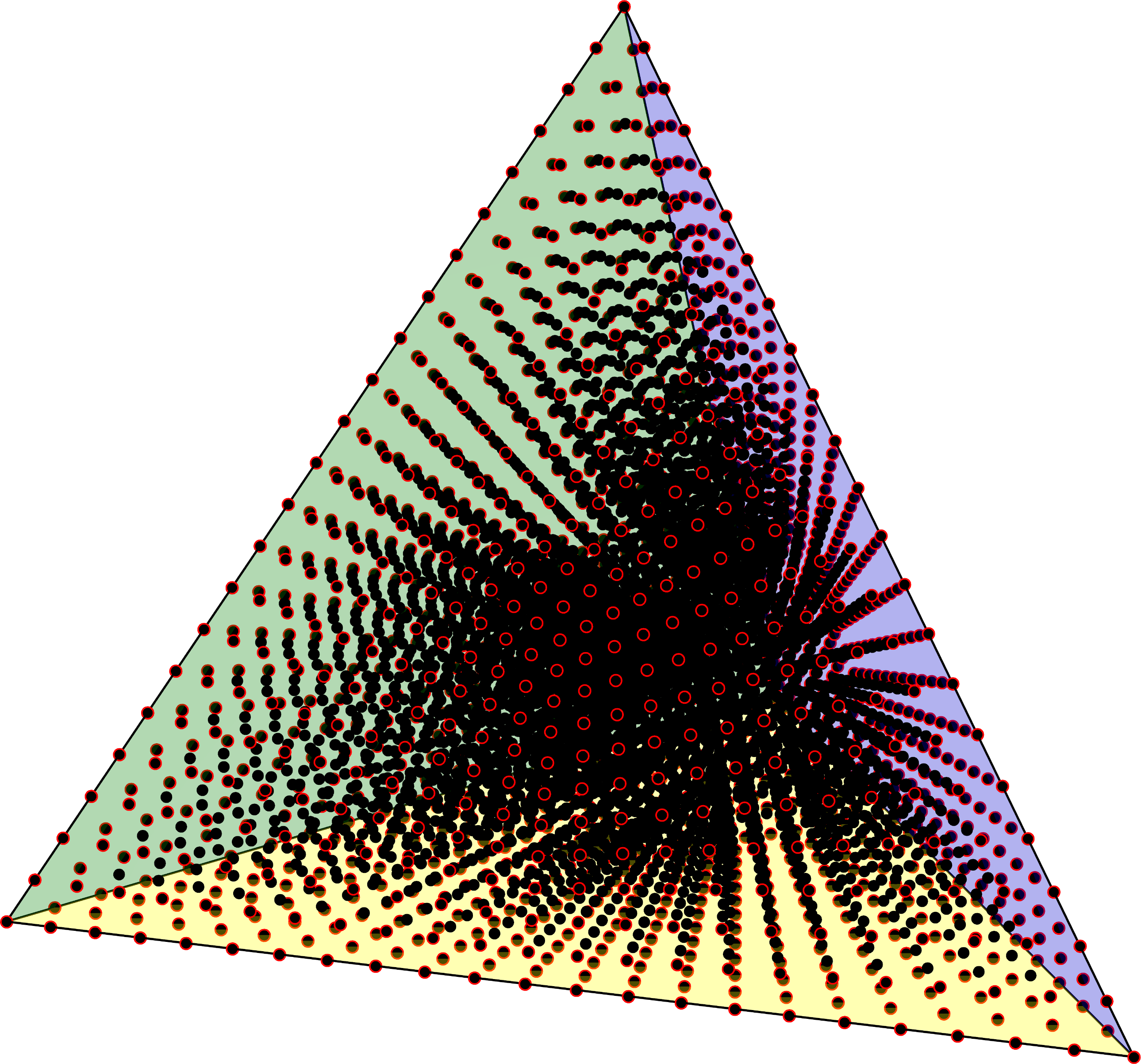}
		\caption{$p=1$, $n_p=6095$ CSBP-TPSS}
	\end{subfigure}
	\caption{\label{fig:TPSS oper} Examples of LGL and CSBP type TPSS operators. The degree $p$ TPSS operators are constructed using degree $p+d-1$ one-dimensional operators. The symbols $\sbullet$ and $\red{\bm{\circ}}$ denote the collocated volume and facet nodes, respectively.}
\end{figure}
\section{Accuracy and sparsity of TPSS operators}\label{sec:acc and sparsity of TPSS}
Although splitting of simplicial meshes guarantees generation of quadrilateral/hexahedral meshes \cite{pietroni2022hex}, the approach is generally avoided for finite-element analysis due to the resulting poor quality elements \cite{owen1998survey}. While the method presented in this work is distinct from directly splitting the physical simplex elements and applying tensor-product operators, further investigation is necessary to study its properties for practical meshes as there is significant similarity between the approaches. That being said, however, we have found the TPSS operators to be substantially more efficient compared to many of the existing multidimensional SBP operators, as will be shown through a number of numerical studies. Their efficiency emanates from their accuracy and sparsity as discussed below. 

\subsection{Accuracy}
The accuracy of TPSS operators depends on the mapping between the reference quadrilateral/hexahedral element and the quadrilateral/hexahedral subdomains in the split simplex reference element. The following theorem establishes the accuracy of the TPSS-SBP operators on the reference simplices. 
\begin{theorem} \label{thm:acc TPSS}
	A TPSS derivative operator, $\D_{\xi_{i}}$, on the reference triangle or tetrahedron, $\widehat{\Omega}$, constructed by mapping a degree $p\ge d-1$ tensor-product SBP diagonal-$\E$ type operator, $\D_{\eta_{i}}$, from the quadrilateral or hexahedral reference element, $\widetilde{\Omega}$, to the split quadrilateral or hexahedral subdomains
	of $\widehat{\Omega}$ and assembled using the procedure described
	in \cref{thm:assembly} is exact for polynomials up
	to degree $p-d+1$, i.e., 
	\begin{equation}\label{eq:acc TPSS}
		\left[\D_{\xi_{i}}\bm{p}\right]_{m}=\pder[\fn P]{\xi_{i}}(\bm{\xi}^{(m)}),\qquad\forall\bm{\xi}^{(m)}\in S_{\widehat{\Omega}},\qquad\forall\fn P\in\mathbb{P}^{p-d+1}(\widehat{\Omega}).
	\end{equation}
\end{theorem}
\begin{proof}
	The proof follows a similar strategy as the proof of Theorem 9 in
	\cite{crean2018entropy}. We will show the accuracy of the derivative
	operator on one of the split subdomains, $\widehat{\Omega}^{(\ell)}$,
	of the simplex, as the accuracy of the assembled derivative operator
	follows from \cref{thm:assembly}. Let $\Lambda_{\eta_{j},\xi_{i}}$
	and $\J$ be diagonal matrices containing the evaluations of $\left|\fn J\right|\pder[\eta_{j}]{\xi_{i}}$
	and $\left|\fn J\right|$ at the nodes of $S^{(\ell)}$, respectively,
	$\bm{p}^{(\ell)}$be the restriction of $\fn P$ onto the nodes $S^{(\ell)}$,
	and $\tilde{\H}$ be the norm matrix of $\widetilde{\Omega}$. Using the relations in \cref{eq:Sxik,eq:Qxik,eq:Dxik}, we can write 
	\begin{align}
		\D_{\xi_{i}}^{(\ell)}\bm{p}^{(\ell)} & =\frac{1}{2}\J^{-1}\tilde{\H}^{-1}\left[\sum_{j=1}^{d}\Lambda_{\eta_{j},\xi_{i}}\Q_{\eta_{j}}-\sum_{j=1}^{d}\Q_{\eta_{j}}^{T}\Lambda_{\eta_{j},\xi_{i}}+\E_{\xi_{i}}^{(\ell)}\right]\bm{p}^{(\ell)}\nonumber \\
		& =\frac{1}{2}\J^{-1}\sum_{j=1}^{d}\Lambda_{\eta_{j},\xi_{i}}\D_{\eta_{j}}\bm{p}^{(\ell)}+\frac{1}{2}\J^{-1}\sum_{j=1}^{d}\D_{\eta_{j}}\Lambda_{\eta_{j},\xi_{i}}\bm{p}^{(\ell)}+\frac{1}{2}\J^{-1}\tilde{\H}^{-1}\bigg(\E_{\xi_{i}}^{(\ell)}-\sum_{j=1}^{d}\E_{\eta_{j}}\Lambda_{\eta_{j},\xi_{i}}\bigg)\bm{p}^{(\ell)},\label{eq:acc1}
	\end{align}
	where we have applied the SBP property to the second term on the RHS of the first equality to arrive at the second line. The last term in (\ref{eq:acc1}) is identically zero due to the exactness
	of the extrapolation matrices and because $\E_{\xi_{i}}^{(\ell)}$ and $\E_{\eta_{j}}$ are diagonal, see \cref{eq: Exik}. Additionally, we note that due to the specific mappings in \cref{eq:bilinear Psi} and \cref{eq:trilinear Psi}, a polynomial in the $\bm{\xi}$ coordinates remains a polynomial of the same degree in the $\bm{\eta}$ coordinates. Hence, since $\D_{\eta_{j}}^{(\ell)}$ is exact for polynomials of degree $p$, the first term in \cref{eq:acc1} gives
	\begin{equation}
		\left[\frac{1}{2}\J^{-1}\sum_{j=1}^{d}\Lambda_{\eta_{j},\xi_{i}}\D_{\eta_{j}}\bm{p}^{(\ell)}\right]_{m}=\frac{1}{2}\pder[\fn P]{\xi_{i}}(\bm{\xi}^{(m)}), \qquad \forall \bm{\xi}^{(m)} \in S^{(\ell)}.\label{eq:acc2}
	\end{equation}
	We also note that the metric terms, $\Lambda_{\eta_{j},\xi_{i}}$, are
	polynomials of degree $d-1$ in the $\eta_{j}$ direction. To see this, consider examples of the metric terms $\Lambda_{\eta_{1},\xi_{1}} = \left|\fn J\right|\pder[\eta_{1}]{\xi_{1}}=\pder[\xi_{2}]{\eta_{2}}$ in
	two dimensions and $\Lambda_{\eta_{1},\xi_{1}} = \left|\fn J\right|\pder[\eta_{1}]{\xi_{1}}=\pder[\xi_{2}]{\eta_{2}}\pder[\xi_{3}]{\eta_{3}}-\pder[\xi_{2}]{\eta_{3}}\pder[\xi_{3}]{\eta_{2}}$
	in three dimensions, which, after using the definitions of the mapping functions in \cref{eq:bilinear Psi} and \cref{eq:trilinear Psi}, produce linear and quadratic terms in the $\eta_{1}$ direction, respectively.
	Hence, the product $\Lambda_{\eta_{j},\xi_{i}}\bm{p}^{(\ell)}$ in
	the second term in (\ref{eq:acc1}) produces a polynomial of degree
	$p$, whose derivative can be  computed exactly using $\D_{\eta_{j}}$.
	Therefore, we can apply the product rule to the second term in (\ref{eq:acc1})
	to find 
	\begin{equation}
		\frac{1}{2}\J^{-1}\sum_{j=1}^{d}\D_{\eta_{j}}\Lambda_{\eta_{j},\xi_{i}}\bm{p}^{(\ell)}=\frac{1}{2}\J^{-1}\sum_{j=1}^{d}\Lambda_{\eta_{j},\xi_{i}}\D_{\eta_{j}}\bm{p}^{(\ell)}+\frac{1}{2}\J^{-1}{\rm diag}(\bm{p}^{(\ell)})\sum_{j=1}^{d}\D_{\eta_{j}}\Lambda_{\eta_{j},\xi_{i}}\bm{1}.\label{eq:acc3}
	\end{equation}
	The first term on the RHS of \cref{eq:acc3} is identical to the LHS of (\ref{eq:acc2}) and the last term is identically zero, as $\D_{\eta_{j}}$ is exact for polynomials of degree $p\ge d-1$ by assumption, and because
	\begin{equation}
		\left[\sum_{j=1}^{d}\D_{\eta_{j}}\Lambda_{\eta_{j},\xi_{i}}\bm{1}\right]_{m}=\sum_{j=1}^{d}\pder{\eta_{j}}\left(\left|\fn J\right|\pder[\eta_{j}]{\xi_{i}}\right)\Bigg{|}_{\bm{\eta}^{(m)}}=0,\qquad m\in\{1,\dots,n_1^{d}\},
	\end{equation}
	is the metric identity.\ignore{, which we have verified to hold for both the bilinear and trilinear mappings through direct computations.} Therefore, adding the results of (\ref{eq:acc2}) and (\ref{eq:acc3}) gives 
	\begin{equation}
		\left[\D_{\xi_{i}}^{(\ell)}\bm{p}^{(\ell)}\right]_{m}=\pder[\fn P]{\xi_{i}}(\bm{\xi}^{(m)}),\qquad \forall \bm{\xi}^{(m)} \in S^{(\ell)}.
	\end{equation}
	Finally, we invoke \cref{thm:assembly} to conclude that \cref{eq:acc TPSS} holds. 
\end{proof}

\cref{thm:acc TPSS} states that TPSS operators lose one and two degrees of accuracy in two and three dimensions, respectively, compared to the tensor-product operators used to construct them. The condition in \cref{eq:acc TPSS} is in fact the same as the accuracy condition (first property) in \cref{def:sbp}, except for the lower polynomial degree accuracy, $p-d+1$, for $d\ge2$, \ie, a TPSS operator constructed using a degree $p$ tensor-product operator is degree $p-d+1$ accurate. This motivates the following definition of TPSS-SBP operators. 
\begin{definition}\label{def:TPSS}
	A TPSS derivative operator is a degree $p$ SBP operator if it is constructed as discussed in \cref{sec:TPSS construction} and using a tensor-product SBP operator of degree $p+d-1$.
\end{definition} 

\cref{def:TPSS} ensures that a degree $p$ TPSS-SBP operator is exact for polynomials up to degree $p$. Hence, we expect a solution convergence rate of $p+1$ when a degee $p$ TPSS-SBP operator is used to discretize PDEs with sufficiently smooth solutions. This expectation matches the theoretical error estimates presented for finite-element methods constructed by splitting simplicial meshes into quadrilateral and hexahedral elements and mapping the LGL tensor-product operators onto the split subdomains \cite{durufle2009influence}. Based on numerical results, Durufl{\'e} \etal \cite{durufle2009influence} conjectured that the solution convergence rate of their simplex-splitting-based LGL finite-element method is one degree higher than the theoretical estimates. However, all of our numerical results closely match the theoretical solution error estimates. Finally, the study in \cite{durufle2009influence} identified that the main cause of the loss of accuracy on the split finite elements is the suboptimal quadrature accuracy of the LGL operators, which is a consequence of the non-affine mappings. This suggests that SBP-TPSS operators constructed with a higher degree, and thus increased quadrature accuracy, of LGL tensor-product operators will have significantly smaller error coefficients compared to other types of SBP operators of the same degree.

\subsection{Sparsity}
Sparsity is crucial for efficiency of SBP operators. This is especially true for entropy-stable SBP-SAT discretizations based on two-point fluxes, as the number of nonzero entries in the derivative operator is proportional to the number of two-point flux function evaluations required to approximate the derivative of the inviscid fluxes in the Euler and Navier--Stokes equations. By studying the tensor-product nodal distribution of the TPSS operators, one can find that the total number of nodes in a TPSS operator constructed using a tensor-product operator with $n_1$ nodes in each direction is 
\begin{equation}\label{eq:TPSS num node}
	n_{p}=\left(d+1\right)n_{1}^{d}-\left(d+d^{d-2}\right)n_{1}^{d-1}+\left(d-2\right)\left(\left(d+1\right)n_{1}-2\right)+1.
\end{equation}
Furthermore, an estimate (not necessarily an upper bound) of the number of the nonzero entries in the $\D_{\xi_{i}}$ matrix is
\begin{equation}\label{eq:nnz TPSS}
	\text{nnz}=\left(n_{1}d - d + 1 \right)n_{p}.
\end{equation}
The sparsity of an $n_p\times n_p$ matrix operator is defined as 
\begin{equation}\label{eq:sparsity}
	s=1-\frac{\text{nnz}}{n_{p}^{2}},
\end{equation}
which after simplifications gives the sparsity of the two and three dimensional TPSS derivative operators as 
\begin{equation}\label{eq:2d 3d TPSS sparsity}
	s_{\rm 2D}=1-\frac{2n_{1}-1}{3n_{1}^{2}-3n_{1}+1}, \qquad s_{\rm 3D}=1-\frac{3n_{1}-2}{4n_{1}^{3}-6n_{1}^{2}+4n_{1}-1}.
\end{equation}
The sparsity estimates in \cref{eq:2d 3d TPSS sparsity} indicate that TPSS operators are very sparse, \eg, in three dimensions, the $p=1$ and $p=10$ LGL-TPSS operators are more than 94\% and 99.5\% sparse, respectively. 

As sparsity is directly related to computational cost, it can be used as a metric to compare the relative efficiency of different types of SBP operators. Of course efficiency depends on many other factors, such as accuracy and CFL restriction, so sparsity is just a single metric that can provide some valuable insights into efficiency but must be complemented by several other properties to assess overall performance. In this work, we will use the label ``dense'' to refer to SBP operators with non-tensor-product derivative matrix. This should not be confused with ``dense-norm'' SBP operators, which refer to SBP operators with dense $\H$ matrix.  \cref{tab:nnz TPSS} presents a comparison of the estimated number of nonzero entries of the LGL-TPSS operators and the derivative matrices constructed using the quadrature rules derived for dense SBP diagonal-$\E$ \cite{worku2023quadrature} (with LGL facet nodes on triangles) and dense SBP-$\Omega$ operators that would be constructed using estimates of the lower bound of the number of nodes of quadrature rules with positive weights and interior nodes (PI) \cite{lyness1975moderate,wang2023explicit}. Note that the lower bound of the number of nodes of PI quadrature rules may not be achievable; hence, SBP-$\Omega$ operators will usually have more nodes, implying that the sparsity ratios in \cref{tab:nnz TPSS} are smaller in practice. The dense SBP diagonal-$\E$ and SBP-$\Omega$ operators produce derivative operators with $n_p^2$ nonzero entries. 
\begin{table*}[!t]
	\footnotesize
	\centering
	\captionsetup{skip=3pt}
	\caption{\label{tab:nnz TPSS} Number of nonzero entries in the derivative matrices of the dense SBP diagonal-$\E$ and LGL-TPSS operators normalized by the number of nonzero entries of dense SBP-$\Omega$ operators that would be constructed using the lower bound estimates of the number of nodes of quadrature rules with positive weights and interior nodes.}
	\begin{threeparttable}
		\setlength{\tabcolsep}{0.225em}
		\renewcommand*{\arraystretch}{1.2}
		\begin{tabular}{cl @{\hspace{1.0em}}lllll lllll lllll lllll}
			\toprule
			$d$ &$p=\frac{q}{2}$ & 1 & 2 & 3 & 4 & 5 & 6 & 7 & 8 & 9 & 10 & 11 & 12 & 13 & 14 & 15 & 16 & 17 & 18 & 19 & 20 \\
			\midrule 
			\multirow[t]{2}{*}{2} & SBP-$\E$ & 5.44 & 4.00 & 2.25 & 2.85 & 2.25 & 2.12 & 1.84 & 1.92 & 1.99 & 1.74 &  &  &  & &  &  &  &  &  & \\
			& TPSS & 10.56 & 7.19 & 3.81 & 3.91 & 2.87 & 2.33 & 2.09 & 1.90 & 1.60 & 1.50 & 1.36 & 1.20 & 1.13 & 1.03 & 0.96 & 0.90 & 0.85 & 0.79 & 0.76 & 0.72\\
			\addlinespace
			\multirow[t]{2}{*}{3} & SBP-$\E$ & 3.06 & 4.37 & 4.52 & 4.28 & 4.55 &  & &  &  &  &  &  &  &  &  & &  &  &  & \\
			& TPSS & 109.38 & 39.64 & 18.64 & 11.35 & 8.06 & 4.50 & 3.62 & 2.65 & 2.05 & 1.68 & 1.36 & 1.10 & 0.94 & 0.79 & 0.68& 0.59& 0.51 & 0.45 & 0.40 & 0.36\\
			\bottomrule
		\end{tabular}
	\end{threeparttable}
\end{table*}
The table shows that despite producing very large matrices, at high degrees, the SBP-TPSS operators are sparser than both the dense SBP-$\Omega$ and dense SBP diagonal-$\E$ operators. More specifically, the LGL-TPSS operators are sparser than dense SBP-$\Omega$ operators for degrees greater than $14$ and $12$ on triangles and tetrahedra, respectively. They are also sparser than dense SBP diagonal-$\E$ operators for degrees greater than $7$ on triangles, and should be sparser on tetrahedra as well if higher degree dense SBP diagonal-$\E$ operators are constructed, since these operators have more nodes than dense SBP-$\Omega$ operators. It is interesting that, despite having more than $15$ times more degrees of freedom,  the number of nonzero entries in the degree $p=20$ TPSS operator on the tetrahedron is only about $36\%$ that of the same degree dense SBP-$\Omega$ operator (assuming lower bound number of nodes). This suggests that the TPSS operators can be substantially more efficient than the existing multidimensional SBP operators at high polynomial degrees, as the lower number of nonzero entries in the TPSS matrices results in fewer floating point operations. As will be shown later, the TPSS operators are also more efficient than dense SBP diagonal-$\E$ operators at low polynomial degrees due to their high accuracy per degree of freedom.

\section{Numerical results}\label{sec:numerical results}
In this section, we investigate the properties of the TPSS-SBP operators through numerical experiments. To limit the myriad combination of potential comparisons, we focus exclusively on the relative performance of the TPSS operators and the dense SBP diagonal-$\E$ operators derived in \cite{worku2023quadrature} with quadrature degrees of $q=2p$ and with the LGL facet nodes on triangles. This is motivated by the fact that both operators belong to the SBP diagonal-$\E$ family, which has a relatively low SAT coupling cost for entropy-stable discretizations due to the collocated facet and volume nodes, yet incurs a relatively high volume-flux computation cost due to its large number of nodes. In contrast, dense SBP-$\Omega$ operators have fewer nodes, resulting in lower volume flux computation costs, but do not have collocated volume and facet nodes, leading to higher SAT coupling costs. Although detailed comparisons of existing multidimensional operators are lacking in the literature, it is reasonable to assume that the cost difference between the dense SBP diagonal-$\E$ and SBP-$\Omega$ operators is not very large. Furthermore, the two families of operators generally exhibit comparable accuracy. Therefore, comparing the TPSS operators with dense SBP diagonal-$\E$ operators can represent their performance relative to many of the existing multidimensional SBP operators. In tables and figures, dense SBP diagonal-$\E$ operators are denoted by SBP-$\E$. 

The TPSS operators are implemented using the sparse matrix functionality provided by the Armadillo library \cite{sanderson2016armadillo} in C++; however, there are possibly more efficient algorithms that take into consideration the sparsity patterns of the operators, such as sum factorization, to avoid the overhead cost of using sparse matrices. This topic is left for future studies. Both operators are executed within the same code base, with only the type of operator being altered, making runtime comparisons reasonable. 

The numerical comparisons are conducted using the advection, Euler, and Navier--Stokes equations. Periodic domains are considered in all cases. The upwind SAT is used for the advection equation.  The problems governed by the Euler and Navier--Stokes equations use the Ismail-Roe two-point fluxes without interface dissipation. The viscous terms in the Navier--Stokes equations are coupled using the Baumann--Oden SAT \cite{baumann1999discontinuous,worku2021simultaneous}.  Finally, the isentropic vortex problem uses the relaxation RK4 time-marching method, while all the other test cases are implemented using the standard RK4 method.

\subsection{Advection problem}
We consider the linear advection equation, 
\begin{equation}\label{eq:adv eq}
	\pder[\fn{U}]{t} + \sum_{i=1}^{d}{c}_{i}\pder[\fn{U}]{ {{x}}_{i}} = 0,
\end{equation} 
on the periodic domain $ \Omega=[0,1]^{d} $. The initial condition, $\fn{U}(\bm{x},t=0)$, and the exact solution, $\fn{U}(\bm{x},t)$, for the problem are given by
\begin{equation}\label{eq:exact adv sol}
	\fn{U}({\bm{x}},t) =\prod_{i=1}^{d} \sin(\omega\pi ({{x}}_{i}-{c} _{i}t)),
\end{equation}
where $ \bm{c}=[5/4,\sqrt{7}/4]^T $ in 2D or $ \bm{c}=[3/2,1/2,1/\sqrt{2}]^T $ in 3D is used in all cases, as in \cite{worku2023quadrature}. The values of $\bm{c}$ are chosen to set the wave speed magnitude at $\sqrt{d}$. The direction of the wave propagation depends on $ \bm{c} $ and affects numerical errors, \eg, see \cite{zingg1993finite}, and mesh convergence rates in some cases. The values of $\bm{c}$ are chosen to prevent potential alignment of the wave with the meshes used. The problem is run up to $ t=1 $ with the $ \omega $ parameters in \cref{eq:exact adv sol} set to $ 8 $ and $ 2 $ for the 2D and 3D cases, respectively.

We first investigate the maximum time-step restrictions for the TPSS operators, which when combined with the accuracy and sparsity discussion in the previous section helps derive more meaningful conclusions. \cref{tab:adv cfl} presents the maximum time-step restrictions obtained with the TPSS operators in comparison to those of the dense SBP diagonal-$\E$ operators reported in \cite[Table 5.2]{worku2023quadrature}. The test is conducted on uniform triangular and tetrahedral meshes which are generated by dividing each element of a $4\times4$ quadrilateral mesh and a $4\times4\times4$ hexahedral mesh into $2$ triangles and $6$ tetrahedra, respectively. The golden section optimization is employed to determine the maximum time step using the criterion that the change in energy remains less than or equal to zero after 5 periods, \ie, $\sum_{\Omega_{k}\in\fn{T}_h} (\bm{u}_k^T\H_k\bm{u}_k - \bm{u}_{0,k}^T\H_k\bm{u}_{0,k})\le 0$, where $\bm{u}_{0,k}$ is the vector of the initial solution on the nodes of $\Omega_{k}$. In two dimensions, the TPSS operators allow maximum time steps within the range of $70\%$ to $90\%$ of those of the dense SBP diagonal-$\E$ operators. In three dimensions, the degree $1$, $3$, and $5$ TPSS operators allow time-step values about $40\%$, $72\%$, and $63\%$ of the maximum time-step values of the dense SBP diagonal-$\E$ operators, respectively, while the even degree TPSS operators allow slightly larger time-step values. In general, it can be concluded that the TPSS operators require smaller stable time-step values for explicit time-marching methods, and this must be taken into account when examining their relative efficiency for practical problems. Nevertheless, the reduction in stable maximum time step is not as significant as the accuracy and runtime gains afforded by the TPSS operators, as demonstrated below.

\begin{table*}[t]
	\footnotesize
	\centering
	\captionsetup{skip=3pt}
	\caption{\label{tab:adv cfl} Maximum time-step restrictions and percent ratios ($r$) for the advection problem discretized with the TPSS operators and the dense SBP diagonal-$\E$ operators with quadrature degrees of $2p$ and LGL facet nodes on triangles.}
	\begin{threeparttable}
		\setlength{\tabcolsep}{0.35em}
		\renewcommand*{\arraystretch}{1.2}
		\begin{tabular}{cl @{\hspace{1.0em}}lllll lllll}
			\toprule
			$d$ &$p$ & 1& 2 & 3 & 4 & 5 & 6 & 7 & 8 & 9 & 10  \\
			\midrule 
			\multirow[t]{2}{*}{2} & SBP-$\E$ & 0.0446& 0.0235 & 0.0146 & 0.0106 & 0.0066 & 0.0059 & 0.0048 & 0.0038 &0.0033 &0.0025   \\
			& TPSS & 0.0312& 0.0173& 0.0112& 0.0079& 0.0060 & 0.0046 & 0.0037 & 0.0030 & 0.0024 & 0.0020 \\
			& $r(\%)$ & 70.0 & 73.6 & 76.7 & 74.5 & 90.9 & 78.0 & 77.1 & 79.0 & 72.7 & 80.0 \\
			\addlinespace
			\multirow[t]{2}{*}{3} & SBP-$\E$ & 0.0388 & 0.0065& 0.0097 & 0.0049 & 0.0059 &  & &  &  &   \\
			& TPSS & 0.0154& 0.0099 & 0.0070 & 0.0050 &  0.0037&  &  & & &  \\
			& $r(\%)$  & 39.7 & 152.3 & 72.2 & 102.0 & 62.7&  &  & & &   \\		
			\bottomrule
		\end{tabular}
	\end{threeparttable}
\end{table*}

As the same degree TPSS and dense SBP diagonal-$\E$ operators have substantially different numbers of nodes, the accuracy of their solutions is compared using the number of degrees of freedom rather than the number of elements in a mesh. Computational cost in a computer implementation mainly depends on the total number of degrees of freedom rather than the number of elements; hence, we are interested in the accuracy attained by the two operators for a given number of degrees of freedom. It should be noted that accuracy comparisons based on number of elements would highly favour the TPSS operators.  The convergence rates for the 2D and 3D cases are tabulated in \cref{sec:grid conv adv}, \cref{tab:grid conv adv}, which shows that both operators attain convergence rates between $p$ and $p+1$ in the $\H$-norm. All the accuracy results are obtained using sufficiently small CFL values to ensure that the temporal error is negligible. A comparison of the accuracy per degree of freedom of the TPSS and dense SBP diagonal-$\E$ operators for the advection equation is shown in \cref{fig:adv acc}. The figure shows that, for the same number of degrees of freedom, the TPSS operators are substantially more accurate than the SBP diagonal-$\E$ operators, with about 1 and 2 orders of magnitude improvements for the $p=5$ operators in the two- and three-dimensional cases, respectively. 
\begin{figure}[t]
	\centering
	\captionsetup{belowskip=0pt, aboveskip=3pt}
	\begin{subfigure}{0.5\textwidth}
		\centering
		\includegraphics[scale=0.5]{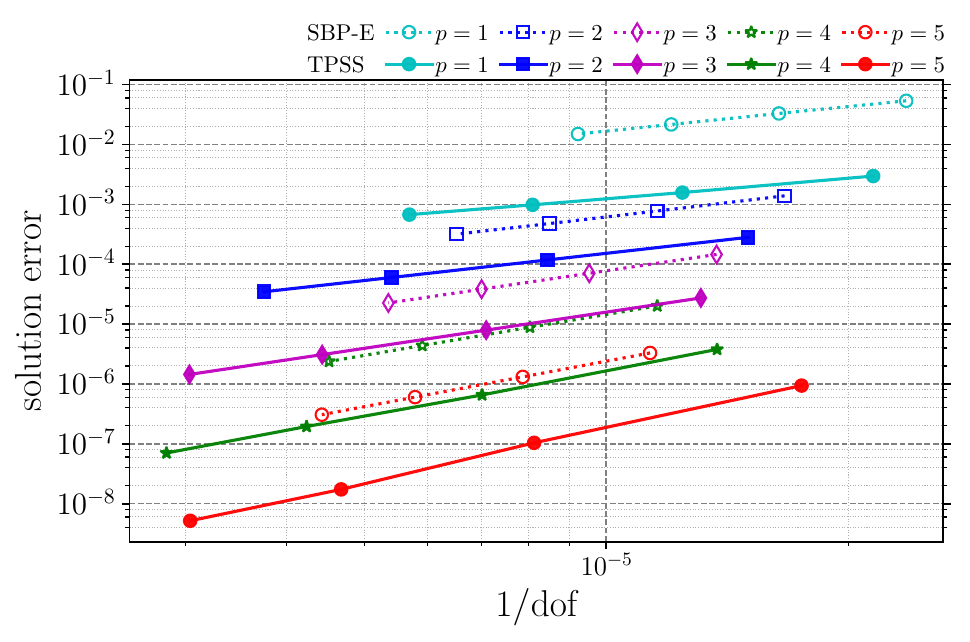}
	\end{subfigure}\hfill
	\begin{subfigure}{0.5\textwidth}
		\centering
		\includegraphics[scale=0.5]{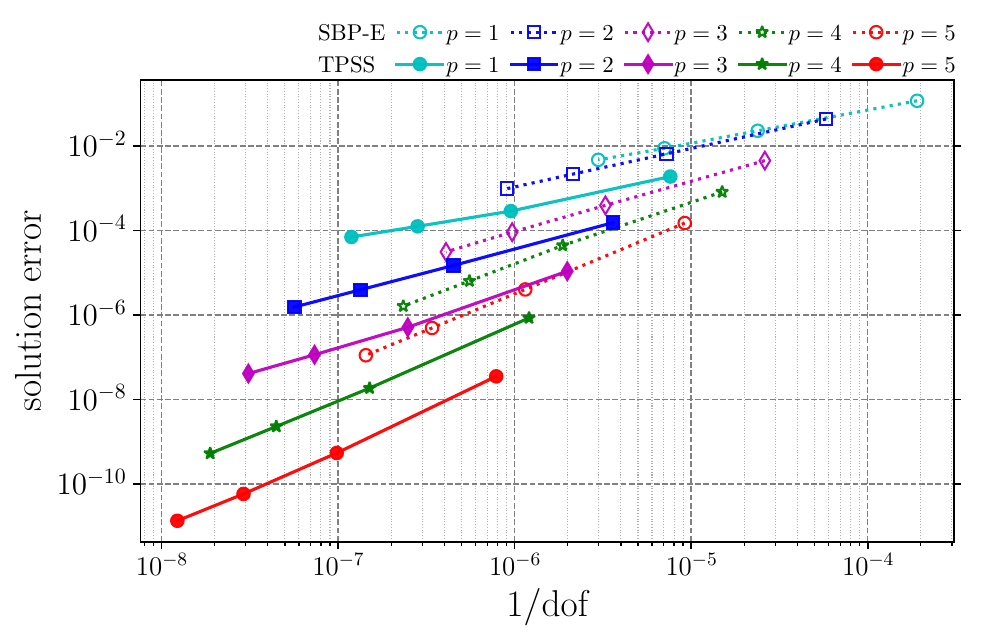}
	\end{subfigure} 
	\caption{\label{fig:adv acc} Grid convergence study of the $\H$-norm solution error for the 2D (left) and 3D (right) advection problems.}
\end{figure}

Considering the difference in the maximum stable time-step limits for the TPSS and dense SBP diagonal-$\E$ operators is not substantially large, another approach to assess the efficiency gains of the TPSS operators is to examine the solution error for a given computational time or, equivalently, the computational time required to achieve a specified error threshold. To make this comparison, we run the advection problem for one time step (with a step size of $10^{-4}$) using each type of operator and plot the  $\H$-norm solution error against the time required to compute the spatial residual vector. Again, the time step is set to a very small value such that the temporal error is negligible. The reported spatial residual computation time is for implementations without parallelization and includes both the time required to perform the spatial discretization operations and the time required to construct SBP operators on the physical elements. Considering only the former would slightly favor the TPSS operators. \cref{fig:adv res time} shows that the $p=5$ TPSS operators on triangles and tetrahedra produce about 5 times and 100 times smaller solution errors compared to the dense SBP diagonal-$\E$ operators for a given computational runtime. Furthermore, for a specified error threshold, the same degree TPSS operator is about 2 times and 15 times faster on triangles and tetrahedra, respectively. \ignore{We have also plotted the runtime comparisons using the $L^{\infty}$-norm in \cref{fig:adv res time linf}, which exhibit similar trends to those observed in the $\H$-norm.}

\begin{figure}[t]
	\centering
	\captionsetup{belowskip=0pt, aboveskip=4pt}
	\begin{subfigure}{0.5\textwidth}
		\centering
		\includegraphics[scale=0.5]{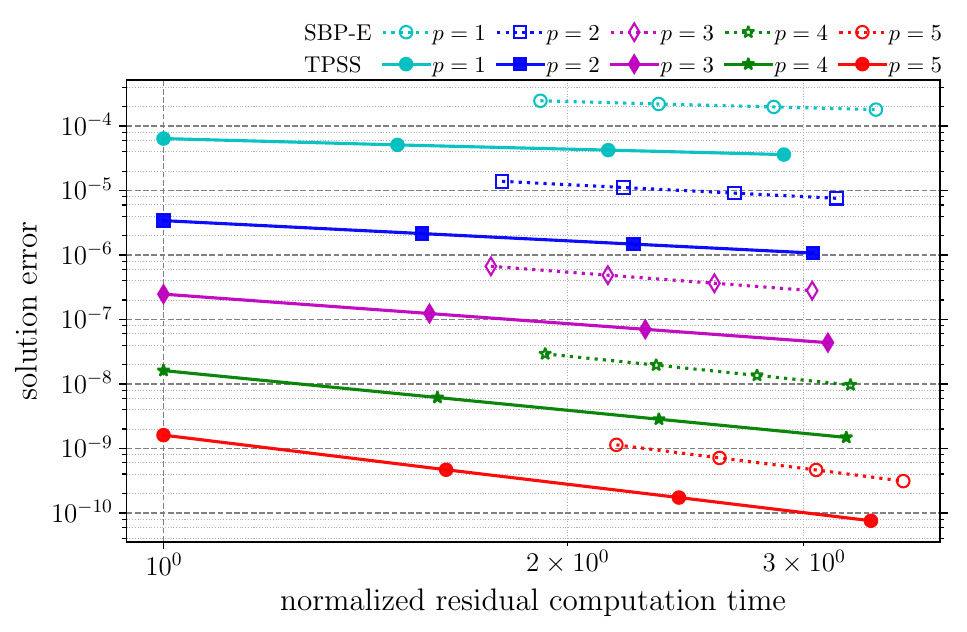}
	\end{subfigure}\hfill
	\begin{subfigure}{0.5\textwidth}
		\centering
		\includegraphics[scale=0.5]{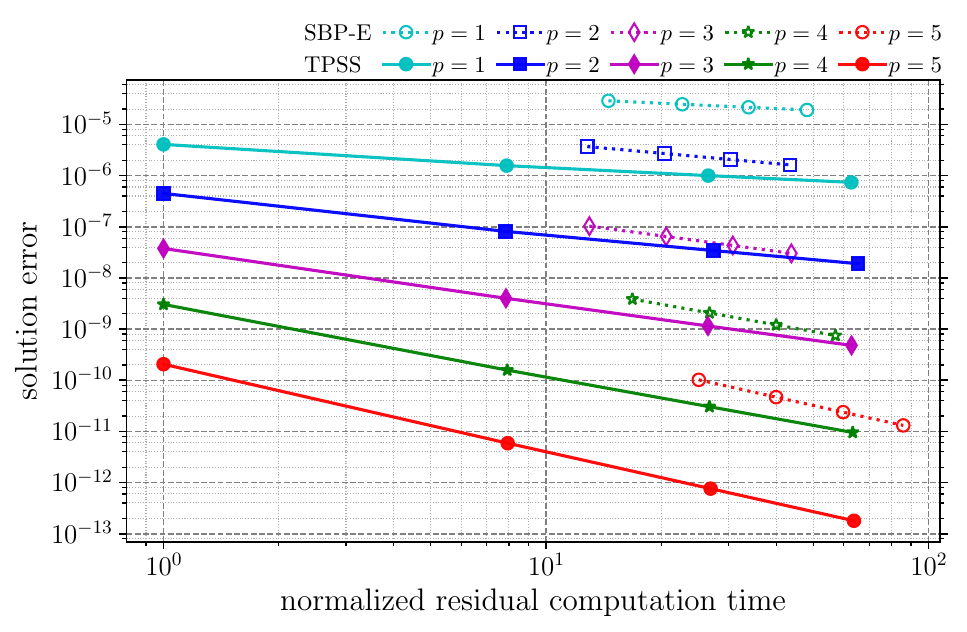}
	\end{subfigure} 
	\caption{\label{fig:adv res time} Normalized computational runtime versus the $\H$-norm solution error for the 2D (left) and 3D (right) advection problem.}
\end{figure}

\ignore{
The $\H$-norm is an approximation to the $L^2$-norm, and it is possible that some types of SBP operators have $\H$ matrices that approximate the $L^2$-norm better than others. This can affect the solution error values. Hence, for a sanity check, the $L^{\infty}$-norm of the solution errors with the TPSS and dense SBP diagonal-$\E$ operators is computed after one time step. The results in \cref{fig:adv res time linf} indicate that the solution errors in the $L^{\infty}$-norm exhibit similar trends to those in the $\H$-norm, suggesting that the integration error due to the use of the $\H$-norm has a negligible impact on the accuracy and efficiency differences between the two operators. 

\begin{figure}[t]
	\centering
	\captionsetup{belowskip=0pt, aboveskip=4pt}
	\begin{subfigure}{0.5\textwidth}
		\centering
		\includegraphics[scale=0.5]{2d_total_time_SBP_E_vs_tss_adv_err_vs_time_linf_tpss.pdf}
	\end{subfigure}\hfill
	\begin{subfigure}{0.5\textwidth}
		\centering
		\includegraphics[scale=0.5]{3d_total_time_SBP_E_vs_tss_adv_err_vs_time_linf_tpss.pdf}
	\end{subfigure} 
	\caption{\label{fig:adv res time linf} Computational runtime versus the $L^{\infty}$-norm solution error for the 2D (left) and 3D (right) advection problem.}
\end{figure}
}
\begin{figure}[!ht]
	\centering
	\begin{subfigure}{0.26\textwidth}
		\centering
		\includegraphics[scale=0.08]{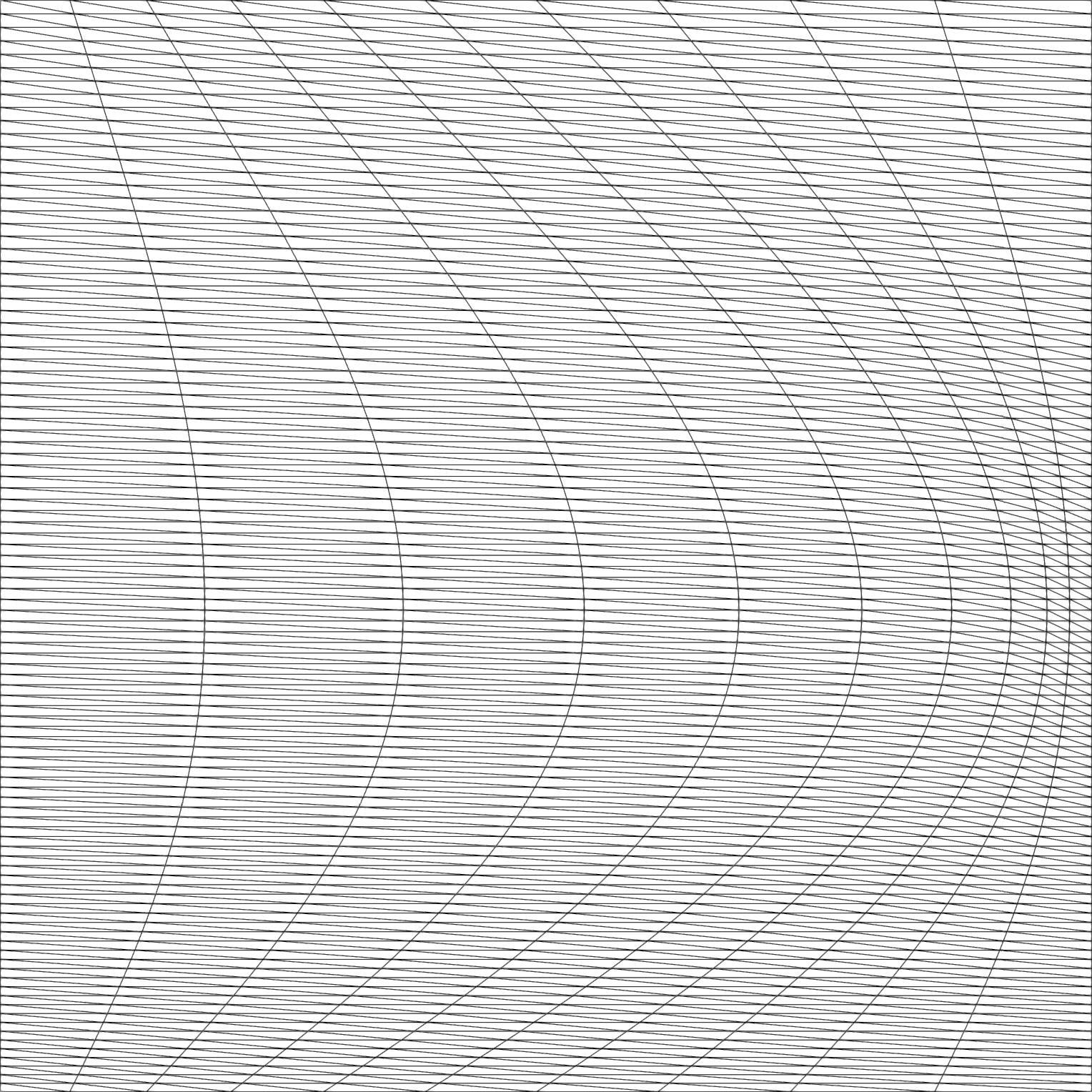}
		\caption{Slightly distorted mesh, $M_1^{2}$}
	\end{subfigure}\hfill
	\begin{subfigure}{0.35\textwidth}
		\centering
		\includegraphics[scale=0.08]{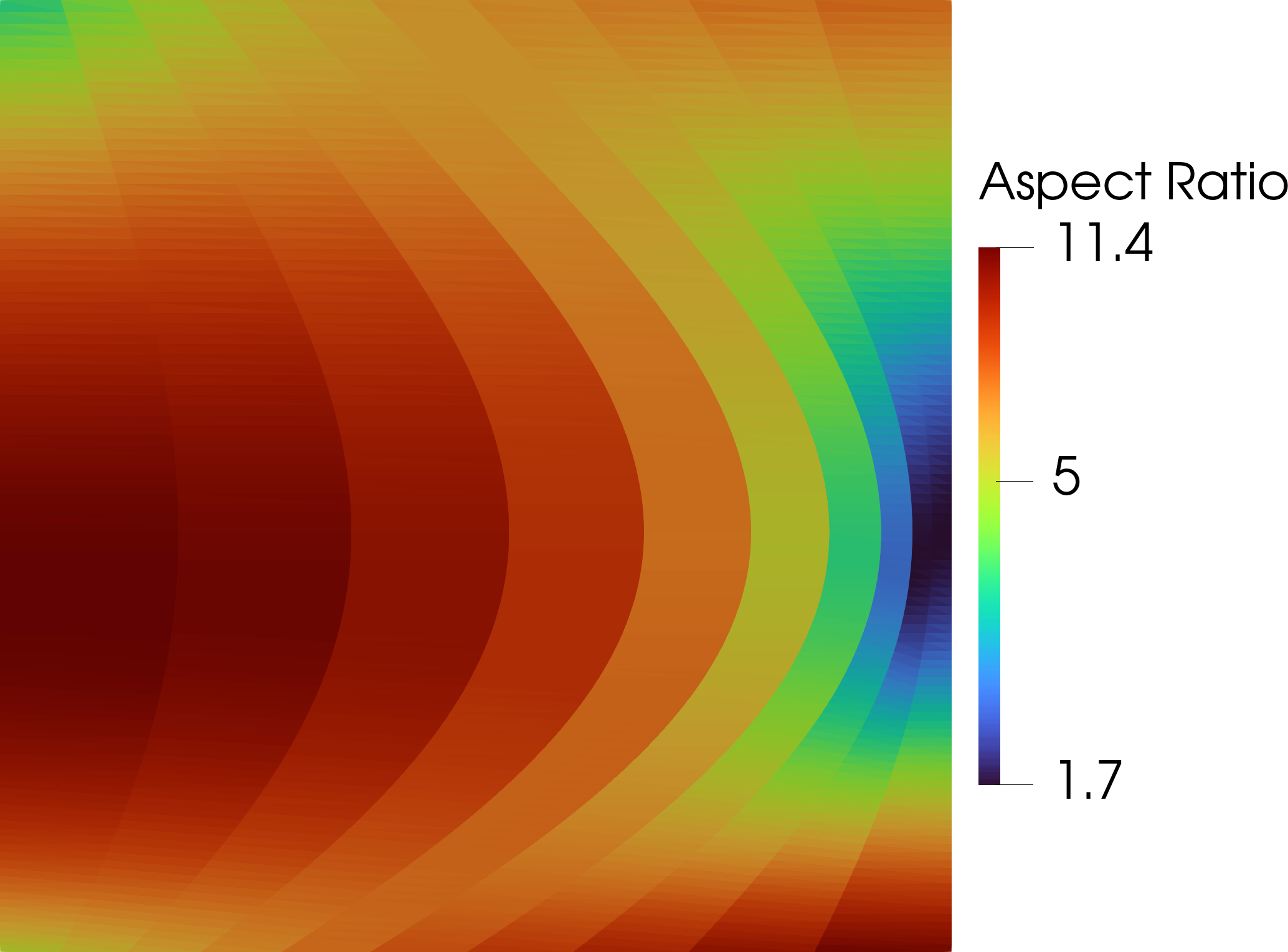}
		\caption{Aspect ratio of $M_1^{2}$}
	\end{subfigure}\hfill
	\begin{subfigure}{0.35\textwidth}
		\centering
		\includegraphics[scale=0.08]{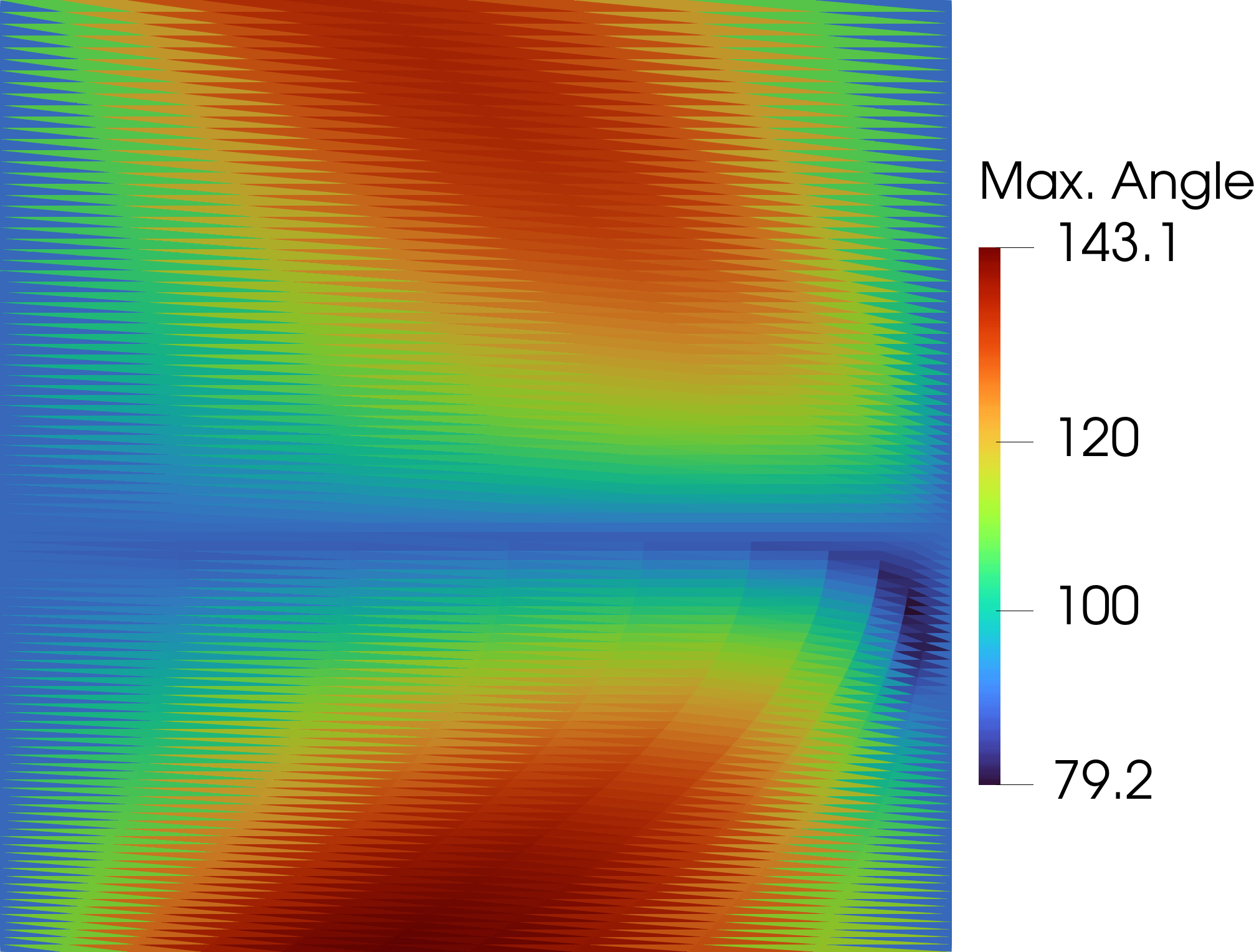}
		\caption{Maximum angle of $M_1^{2}$}
	\end{subfigure}
	\vspace{2mm}
	\\
	\begin{subfigure}{0.26\textwidth}
		\centering
		\includegraphics[scale=0.08]{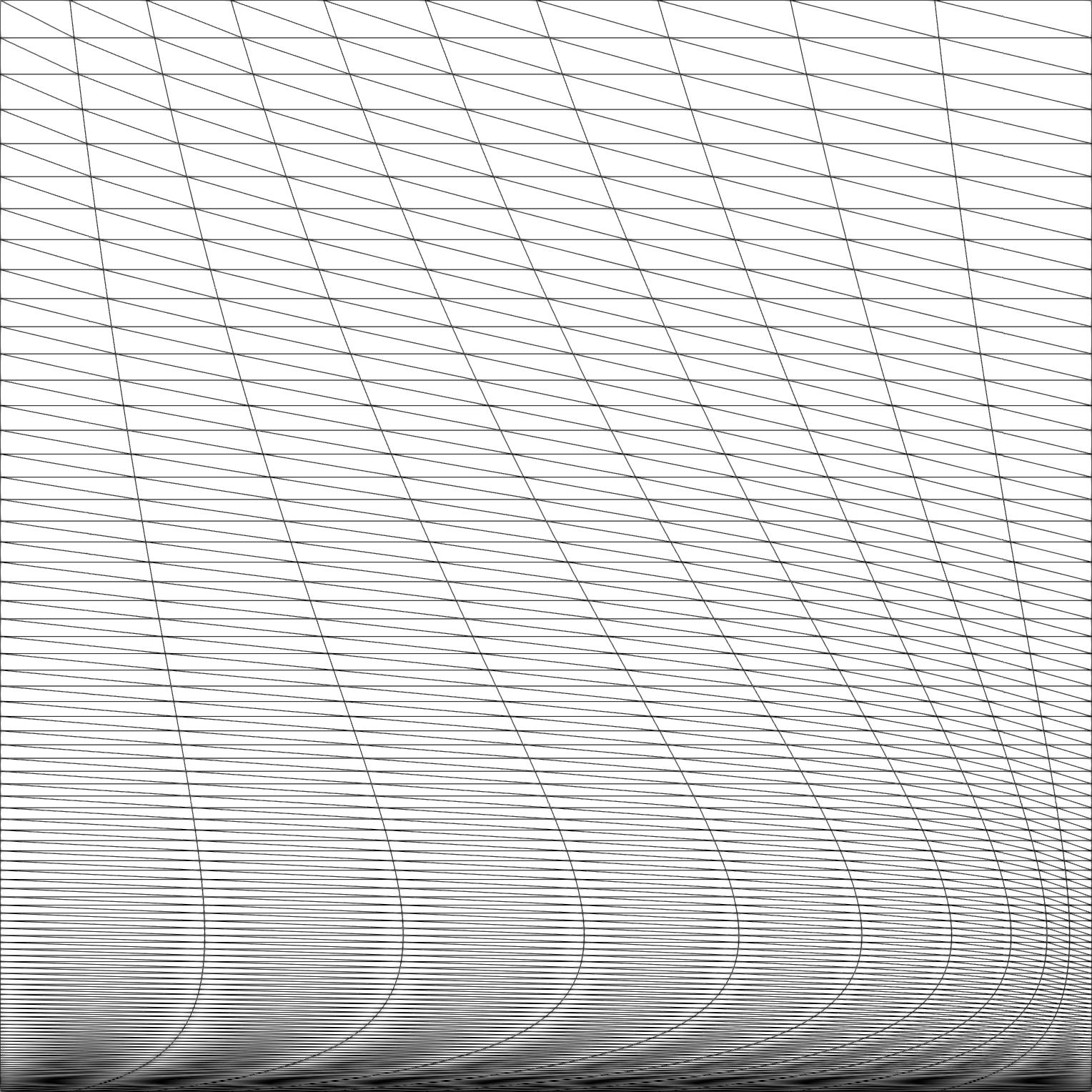}
		\caption{Moderately distorted mesh, $M_2^{2}$}
	\end{subfigure}\hfill
	\begin{subfigure}{0.35\textwidth}
		\centering
		\includegraphics[scale=0.08]{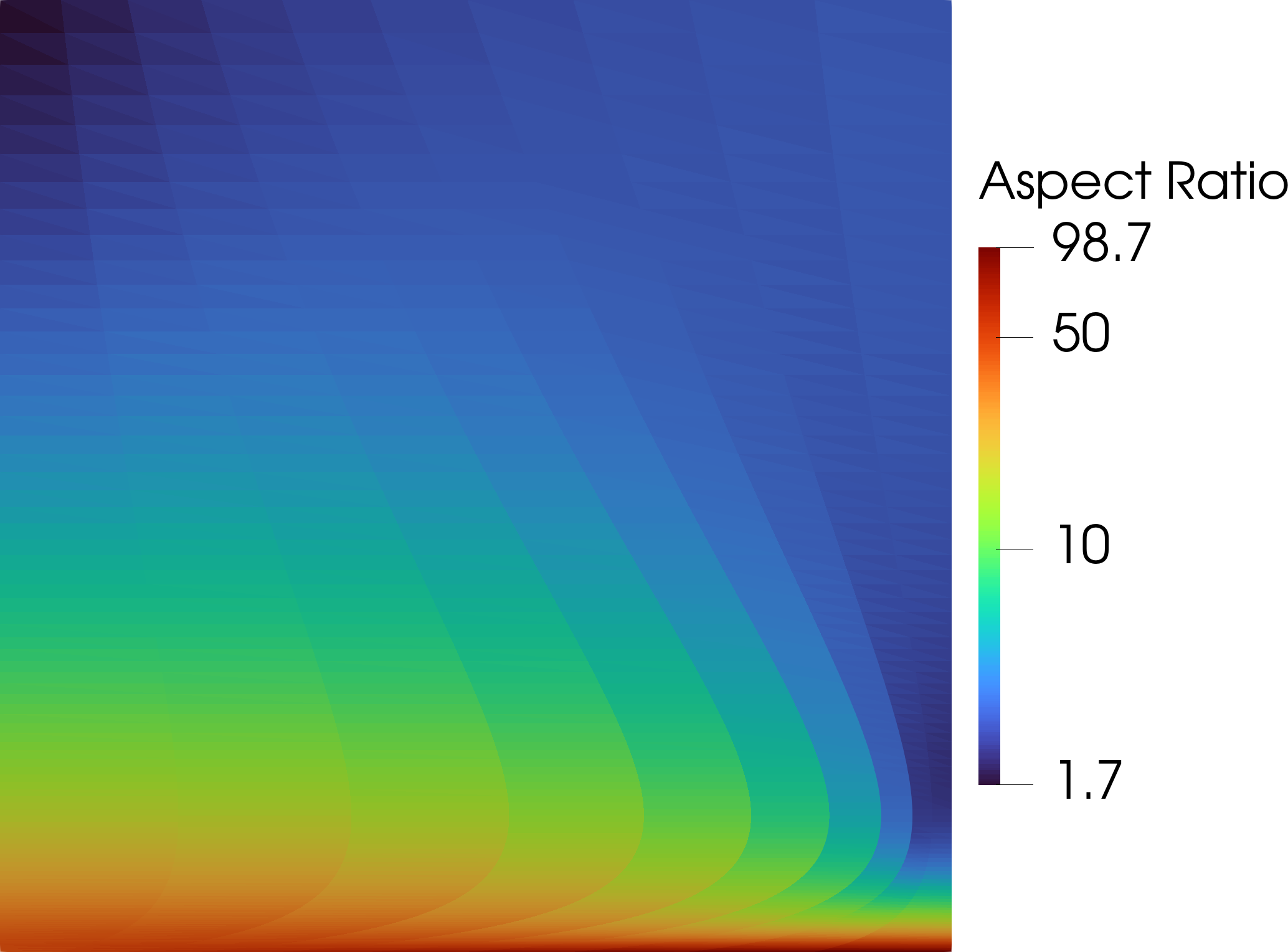}
		\caption{Aspect ratio of $M_2^{2}$}
	\end{subfigure}\hfill
	\begin{subfigure}{0.35\textwidth}
		\centering
		\includegraphics[scale=0.08]{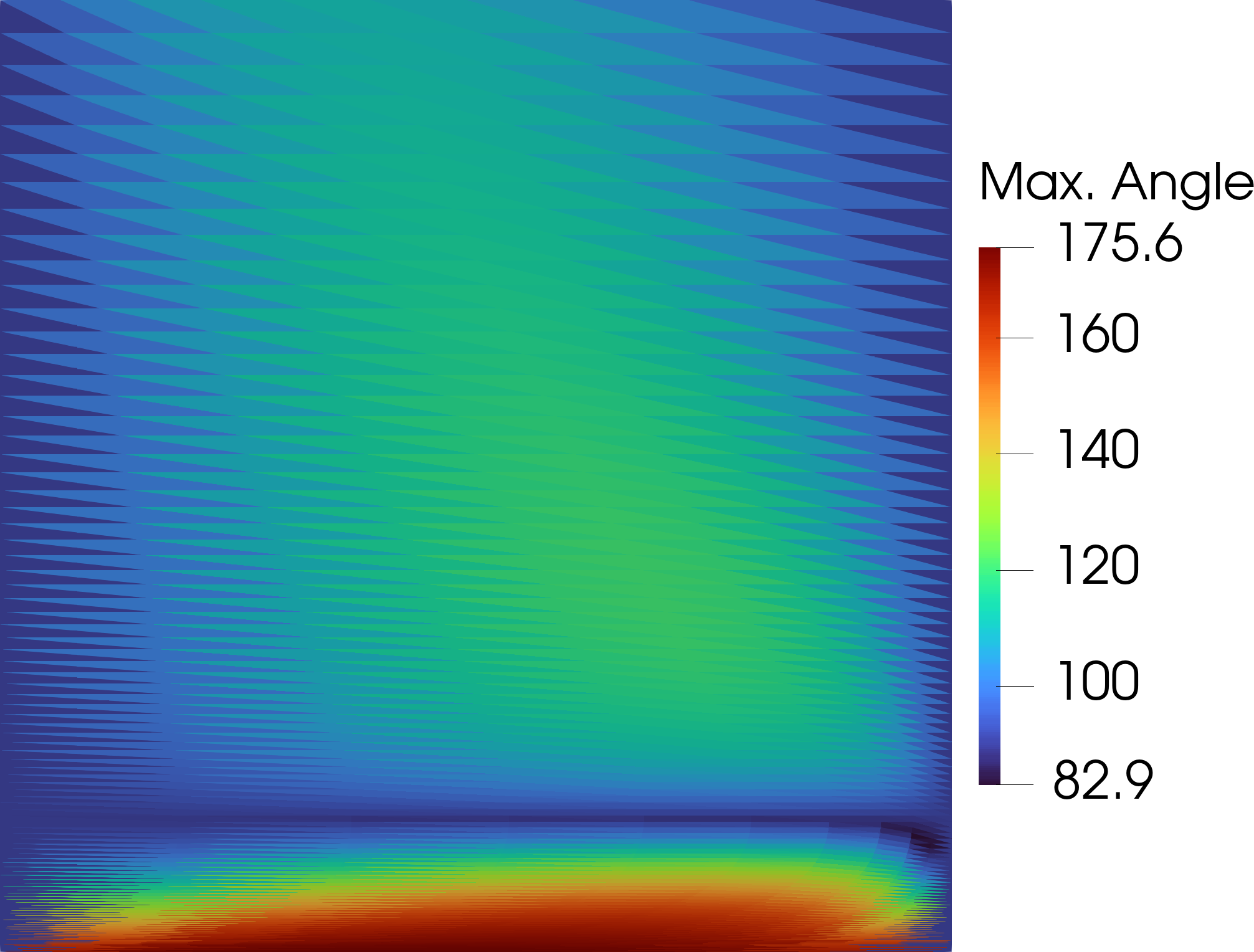}
		\caption{Maximum angle of $M_2^{2}$}
	\end{subfigure}
	\vspace{2mm}
	\\
	\begin{subfigure}{0.26\textwidth}
		\centering
		\includegraphics[scale=0.08]{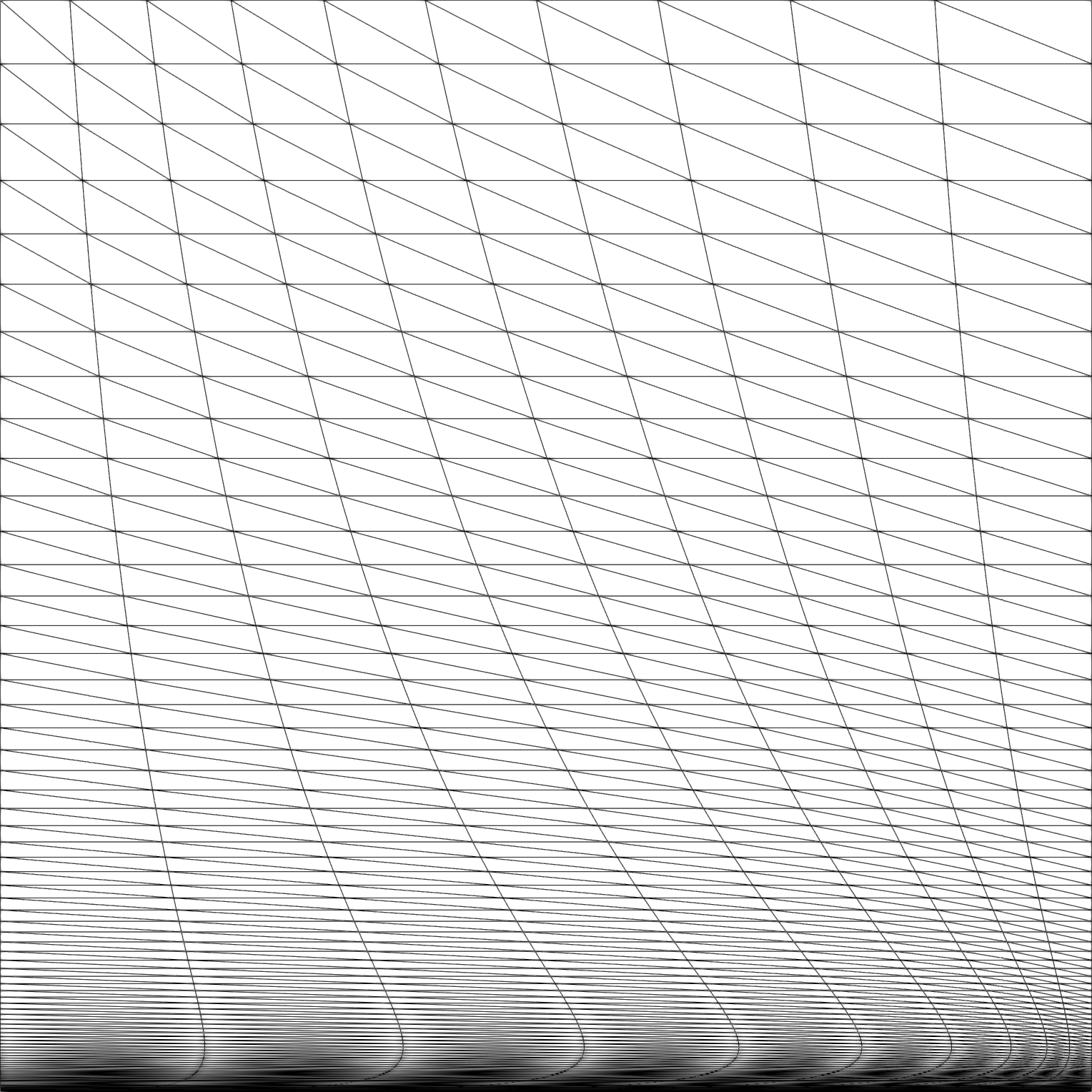}
		\caption{Highly distorted mesh, $M_3^{2}$}
	\end{subfigure}\hfill
	\begin{subfigure}{0.35\textwidth}
		\centering
		\includegraphics[scale=0.08]{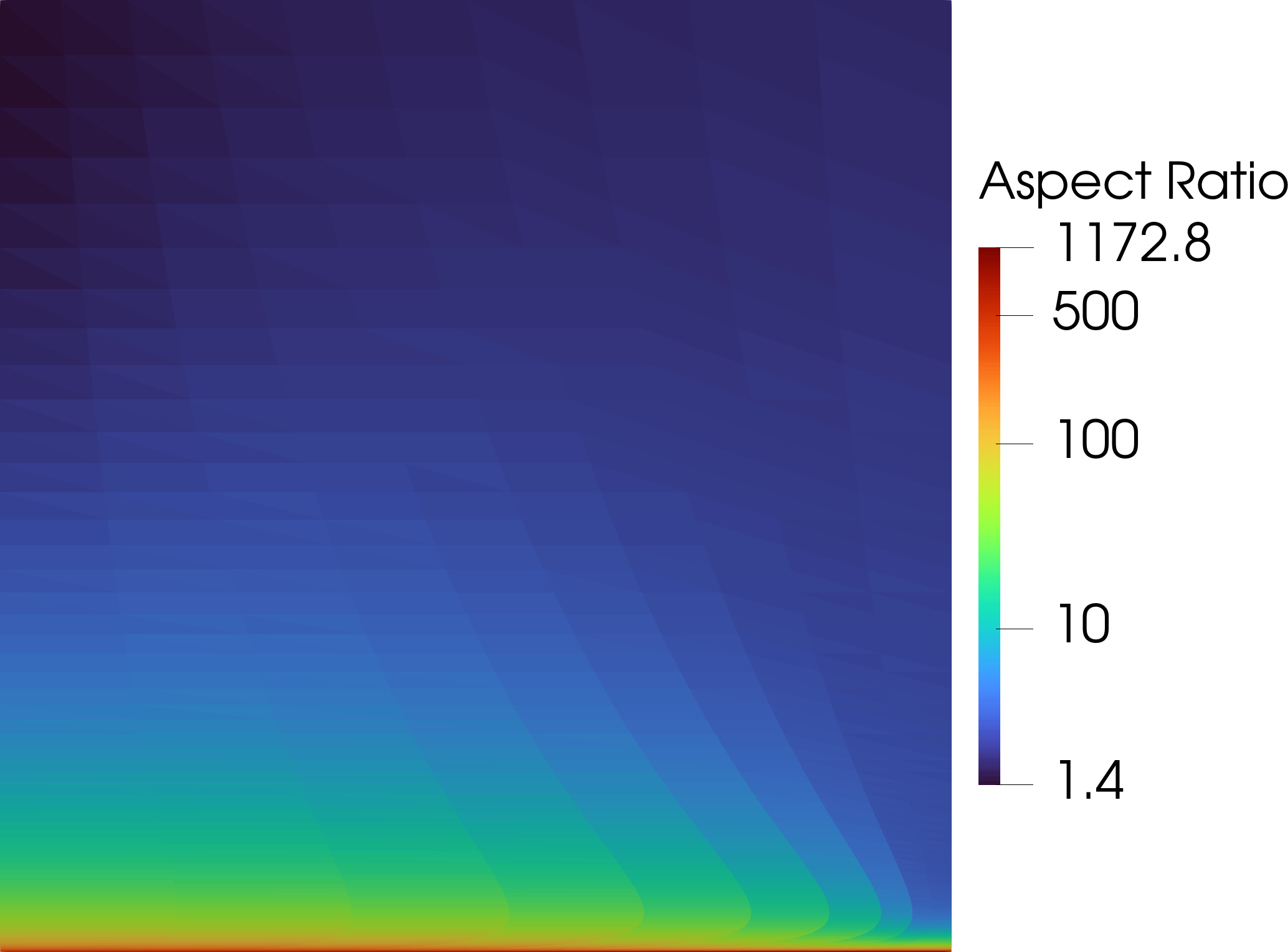}
		\caption{Aspect ratio of $M_3^{2}$}
	\end{subfigure}\hfill
	\begin{subfigure}{0.35\textwidth}
		\centering
		\includegraphics[scale=0.08]{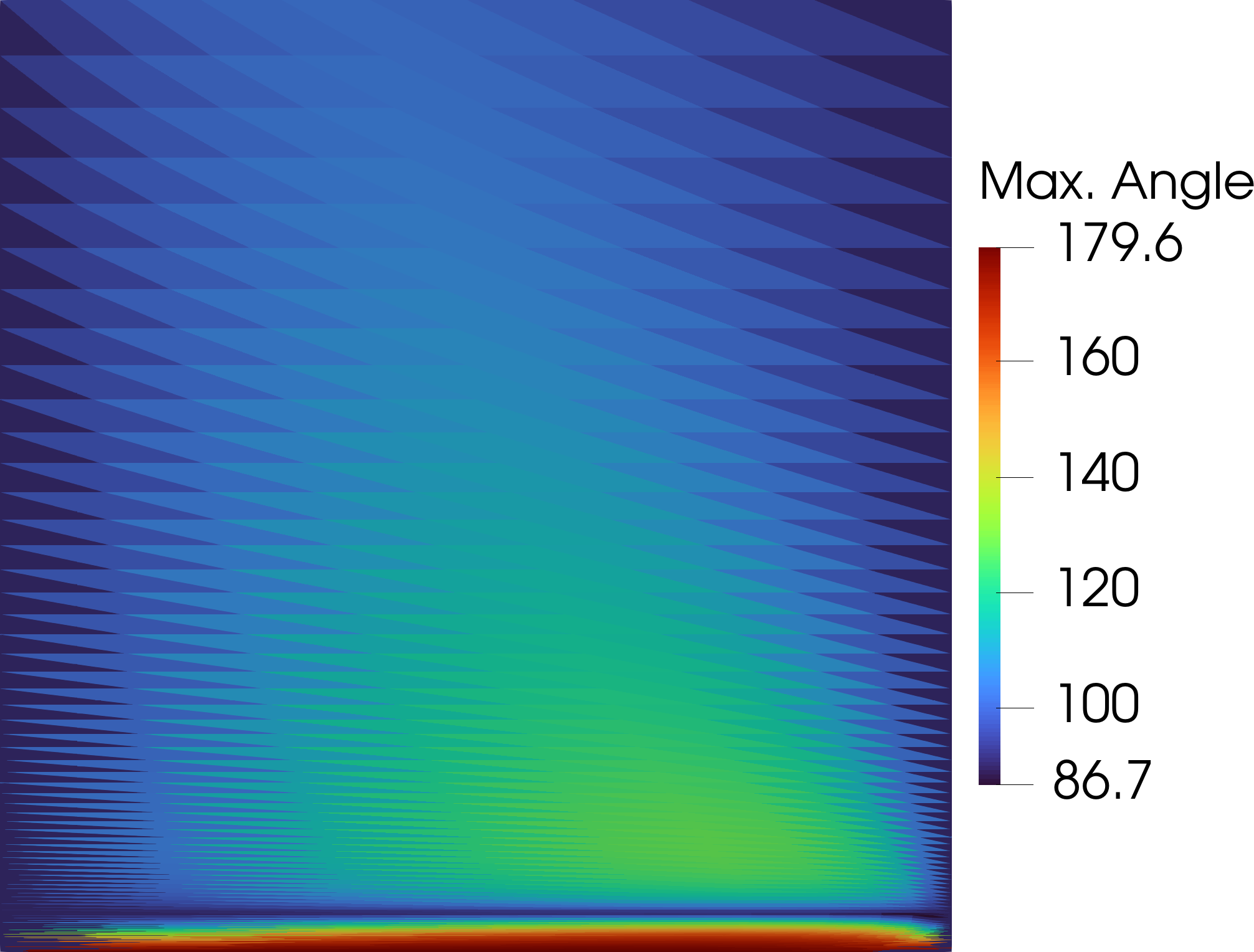}
		\caption{Maximum angle of $M_3^{2}$}
	\end{subfigure}
	\caption{\label{fig:distorted mesh} Distorted meshes with 2000 triangular elements to study the accuracy and maximum stable time step of the TPSS operators for the advection problem.}
\end{figure} 

In order to test the performance of the TPSS operators on distorted meshes, we apply the perturbation function 
\begin{equation}\label{eq:adv mapping}
	\begin{aligned}
			{x}_{1}	&=\hat{x}_{1}\exp(0.5(\hat{x}_{1}-1))+0.4\sin(\pi \hat{x}_{1})\sin(\pi \hat{x}_{2}),\\
		{x}_{2}	&=\hat{x}_{2}\exp\left(\alpha_{m}\left(\hat{x}_{2}-1\right)\right),
	\end{aligned}
\end{equation}
to a triangular mesh denoted by $M_{0}^{2}$, where the superscript denotes the spatial dimension, and obtained by subdividing a quadrilateral mesh constructed using $10$ and $100$ edges in the $x_1$ and $x_2$ directions, respectively. The hat symbol, $ (\hat{\cdot}) $, denotes the coordinate of a node before perturbation. We consider three values of $\alpha_{m}$ to produce three sets of meshes with slight, $\alpha_{m}=0.25$, moderate, $\alpha_{m}=2.5$, and high, $\alpha_{m}=5$, perturbations, which are denoted by $M_{1}^{2}$, $M_{2}^{2}$, and $M_{3}^{2}$, respectively, see \cref{fig:distorted mesh}. The mapping in \cref{eq:adv mapping} is applied to the nodes of degree-one multidimensional Lagrange finite elements constructed on the uniform triangular mesh, and the SBP nodes are obtained using a linear polynomial interpolation with the perturbed Lagrange finite-element nodes. The distorted meshes have different values of maximum aspect ratio, which is a critical parameter that affects the accuracy of operators on simplicial meshes, \eg, see \cite{chang2017cause,sun2012impact}.

The advection problem with $\omega=2$ is run on the uniform and the three distorted meshes using a CFL value of $0.5$ and up to $t=1$. \cref{tab:adv mesh} shows the $\H$-norm and $L^{\infty}$-norm errors produced by the dense SBP diagonal-$\E$ and TPSS operators. The table indicates that the ratios of the solution errors produced by the degree $p=1$ through $p=5$ dense SBP diagonal-$\E$ and TPSS operators remain consistently close on all three meshes. This suggests that the relative solution errors from the TPSS operators do not deteriorate on distorted triangular meshes. Hence, the TPSS operators retain their efficiency advantages relative to the dense SBP diagonal-$\E$ operators on triangular meshes with high aspect ratios and angles. However, further studies are required to ensure that the TPSS operators do not suffer on practically relevant meshes as reported in \cite{blacker2000meeting} for the method of splitting simplicial meshes into quadrilateral and hexahedral elements and applying tensor-product operators. 

\begin{table*}[t]
	\footnotesize
	\centering
	\captionsetup{skip=3pt}
	\caption{\label{tab:adv mesh} Comparison of the change in accuracy with degree of mesh distortion for the advection problem on triangular meshes using dense SBP diagonal-$\E$ and the TPSS operators.}
	\begin{threeparttable}
		\setlength{\tabcolsep}{0.25em}
		\renewcommand*{\arraystretch}{1.2}
		\begin{tabular}{cl @{\hspace{1.0em}}lllll lllll}
			\toprule
				& &  \multicolumn{2}{l}{$p=1$ }&  \multicolumn{2}{l}{$p=2$} &  \multicolumn{2}{l}{$p=3$} &  \multicolumn{2}{l}{$p=4$} &  \multicolumn{2}{l}{$p=5$}   \\
			\cmidrule(l{0em}r{0.4em}){3-4} \cmidrule(l{0em}r{0.4em}){5-6} \cmidrule(l{0em}r{0.4em}){7-8} \cmidrule(l{0em}r{0.4em}){9-10} \cmidrule(l{0em}r{0.4em}){11-12}  
			Mesh & & $\H$-norm& $L^\infty$-norm & $\H$-norm& $L^\infty$-norm& $\H$-norm& $L^\infty$-norm& $\H$-norm& $L^\infty$-norm& $\H$-norm& $L^\infty$-norm\\
			\midrule 
			\multirow[t]{3}{*}{$M_{0}^{2}$} & SBP-$\E$ & 1.171e-02& 3.803e-02&4.557e-04& 3.020e-03& 2.141e-05& 1.363e-04& 1.416e-06& 7.281e-06& 1.046e-07& 4.149e-07\\
			& TPSS &7.810e-04& 6.616e-03&1.941e-05& 1.262e-04&4.243e-07 & 4.242e-06& 2.074e-08& 1.447e-07&1.220e-10 & 1.359e-09\\
			& Ratio&15.0 & 5.8 & 23.5 & 23.9 & 50.5 & 32.1 & 68.3 & 50.3 & 856.8 & 305.4  \\
			\addlinespace
			\multirow[t]{3}{*}{$M_{1}^{2}$} & SBP-$\E$ & 2.309e-02& 1.433e-01&1.403e-03&1.853e-02 & 9.377e-05& 1.483e-03& 1.048e-05& 1.304e-04& 8.839e-07& 7.939e-06\\
			& TPSS &1.713e-03 & 1.687e-02& 5.616e-05& 7.792e-04& 1.957e-06& 3.500e-05& 6.458e-08& 9.840e-07& 1.512e-09& 3.521e-08 \\
			& Ratio &13.5 & 8.5 & 25.0 & 23.8 & 47.9 & 42.4 & 162.2 & 132.5 & 584.8 & 225.5 \\
			\addlinespace
			\multirow[t]{3}{*}{$M_{2}^{2}$} & SBP-$\E$ & 2.185e-02& 1.543e-01&1.148e-03 & 1.190e-02& 7.292e-05& 8.028e-04& 7.121e-06& 8.174e-05& 6.565e-07& 5.859e-06\\
			& TPSS & 1.419e-03& 1.395e-02&5.637e-05& 5.522e-04& 1.741e-06& 2.103e-05& 7.011e-08& 1.618e-06& 9.973e-10&2.003e-08 \\
			& Ratio &15.4 & 11.1 & 20.4 & 21.5 & 41.9 & 38.2 & 101.6 & 50.5 & 658.2 & 292.5 \\
			\addlinespace
			\multirow[t]{3}{*}{$M_{3}^{2}$} & SBP-$\E$ &2.222e-02 &1.978e-01 & 1.174e-03& 1.292e-02& 8.440e-05& 1.206e-03& 6.394e-06&7.923e-05 &6.382e-07&6.872e-06 \\
			& TPSS & 1.444e-03& 1.366e-02&7.497e-05& 6.199e-04& 1.986e-06& 2.982e-05& 6.156e-08& 9.531e-07& 1.349e-09& 2.350e-08\\
			& Ratio& 15.4 & 14.5 & 15.7 & 20.8 & 42.5 & 40.5 & 103.9 & 83.1 & 473.2 & 292.5\\
			\bottomrule
		\end{tabular}
	\end{threeparttable}
\end{table*}

We also investigate the effect of the distorted meshes on the maximum time step that can be taken with the TPSS operators relative to the dense SBP diagonal-$\E$ operator. For this study, a time step is deemed stable if it produces a change in energy of zero or less than zero after one period. \cref{tab:adv cfl warped 2d} shows the maximum time step that can be taken on the uniform and distorted triangular meshes. As can be seen from the table, the maximum time step with the TPSS operators is slightly reduced on the distorted mesh compared to those obtained on the $M_0^{2}$ mesh; however, the operators still offer stable time steps greater than $50\%$ of those achieved with the dense SBP diagonal-$\E$ operators. Given the number of nodes in the TPSS operators and the severity of the mesh perturbation, with aspect ratios up to $1172.8$ and interior angle up to $179.6$ degrees, the reduction in the maximum time step can be regarded as minimal. 
\begin{table*}[t]
	\footnotesize
	\centering
	\captionsetup{skip=3pt}
	\caption{\label{tab:adv cfl warped 2d} Maximum time-step restrictions and percent ratios ($r$) for the TPSS and dense SBP diagonal-$\E$ operators applied for the advection problem on the uniform and distorted triangular meshes.}
	\begin{threeparttable}
		\setlength{\tabcolsep}{0.125em}
		\renewcommand*{\arraystretch}{1.5}
		\begin{tabular}{l@{\hspace{0.4em}}lll @{\hspace{0.4em}}lll @{\hspace{0.4em}}lll @{\hspace{0.4em}}lll @{\hspace{0.4em}}lll}
			\toprule
			&  \multicolumn{3}{l}{$p=1$ }&  \multicolumn{3}{l}{$p=2$} &  \multicolumn{3}{l}{$p=3$} &  \multicolumn{3}{l}{$p=4$} &  \multicolumn{3}{l}{$p=5$}   \\
			\cmidrule(l{0em}r{0.5em}){2-4} \cmidrule(l{0em}r{0.5em}){5-7} \cmidrule(l{0em}r{0.5em}){8-10} \cmidrule(l{0em}r{0.5em}){11-13} \cmidrule(l{0em}r{0.3em}){14-16}  
			Mesh&TPSS&SBP-$\E$& $r (\%)$& TPSS&SBP-$\E$& $r (\%)$& TPSS&SBP-$\E$& $r (\%)$& TPSS&SBP-$\E$& $r (\%)$& TPSS&SBP-$\E$& $r (\%)$\\
			\midrule 
			$M_{0}^{2}$&2.89e-03 & 4.00e-03 & 72.2 & 1.62e-03 & 2.28e-03 & 70.9 & 1.05e-03 & 1.37e-03 & 76.8 & 7.53e-04 & 1.02e-03 & 73.6 & 5.68e-04 & 6.30e-04 & 90.2 \\
			$M_{1}^{2}$& 2.48e-03 & 3.52e-03 & 70.5 & 1.33e-03 & 2.03e-03 & 65.5 & 8.06e-04 & 1.26e-03 & 64.0 & 5.37e-04 & 9.06e-04 & 59.3 & 3.82e-04 & 5.65e-04 & 67.5\\
			$M_{2}^{2}$&3.56e-04 & 5.46e-04 & 65.1 & 1.95e-04 & 3.17e-04 & 61.4 & 1.21e-04 & 2.07e-04 & 58.5 & 8.22e-05 & 1.50e-04 & 54.7 & 5.95e-05 & 1.03e-04 & 57.7 \\
			$M_{3}^{2}$&2.98e-05 & 4.54e-05 & 65.6 & 1.64e-05 & 2.62e-05 & 62.3 & 1.02e-05 & 1.71e-05 & 59.5 & 6.93e-06 & 1.24e-05 & 55.7 & 5.01e-06 & 8.52e-06 & 58.8 \\
			\bottomrule
		\end{tabular}
	\end{threeparttable}
\end{table*}

We conduct a similar study in three dimensions to investigate the effect of the quality of tetrahedral meshes on the accuracy and maximum time step of the TPSS operators. The three-dimensional advection problem is run on a uniform mesh, $M_{0}^{3}$, produced by dividing a $6\times6\times6$ hexahedral elements and two distorted meshes obtained by applying the perturbation function,
\begin{equation}
	\begin{aligned}
		x_{1} &= \hat{x}_{1} \exp(0.5(\hat{x}_{1}-1)) + 0.4\sin(\pi \hat{x}_1)\sin(\pi \hat{x}_2)\sin(\pi \hat{x}_3),\\
		x_{2} &= \hat{x}_{2} \exp(0.5(\hat{x}_{2}-1)) + 0.4\sin(\pi \hat{x}_1)\sin(\pi \hat{x}_2)\sin(\pi \hat{x}_3),\\
		x_{3} &= \hat{x}_{3} \exp(\alpha_{m}(\hat{x}_{3}-1)),
	\end{aligned}
\end{equation}
where $\alpha_{m}=0.03$ and $\alpha_{m}=3$ are used to generate a slightly distorted, $M_{1}^{3}$, and moderately distorted, $M_{2}^{3}$, meshes shown in \cref{fig:warped mesh 3d}. 
\begin{figure}[t]
	\centering
	\captionsetup{belowskip=0pt, aboveskip=4pt}
	\begin{subfigure}{0.5\textwidth}
		\centering
		\includegraphics[scale=0.06]{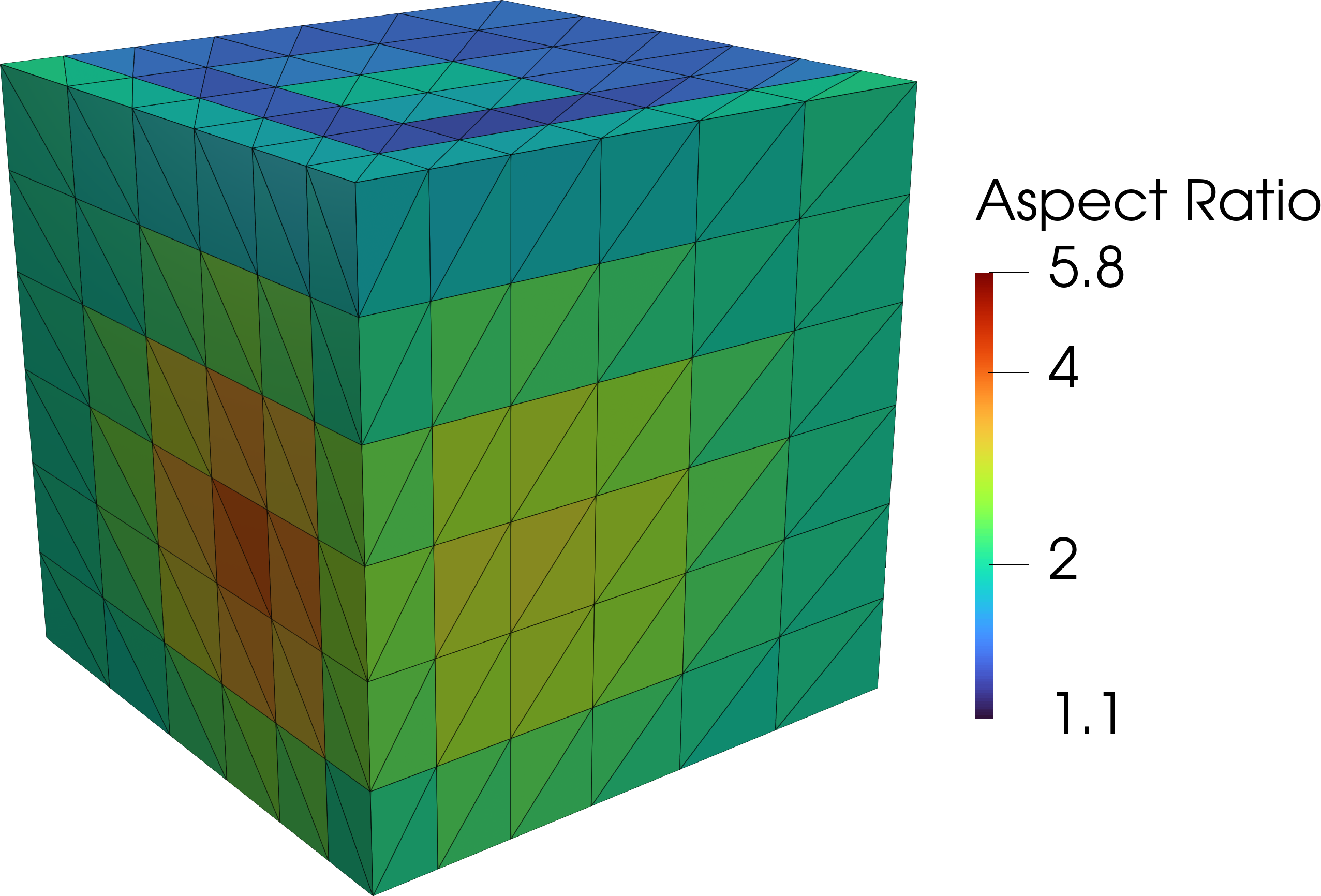}
		\caption{Slightly distorted mesh, $M_1^{3}$}
	\end{subfigure}\hfill
	\begin{subfigure}{0.5\textwidth}
		\centering
		\includegraphics[scale=0.06]{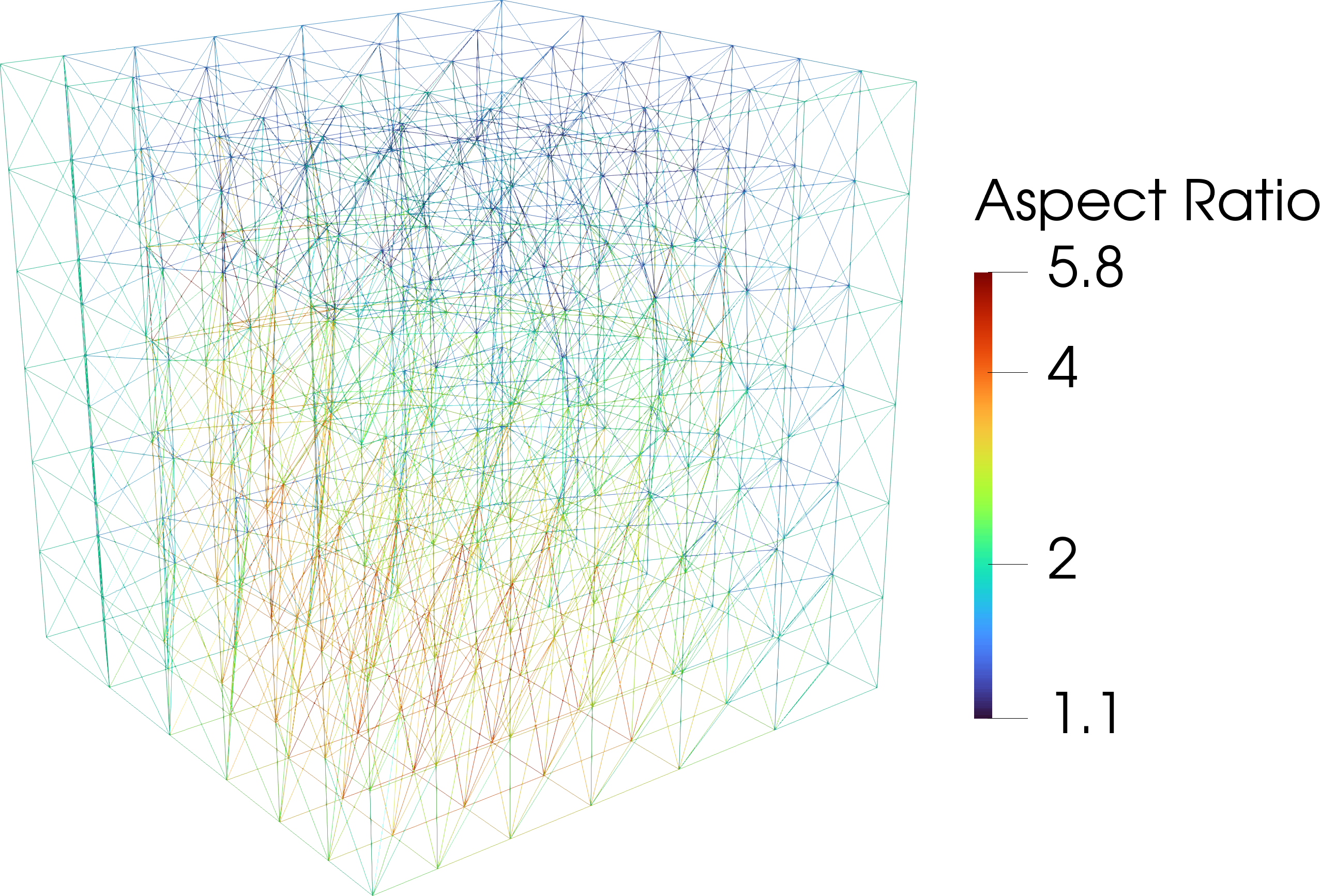}
		\caption{Aspect ratio of $M_1^{3}$}
	\end{subfigure} 
	\vspace{3mm}
	\\
	\begin{subfigure}{0.5\textwidth}
		\centering
		\includegraphics[scale=0.06]{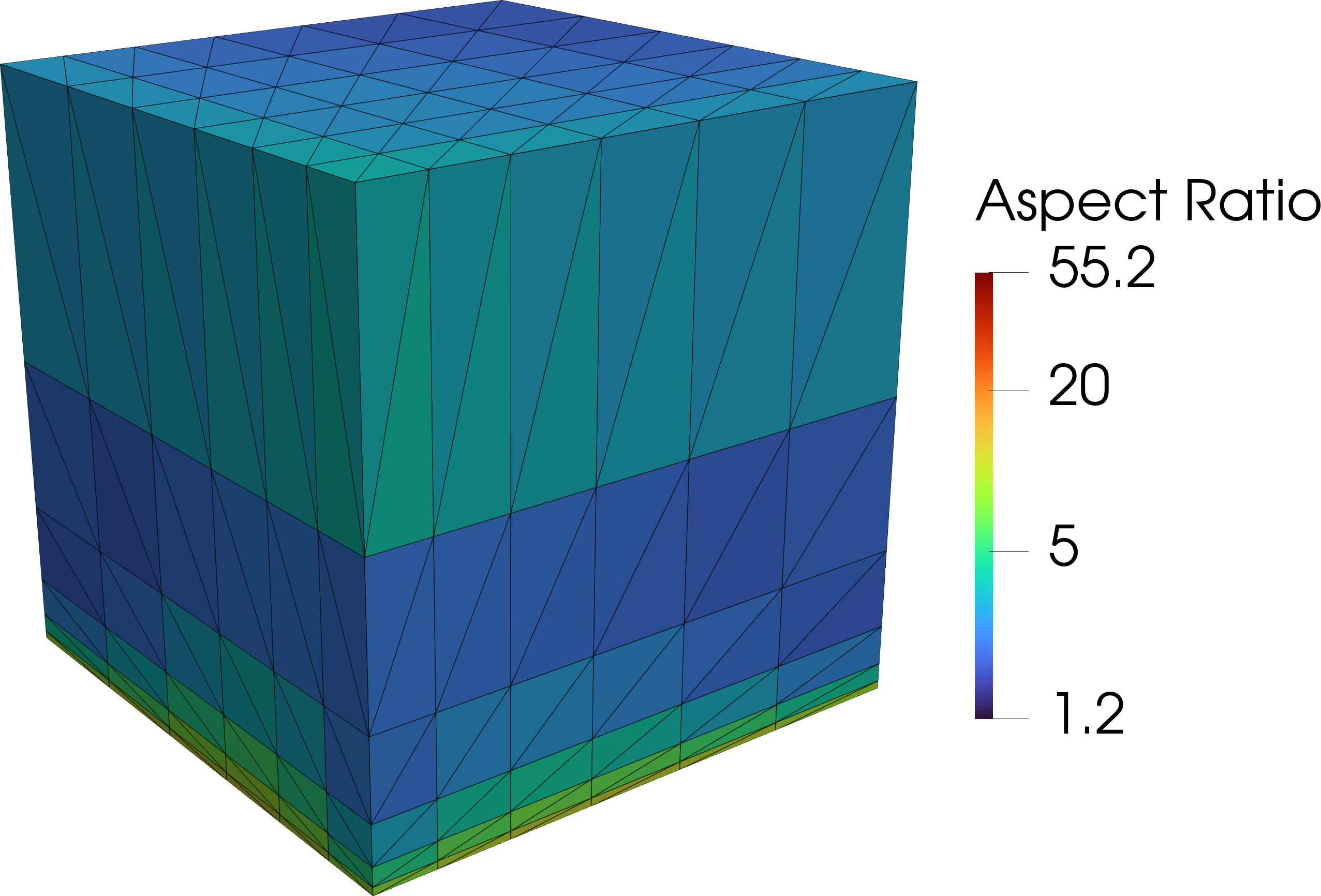}
		\caption{Moderately distorted mesh, $M_2^{3}$}
	\end{subfigure}\hfill
	\begin{subfigure}{0.5\textwidth}
		\centering
		\includegraphics[scale=0.06]{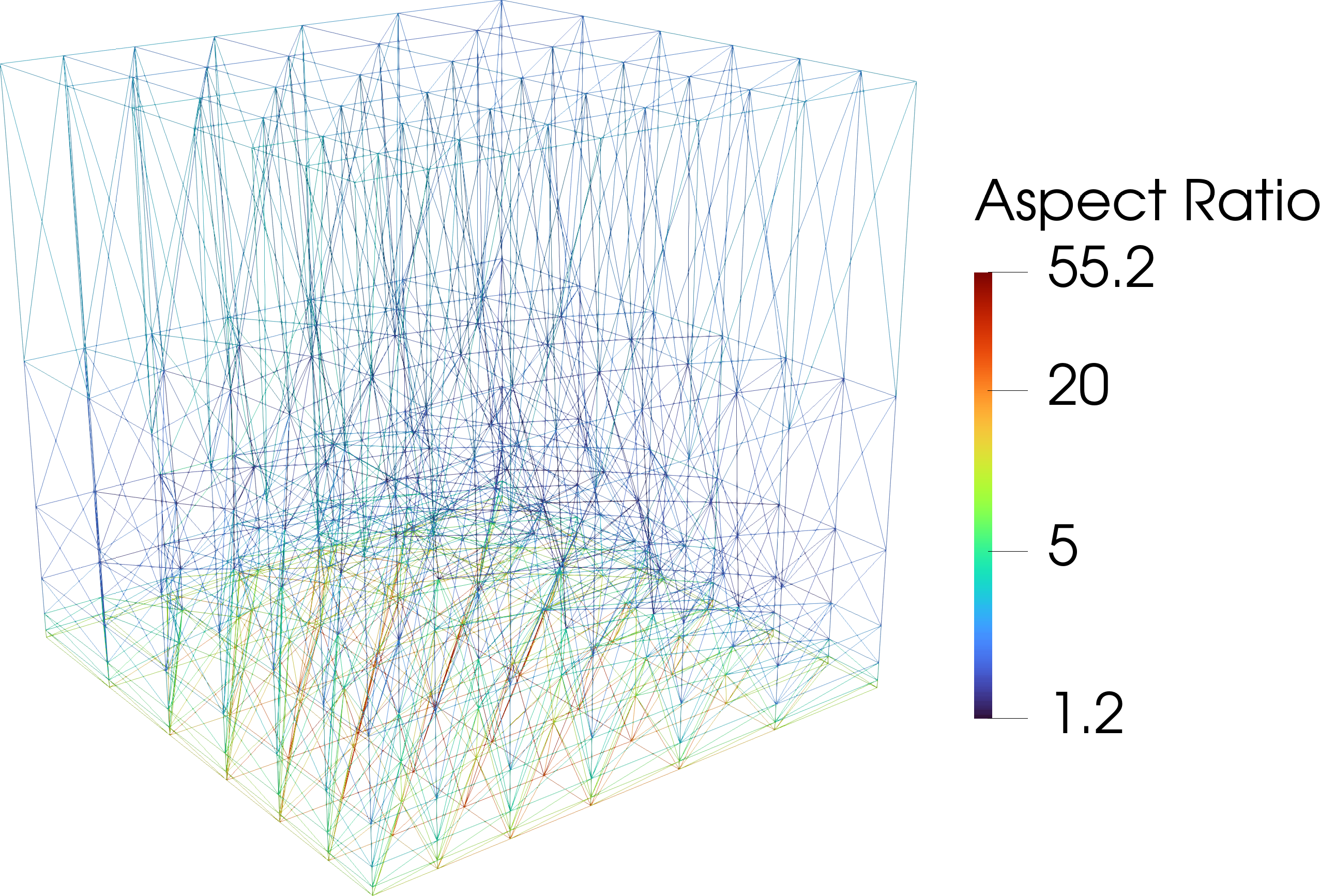}
		\caption{Aspect ratio of $M_2^{3}$}
	\end{subfigure} 
	\caption{\label{fig:warped mesh 3d} Distorted meshes with $1296$ tetrahedral elements to study the accuracy and maximum stable time step of the TPSS operators for the advection problem.}
\end{figure}

The results for the change in accuracy with mesh distortion are shown in \cref{tab:adv 3d mesh}. As the order of the operator increases, the $\H$-norm solution errors with the TPSS operators reduce at faster rates than those obtained with the dense SBP diagonal-$\E$ operators, reaching up to $5199.4$ times smaller values at $p=5$ on the uniform mesh. While this error reduction on the same number of elements (as opposed to the same number of degrees of freedom) does not necessarily translate into equivalent efficiency gains, it can be concluded that the TPSS operators will be even more efficient at higher operator degrees due to their increasing sparsity (see \cref{tab:nnz TPSS}) and accuracy. That being said, the errors from the TPSS operators on the distorted mesh increase by larger factors compared to those of the dense SBP diagonal-$\E$ operators. However, as in the 2D case, the increase in error values does not appear to worsen significantly with the severity of the mesh distortion. Note that $M_{2}^{3}$ has a maximum aspect ratio of $55.2$, which is an order of magnitude larger than that of $M_{1}^{3}$, but the ratios of the errors from the dense SBP diagonal-$\E$ to those of the TPSS operators on $M_{2}^{3}$ remain close to the ratios obtained on the $M_{1}^{3}$ mesh.
\begin{table*}[t]
	\footnotesize
	\centering
	\captionsetup{skip=3pt}
	\caption{\label{tab:adv 3d mesh} Comparison of the change in accuracy for the advection problem on the distorted and uniform tetrahedral meshes using dense SBP diagonal-$\E$ and the TPSS operators.}
	\begin{threeparttable}
		\setlength{\tabcolsep}{0.265em}
		\renewcommand*{\arraystretch}{1.2}
		\begin{tabular}{cl @{\hspace{1.0em}}lllll lllll}
			\toprule
			& &  \multicolumn{2}{l}{$p=1$ }&  \multicolumn{2}{l}{$p=2$} &  \multicolumn{2}{l}{$p=3$} &  \multicolumn{2}{l}{$p=4$} &  \multicolumn{2}{l}{$p=5$}   \\
			\cmidrule(l{0em}r{0.4em}){3-4} \cmidrule(l{0em}r{0.4em}){5-6} \cmidrule(l{0em}r{0.4em}){7-8} \cmidrule(l{0em}r{0.4em}){9-10} \cmidrule(l{0em}r{0.4em}){11-12}  
			Mesh & & $\H$-norm& $L^\infty$-norm & $\H$-norm& $L^\infty$-norm& $\H$-norm& $L^\infty$-norm& $\H$-norm& $L^\infty$-norm& $\H$-norm& $L^\infty$-norm\\
			\midrule 
			\multirow[t]{3}{*}{$M_{0}^{3}$} & SBP-$\E$ & 7.932e-02& 2.002e-01& 2.609e-02& 8.016e-02& 2.336e-03& 1.786e-02& 3.700e-04&3.425e-03 & 5.765e-05& 9.183e-04  \\
			& TPSS & 1.061e-03& 1.499e-02& 8.406e-05& 1.012e-03& 4.481e-06& 6.060e-05&3.099e-07 & 4.618e-06&1.109e-08 &2.961e-07   \\
			& Ratio&74.7 & 13.4 & 310.4 & 79.2 & 521.2 & 294.7 & 1194.0 & 741.6 & 5199.4 & 3101.4 \\
			\addlinespace
			\multirow[t]{3}{*}{$M_{1}^{3}$} & SBP-$\E$  & 1.322e-01 & 4.524e-01 & 5.806e-02 & 4.998e-01 & 1.144e-02 & 2.636e-01 & 2.950e-03 & 1.338e-01 & 6.244e-04 & 2.763e-02\\
			& TPSS & 2.727e-03 & 6.152e-02 & 3.085e-04 & 1.150e-02 & 4.229e-05 & 2.208e-03 & 4.015e-06 & 2.510e-04 & 4.313e-07 & 4.991e-05 \\
			& Ratio & 48.5 & 7.3 & 188.2 & 43.5 & 270.4 & 119.4 & 734.8 & 533.1 & 1447.7 & 553.6\\
			\addlinespace
			\multirow[t]{3}{*}{$M_{2}^{3}$} & SBP-$\E$ &1.870e-01 & 8.810e-01& 8.032e-02& 6.056e-01& 1.637e-02& 2.012e-01& 3.948e-03& 7.191e-02& 8.071e-04& 1.916e-02  \\
			& TPSS & 3.977e-03&5.105e-02 & 4.371e-04& 1.058e-02&6.438e-05 & 1.843e-03& 5.057e-06& 1.766e-04& 5.821e-07& 1.851e-05  \\
			& Ratio &47.0 & 17.3 & 183.8 & 57.2 & 254.2 & 109.2 & 780.8 & 407.2 & 1386.6 & 1034.8 \\
			\bottomrule
		\end{tabular}
	\end{threeparttable}
\end{table*}

Finally, \cref{tab:adv cfl warped 3d} shows the maximum stable time-step values obtained with the TPSS and dense SBP diagonal-$\E$ operators on the distorted three-dimensional meshes. The criterion for determining the maximum stable time step is the same as the 2D case. The TPSS operators on the tetrahedral meshes show a reduction in the stable maximum time step on the distorted meshes. However, as in the 2D case, this reduction does not seem to worsen significantly from $M_{1}^{3}$ to $M_{2}^{3}$. While further research is required to examine the properties of the TPSS operators on highly distorted meshes, the current results suggest that, relative to the dense SBP diagonal-$\E$ operators, neither their errors nor their stable time-step values deteriorate substantially with increasing distortion of meshes in both two and three dimensions.
\begin{table*}[t]
	\footnotesize
	\centering
	\captionsetup{skip=3pt}
	\caption{\label{tab:adv cfl warped 3d} Maximum time-step restrictions and percent ratios ($r$) for the TPSS and dense SBP diagonal-$\E$ operators  applied for the advection problem on the uniform and distorted tetrahedral meshes.}
	\begin{threeparttable}
		\setlength{\tabcolsep}{0.125em}
		\renewcommand*{\arraystretch}{1.5}
		\begin{tabular}{l@{\hspace{0.4em}}lll @{\hspace{0.4em}}lll @{\hspace{0.4em}}lll @{\hspace{0.4em}}lll @{\hspace{0.4em}}lll}
			\toprule
			&  \multicolumn{3}{l}{$p=1$ }&  \multicolumn{3}{l}{$p=2$} &  \multicolumn{3}{l}{$p=3$} &  \multicolumn{3}{l}{$p=4$} &  \multicolumn{3}{l}{$p=5$}   \\
			\cmidrule(l{0em}r{0.5em}){2-4} \cmidrule(l{0em}r{0.5em}){5-7} \cmidrule(l{0em}r{0.5em}){8-10} \cmidrule(l{0em}r{0.5em}){11-13} \cmidrule(l{0em}r{0.3em}){14-16}  
			Mesh &TPSS&SBP-$\E$& $r (\%)$& TPSS&SBP-$\E$& $r (\%)$& TPSS&SBP-$\E$& $r (\%)$& TPSS&SBP-$\E$& $r (\%)$& TPSS&SBP-$\E$& $r (\%)$\\
			\midrule 
			$M_{0}^{3}$&1.05e-02 & 2.63e-02 & 40.0 & 6.71e-03 & 4.41e-03 & 152.3 & 4.67e-03 & 6.71e-03 & 69.6 & 3.32e-03 & 3.34e-03 & 99.3 & 2.46e-03 & 3.98e-03 & 61.7 \\
			$M_{1}^{3}$ & 2.13e-03 & 7.04e-03 & 30.3 & 1.31e-03 & 1.31e-03 & 100.5 & 8.87e-04 & 2.45e-03 & 36.3 & 6.38e-04 & 1.07e-03 & 59.4 & 4.79e-04 & 8.67e-04 & 55.3\\
			$M_{2}^{3}$&6.15e-04 & 2.18e-03 & 28.2 & 3.76e-04 & 4.14e-04 & 90.8 & 2.53e-04 & 7.79e-04 & 32.4 & 1.81e-04 & 3.28e-04 & 55.4 & 1.36e-04 & 2.55e-04 & 53.4 \\
			\bottomrule
		\end{tabular}
	\end{threeparttable}
\end{table*}
	
\subsection{Isentropic vortex problem}
The isentropic vortex problem is a smooth problem governed by the Euler equations. Both the 2D and 3D versions of the problem are used to study accuracy, entropy conservation, and efficiency of the TPSS operators. The analytical solution for the 2D case on the domain $ \Omega = \left[0,20\right]\times\left[-5,5\right] $ is given as \cite{erlebacher1997interaction},
\begin{equation}
	\begin{aligned}
		\rho & =\left(1-\frac{\alpha^{2}\left(\gamma-1\right)}{16\gamma\pi}\exp\left(2\left(1-r^{2}\right)\right)\right)^{\frac{1}{\gamma-1}}, & p & =\rho^{\gamma},\\
		V_{1} & =1 - \frac{\alpha}{2\pi}\left({x_{2}}-y_{c}\right)\exp\left(1-r^{2}\right), &
		V_{2} & =\frac{\alpha}{2\pi}\left({x_{1}}-\left(x_{c}+t\right)\right)\exp\left(1-r^{2}\right),
	\end{aligned}
\end{equation}
where $r^{2}=\left({x_{1}}-x_{c}+t\right)^{2}+\left({x_{2}}-y_{c}\right)^{2}$,
$\alpha=3$ is the vortex strength, $\left(x_c,y_c\right)=\left(5,0\right)$ is the center of the vortex at $ t=0 $,
and $\gamma=7/5$. The grids are curved using the perturbation \cite{chan2019efficient}
\begin{equation} \label{eq:curve 2d vortex}
	\begin{aligned}{x_{1}} & =\hat{{x}}_{1}+\frac{1}{8}\cos\left(\frac{\pi}{20}\left(\hat{x}_{1}-10\right)\right)\cos\left(\frac{3\pi}{10}\hat{{x}}_{2}\right), &
		{x_{2}} & =\hat{{x}}_2+\frac{1}{8}\sin\left(\frac{\pi}{5}\left({x_{1}}-10\right)\right)\cos\left(\frac{\pi}{10}\hat{{x}}_{2}\right).
	\end{aligned}
\end{equation}
In 3D, we use the analytical solution \cite{williams2013nodal},
\begin{equation}
	\medmuskip=-1mu
	\begin{aligned}
		\rho & =\left(1-\frac{2}{25}\left(\gamma-1\right)\exp\left(1-\left({x_{2}}-t\right)^{2}-{x}_{1}^{2}\right)\right)^{\frac{1}{\gamma-1}}, \; e=\frac{\rho^{\gamma}}{\gamma(\gamma-1)}+\frac{\rho}{2}\left({V_{1}}^{2}+{V_{2}}^{2}+{V_{3}}^{2}\right),  \\
		{V_{1}} & =-\frac{2}{5}\left({x_{2}}-t\right)\exp\left(\frac{1}{2}\left[1-\left({x_{1}}-t\right)^{2}-{x}_{1}^{2}\right]\right), \;
		{V_{2}}  =1+\frac{2}{5}{x_{1}}\exp\left(\frac{1}{2}\left[1-\left({x_{2}}-t\right)^{2}-{x}_{1}^{2}\right]\right), \;
		{V_{3}} =0.
	\end{aligned}
\end{equation}
The domain is $\Omega = \left[-10,10\right]^{3}$, and the grids are curved according
to the perturbation 
\begin{equation}\label{eq:curve 3d vortex}
	{x}_{i}=\hat{{x}}_{i}+0.05\sin\left(\frac{\pi\hat{{x}}_{i}}{2}\right),\quad\forall i\in\{1,2,3\}.
\end{equation} 
The perturbations are applied as degree-two polynomial geometric mappings; \ie, first, the perturbation is applied to the nodes of the degree-two multidimensional Lagrange finite elements, then a degree-two polynomial interpolation is applied to find the perturbed nodes of the SBP operators.

The problem is solved with the TPSS and dense SBP diagonal-$\E$ operators. The goal is to investigate the efficiency of the TPSS operators and show that they produce entropy-conservative discretizations if no upwinding of the SATs or artificial dissipation is introduced. In all cases, we use sufficiently small CFL values to ensure that the temporal discretization errors are negligible.  The convergence rates for the 2D and 3D cases are tabulated in \cref{sec:grid conv ivp}, \cref{tab:grid conv ivp}, which shows that both the TPSS and dense SBP diagonal-$\E$ operators attain convergence rates between $p$ and $p+1$ in the $\H$-norm. \cref{fig:ivp acc} shows the accuracy per degree of freedom comparison between the TPSS and dense SBP diagonal-$\E$ operators computed at $t=2$ and $t=1$ for the 2D and 3D cases, respectively. Again, the TPSS operators perform much better in terms of accuracy for a given number of degrees of freedom, producing more than an order of magnitude smaller error in the three-dimensional case.
\begin{figure}[t]
	\centering
	\captionsetup{belowskip=0pt, aboveskip=3pt}
	\begin{subfigure}{0.5\textwidth}
		\centering
		\includegraphics[scale=0.5]{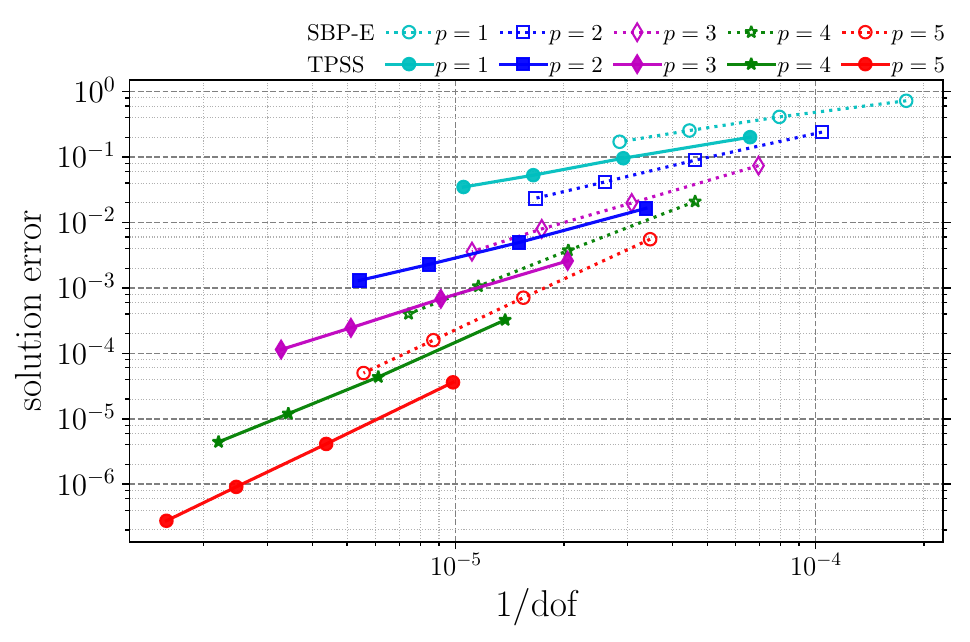}
	\end{subfigure}\hfill
	\begin{subfigure}{0.5\textwidth}
		\centering
		\includegraphics[scale=0.5]{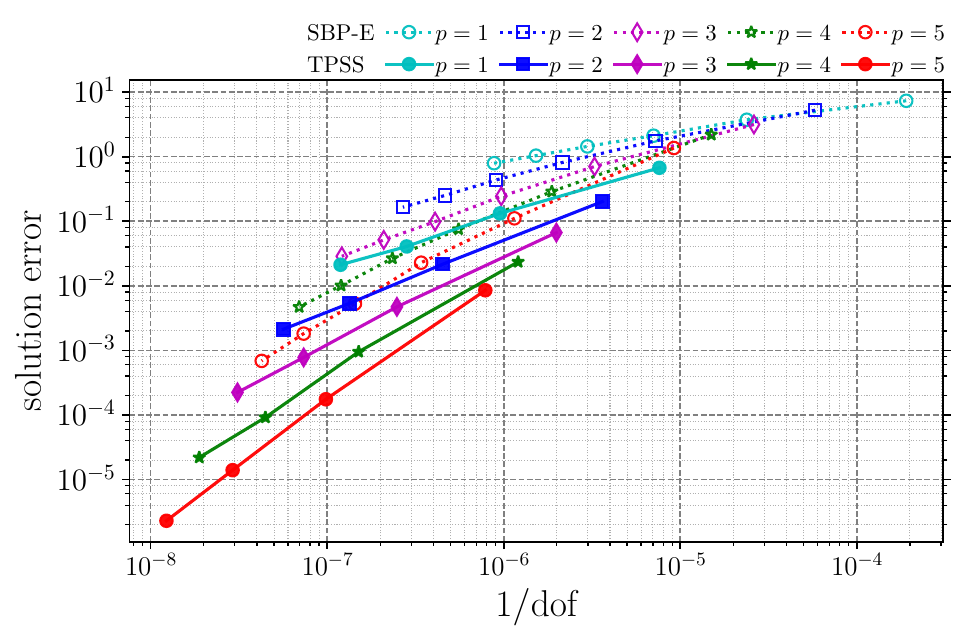}
	\end{subfigure} 
	\caption{\label{fig:ivp acc} Grid convergence study of the $\H$-norm solution error for the 2D (left) and 3D (right) isentropic vortex problems.}
\end{figure}
Furthermore, \cref{fig:ivp res time} shows a large efficiency gap between the two types of operators in terms of computational runtime and accuracy after a single time step (with a step size of $10^{-3}$). On triangles, the TPSS operators require at least 4 times less computational time to achieve the same error levels as the dense SBP diagonal-$\E$ operators, while on the tetrahedron, they require about 20 times less computational time. For a fixed computational time, the TPSS operators yield about one and two orders of magnitude lower error in two and three dimensions, respectively. \ignore{It is interesting that some of the TPSS operators with degrees less than $p$ outperform the degree $p$ SBP diagonal-$\E$ operators in both two and three dimensions, although this may not hold on sufficiently fine meshes due to the difference in convergence rates.}
\begin{figure}[t]
	\centering
	\captionsetup{belowskip=0pt, aboveskip=4pt}
	\begin{subfigure}{0.5\textwidth}
		\centering
		\includegraphics[scale=0.5]{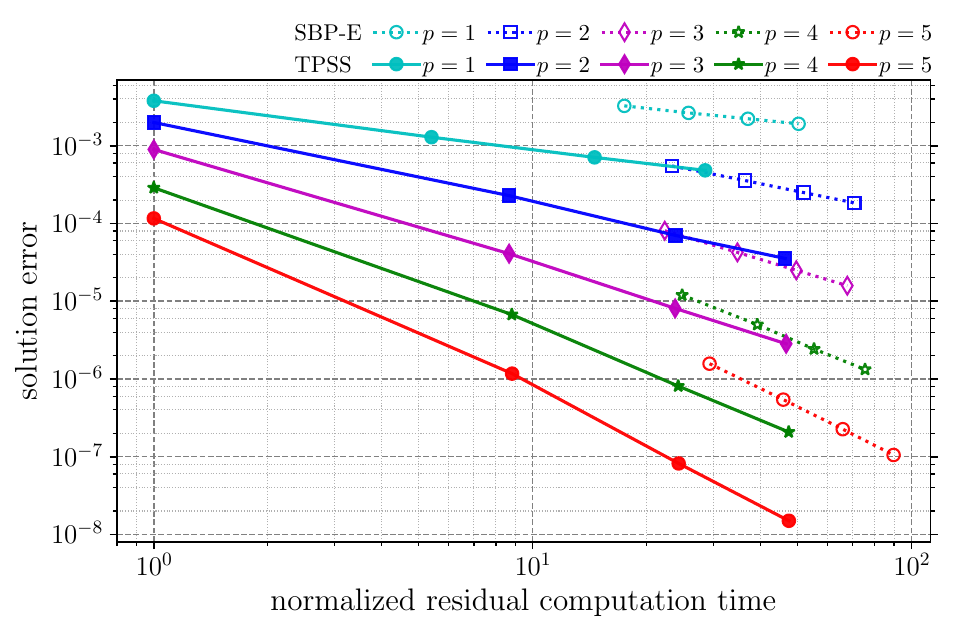}
	\end{subfigure}\hfill
	\begin{subfigure}{0.5\textwidth}
		\centering
		\includegraphics[scale=0.5]{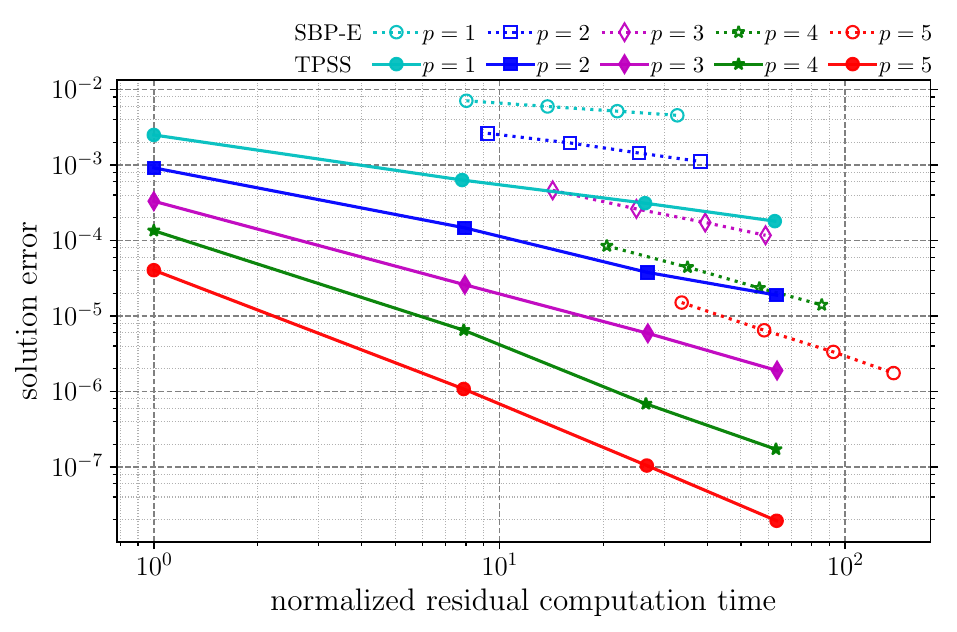}
	\end{subfigure} 
	\caption{\label{fig:ivp res time} Normalized computational runtime versus the $\H$-norm solution error for the 2D (left) and 3D (right) isentropic vortex problem.}
\end{figure}
\ignore{
	The $\H$-norm is an approximation to the $L^2$-norm, and it is possible that some types of SBP operators have $\H$ matrices that approximate the $L^2$-norm better than others. This can affect the solution error values. Hence, for a sanity check, the $L^{\infty}$-norm of the solution errors with the TPSS and SBP diagonal-$\E$ operators is computed after one time step. The results in \cref{fig:ivp res time linf} show that the solution errors in the $L^{\infty}$-norm are close to the values obtained with the $\H$-norm, suggesting that the integration error due to the use of the $\H$-norm is negligible. 
	
	\begin{figure}[t]
		\centering
		\captionsetup{belowskip=0pt, aboveskip=4pt}
		\begin{subfigure}{0.5\textwidth}
			\centering
			\includegraphics[scale=0.5]{2d_total_time_SBP_E_vs_tss_ivp_err_vs_time_linf.pdf}
		\end{subfigure}\hfill
		\begin{subfigure}{0.5\textwidth}
			\centering
			\includegraphics[scale=0.5]{3d_total_time_SBP_E_vs_tss_ivp_err_vs_time_linf.pdf}
		\end{subfigure} 
		\caption{\label{fig:ivp res time linf} Computational runtime versus the $L^{\infty}$-norm solution error for the 2D (left) and 3D (right) isentropic vortex problem.}
	\end{figure}
}

Finally, to verify that the TPSS operators are constructed such that the SBP property is satisfied and thus lead to entropy-conservative discretizations, we run the isentropic vortex problem using the TPSS operators up to $t=20$ on coarse meshes with $18$ triangular elements in two dimensions and $162$ tetrahedral elements in three dimensions. For the 2D case, the change in entropy is computed as $\Delta s_t = s_t - s_0$, 
where $ s_t $ and $ s_0 $ are the integrals of the mathematical entropy, $ \fn{S}(\bm{x},t) = -\rho \ln(p\rho^{-\gamma})/(\gamma -1)$, at time $ t $ and the initial time $t=0$, respectively. In the 3D case, $\Delta s_t$ is normalized by $s_0$, as the entropy is nonzero. The integral of the mathematical entropy is approximated as $ s = \sum_{\Omega_{k}\in\fn{T}_h}\bm{1}^T\H_{k}\bm{s}_k $. The change in the integral of the entropy is shown in \cref{fig:ivp entropy}, where it can be seen that the entropy is conserved to machine precision.
\begin{figure}[t]
	\centering
	\captionsetup{belowskip=0pt, aboveskip=3pt}
	\begin{subfigure}{0.5\textwidth}
		\centering
		\includegraphics[scale=0.36]{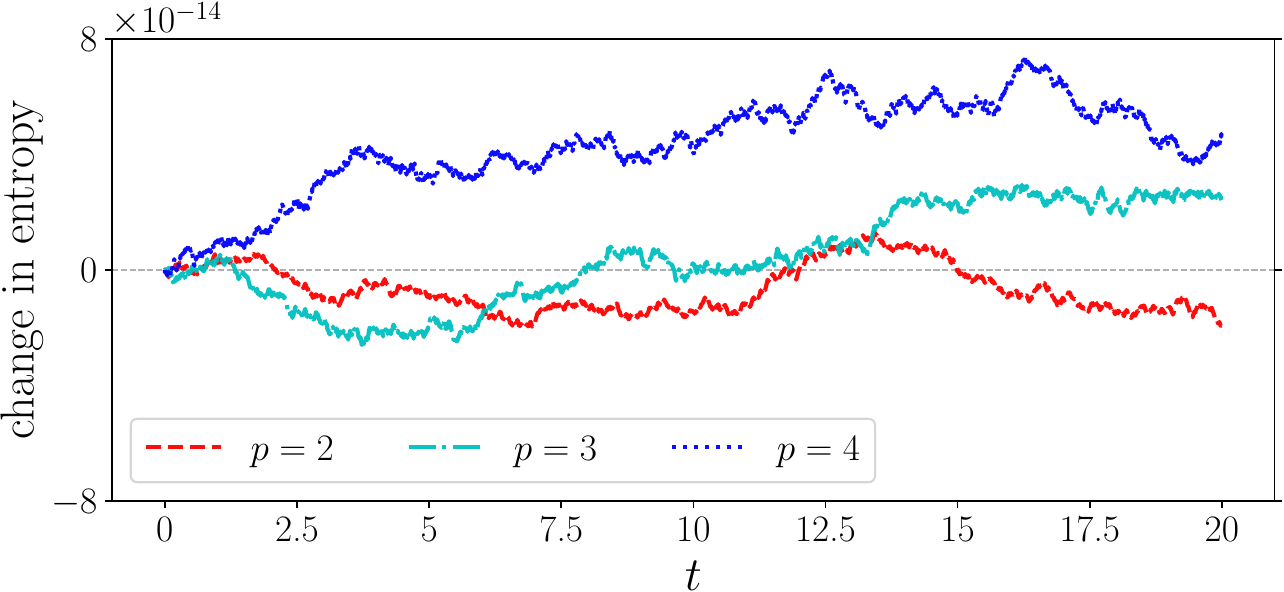}
	\end{subfigure}\hfill
	\begin{subfigure}{0.5\textwidth}
		\centering
		\includegraphics[scale=0.36]{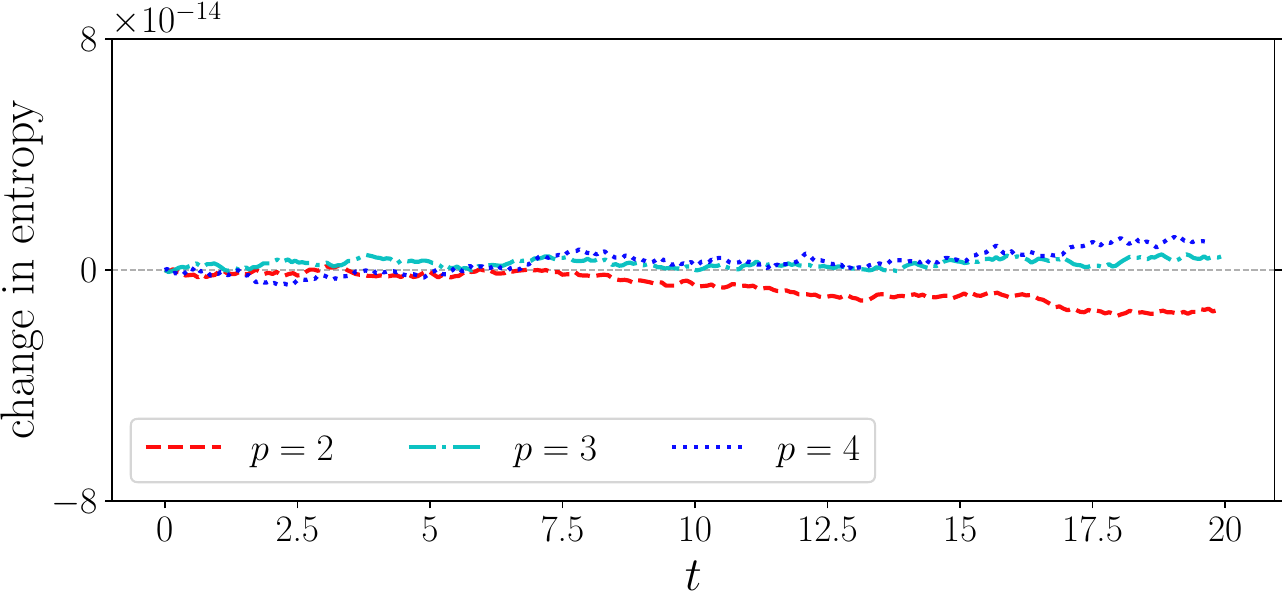}
	\end{subfigure} 
	\caption{\label{fig:ivp entropy} Entropy conservation of the discretization of the 2D (left) and 3D (right) isentropic vortex problem with the TPSS operators.}
\end{figure}

\subsection{Viscous Taylor--Green vortex problem} 
The viscous Taylor--Green vortex problem is governed by the Navier--Stokes equations and is defined on the domain $ \Omega = \left[-\pi,\pi\right]^{3} $ with the initial conditions \cite{debonis2013solutions}
\begin{alignat*}{3}
	{V_{1}}&=u_{0}\sin\left(\frac{{x_{1}}}{L}\right)\cos\left(\frac{{x_{2}}}{L}\right)\cos\left(\frac{{x_{3}}}{L}\right)\!, \quad p\!=\!p_{0}+\frac{\rho_{0}u_{0}^{2}}{16}\left[\cos\left(\frac{2{x_{1}}}{L}\right)+\cos\left(\frac{2{x_{2}}}{L}\right)\right]\left[\cos\left(\frac{2{x_{3}}}{L}\right)+2\right]\!\!,\\
	{V_{2}}&=-u_{0}\cos\left(\frac{{x_{1}}}{L}\right)\sin\left(\frac{{x_{2}}}{L}\right)\cos\left(\frac{{x_{3}}}{L}\right)\!, \quad {V_{3}}=0, \quad  \rho=\frac{p}{RT_{0}},
\end{alignat*}
where $ u_{0} =  \rho_{0}= R = L = 1$, $ T_{0} = p_{0}/(R\rho_{0}) $, $ M=u_{0}/\sqrt{\gamma R T_{0}}=0.1 $ is the Mach number, and $ p_{0} = \rho_{0} u_{0}^{2}/(\gamma M^2)$. The Reynolds number is $ Re=\rho_{0} u_{0}L/\mu_{0}=1600$, where $ \mu_{0} $ is a constant viscosity, and the Prandtl number is set to $ Pr=0.71 $. The scaled kinetic energy is computed as
\begin{equation}
	E_{k}=\frac{1}{2\left|\Omega\right|}\int_{\Omega}\rho{\bm{V}\cdot \bm{V}}\dd{\Omega},
\end{equation}
where $ \left|\Omega\right|=8\pi^{3} $, $\bm{V}=[V_1,V_2,V_3]^T$ is the velocity vector, and the kinetic energy dissipation rate is given by 
\begin{equation}
	\varepsilon=-\der[E_{k}]t,
\end{equation}
which is computed using the fifth-order CSBP derivative operator, \ie, the derivative operator is applied to the column vector containing the values of the kinetic energy, $E_{k}$, at constant time steps from $t=0$ up to $t=20$. The enstrophy is given by 
\begin{equation}
	\zeta=\frac{1}{2\left|\Omega\right|}\int_{\Omega}\rho{\bm{\omega}\cdot \bm{\omega}}\dd{\Omega},
\end{equation}
where $\bm{\omega}=\nabla\times\bm{V}$ is the vorticity vector. 

The TPSS operators are applied to the Taylor--Green vortex problem to further verify their accuracy and stability. Operators with polynomial degrees of 3, 4, and 5 and meshes with numbers of degrees of freedom approximately equal to $128^3$ and $164^3$ are used for this study. The ratios of the actual to the nominal number of degrees of freedom for each case are presented in \cref{tab:tgv dof}. All the results for the kinetic energy, dissipation rate, and enstrophy with the TPSS and dense SBP diagonal-$\E$ operators of varying degrees and degrees of freedom are provided in \cref{sec:tgv figures}. In general, the results show good agreement with the DNS results of Dairay \etal, \cite{dairay2017numerical}. 
\begin{table*}[t]
	\footnotesize
	\centering
	\captionsetup{skip=3pt}
	\caption{\label{tab:tgv dof} Ratios of the actual to nominal number of degrees of freedom for the Taylor--Green vortex problem.}
	\begin{threeparttable}
		\setlength{\tabcolsep}{0.4em}
		\renewcommand*{\arraystretch}{1.2}
		\begin{tabular}{r@{\hspace{1.5em}}ll @{\hspace{1.5em}}ll @{\hspace{1.5em}}ll}
			\toprule
			& \multicolumn{2}{l}{$p=3$} &  \multicolumn{2}{l}{$p=4$} &  \multicolumn{2}{l}{$p=5$}   \\
			\cmidrule(l{0em}r{1em}){2-3} \cmidrule(l{0em}r{1em}){4-5} \cmidrule(l{0em}r{1em}){6-7}  
			Nominal DOF &$128^3$&$164^3$& $128^3$&$164^3$& $128^3$&$164^3$\\
			\midrule 
			SBP-$\E$ & 1.0008& 1.0840 &1.0433 &1.1212 & 1.1383& 1.1503 \\
			TPSS &0.9829 & 0.9127& 1.0844&  1.0957& 1.0475& 1.1805 \\
			\bottomrule
		\end{tabular}
	\end{threeparttable}
\end{table*}

Compared to the dense SBP diagonal-$\E$ operators, the TPSS operators yield substantially more accurate results with lower operator degrees and on meshes with fewer degrees of freedom. For instance, \cref{fig:tgv p345 zoomed} shows that close to the starting time (before significant error accumulation) the degree $p=3$ TPSS operator with $128^3$ degrees of freedom yields more accurate results than the degree $p=5$ dense SBP diagonal-$\E$ operator with $164^3$ degrees of freedom. Furthermore, \cref{fig:tgv equal dof} shows that for an equal number of degrees of freedom, the degree $p=3$ TPSS operator yields a more accurate approximation of the kinetic energy compared to the degrees $p=4$ and $p=5$ dense SBP diagonal-$\E$ operators. Finally, as shown in \cref{fig:tgv p3344}, the TPSS operator yields more accurate values of the enstrophy with about half the degrees of freedom as dense SBP diagonal-$\E$ operators of the same degree. Additional evidence of the superior performance of the TPSS operators for the Taylor--Green vortex problem can be inferred from the results presented in \cref{fig:tgv} in \cref{sec:tgv figures}.

\begin{figure}[t]
	\centering
	\captionsetup{belowskip=0pt, aboveskip=4pt}
	\begin{subfigure}{0.5\textwidth}
		\centering
		\includegraphics[scale=0.45]{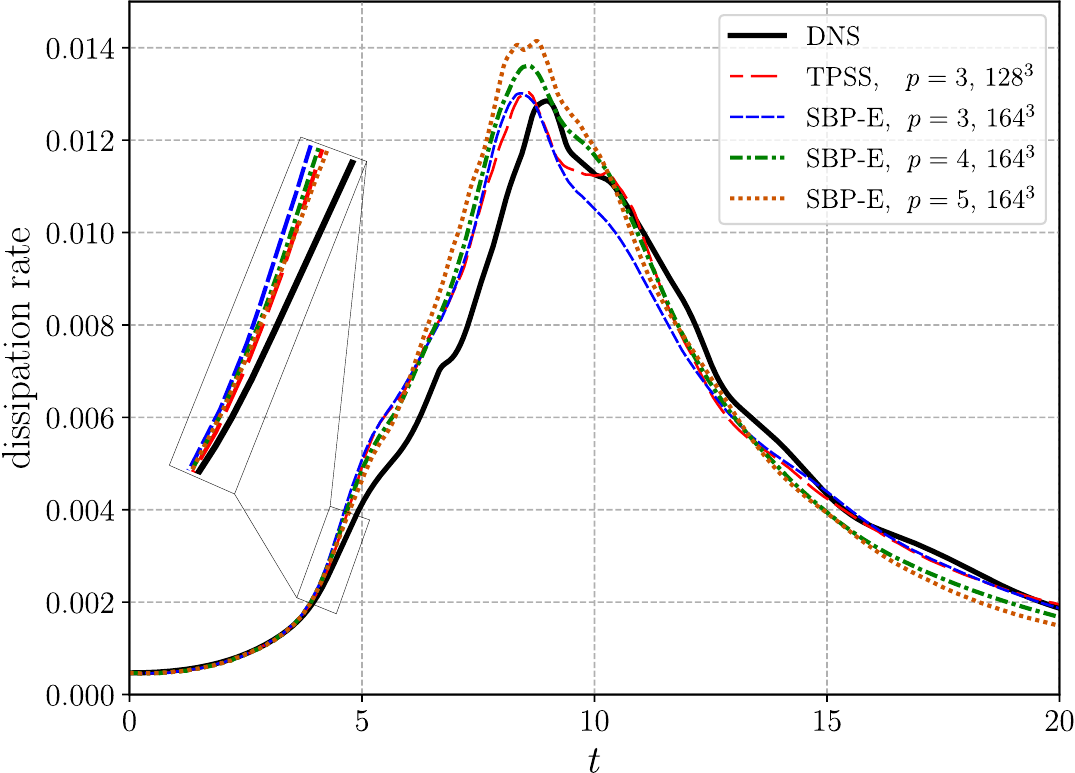}
	\end{subfigure}\hfill
	\begin{subfigure}{0.5\textwidth}
		\centering
		\includegraphics[scale=0.45]{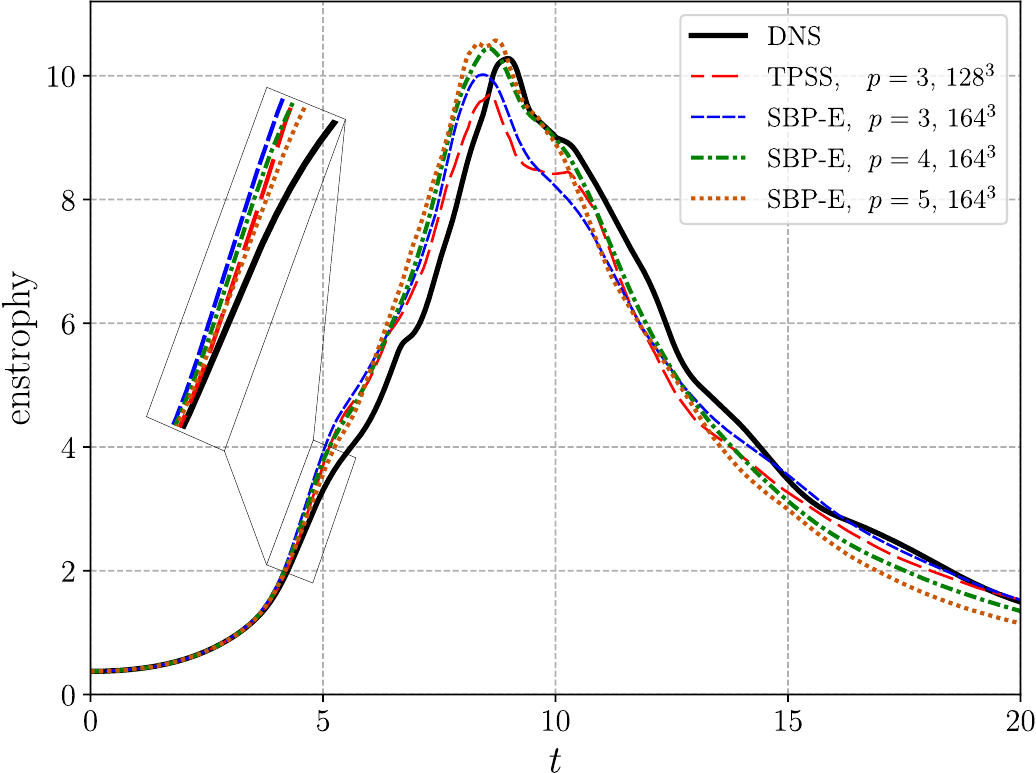}
	\end{subfigure} 
	\caption{\label{fig:tgv p345 zoomed} Comparison of accuracy of the TPSS and dense SBP diagonal-$\E$ operators for Taylor--Green vortex problem with varying number of degrees of freedom and polynomial degree. Line thickness is reduced in the close-up for better illustration.}
\end{figure}

\begin{figure}[t]
	\centering
	\captionsetup{belowskip=0pt, aboveskip=4pt}
	\begin{subfigure}{0.5\textwidth}
		\centering
		\includegraphics[scale=0.45]{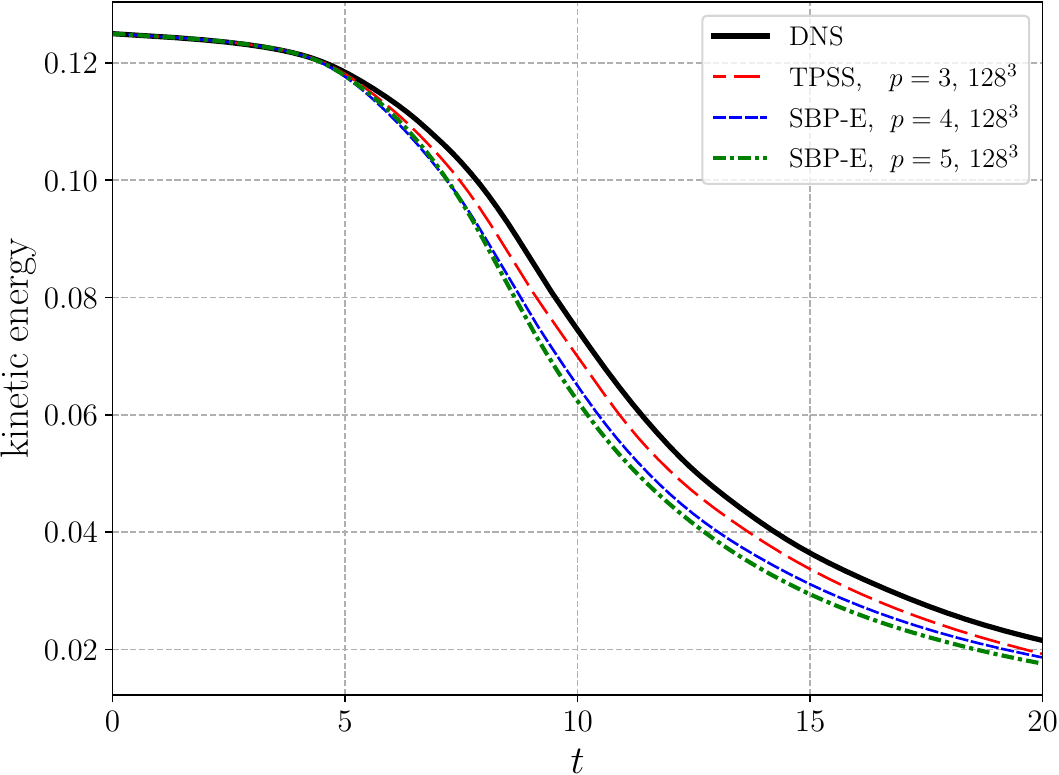}
	\end{subfigure}\hfill
	\begin{subfigure}{0.5\textwidth}
		\centering
		\includegraphics[scale=0.45]{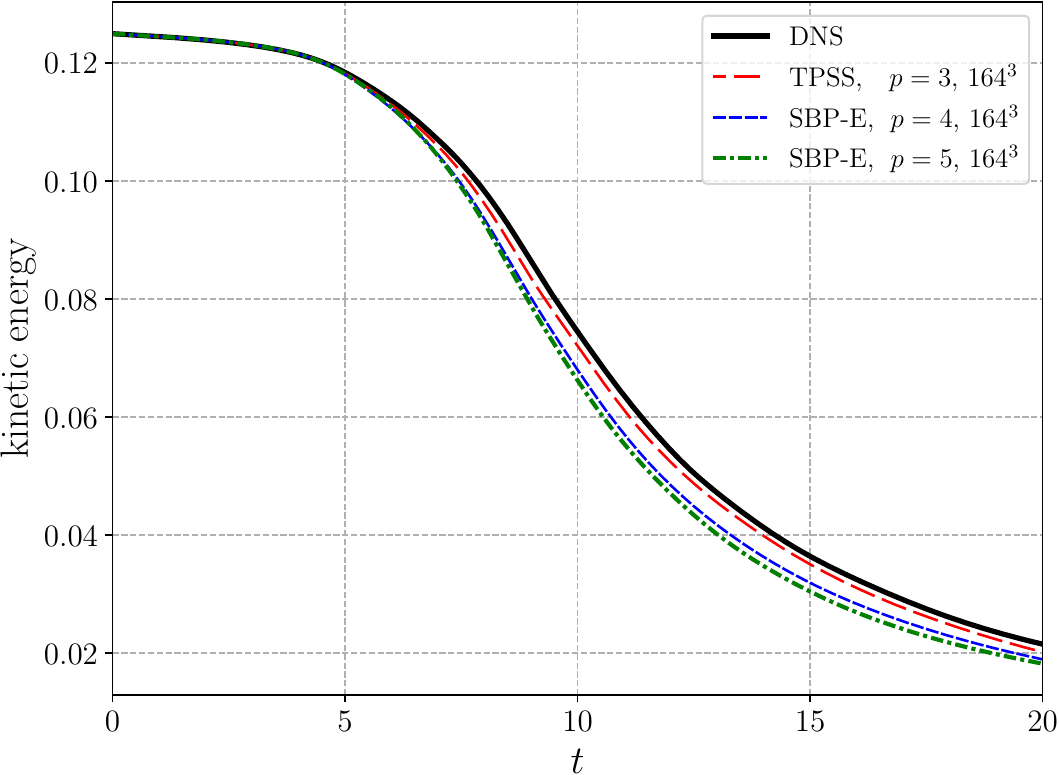}
	\end{subfigure} 
	\caption{\label{fig:tgv equal dof} Comparison of accuracy of the TPSS and dense SBP diagonal-$\E$ operators for Taylor--Green vortex problem for equal degrees of freedom and varying polynomial degrees.}
\end{figure}

\begin{figure}[t]
	\centering
	\captionsetup{belowskip=0pt, aboveskip=4pt}
	\begin{subfigure}{0.5\textwidth}
		\centering
		\includegraphics[scale=0.45]{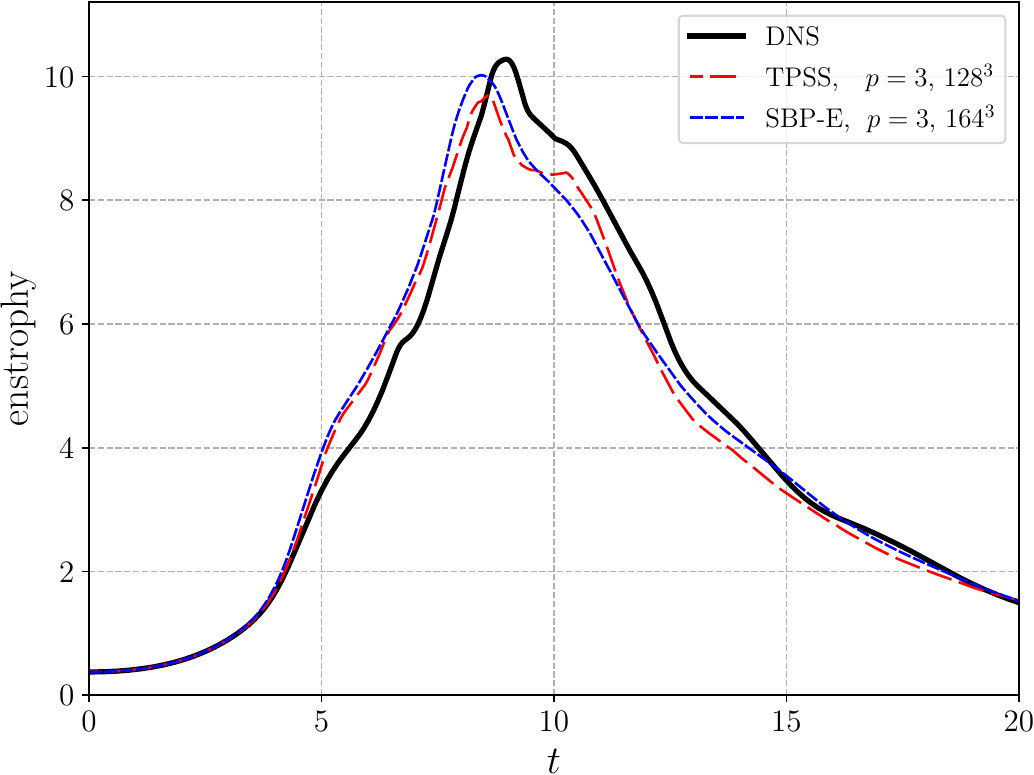}
	\end{subfigure}\hfill
	\begin{subfigure}{0.5\textwidth}
		\centering
		\includegraphics[scale=0.45]{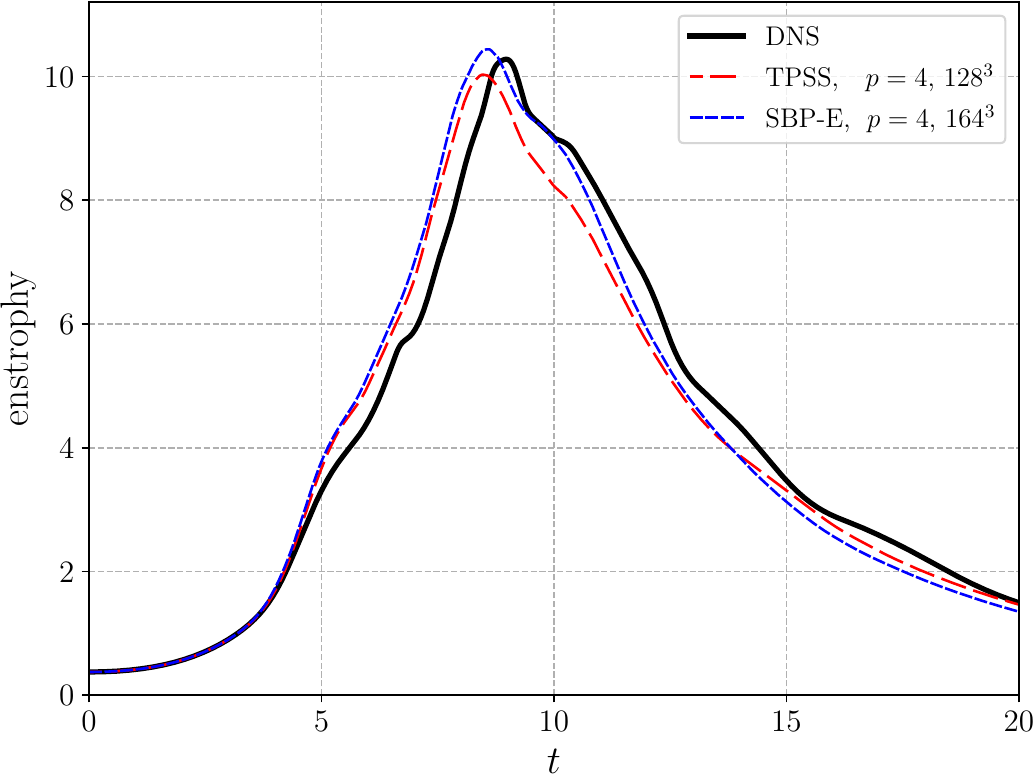}
	\end{subfigure} 
	\caption{\label{fig:tgv p3344} Comparison of accuracy of the TPSS and dense SBP diagonal-$\E$ operators for Taylor--Green vortex problem at equal polynomial degrees and varying number of degrees of freedom.}
\end{figure}

\section{Conclusions}\label{sec:conclusions}
We have developed novel tensor-product split-simplex multidimensional SBP operators on triangles and tetrahedra. The operators are constructed by splitting the simplicial elements into quadrilateral or hexahedral subdomains, mapping tensor-product operators into each subdomain, and assembling back the resulting operators. The procedure of splitting simplicial elements and applying tensor-product operators on the split subdomains has been used in the literature, but the TPSS operators are distinct for the following reasons:
\begin{itemize}[label=\textbullet,topsep=0pt, partopsep=0pt, itemsep=0pt, parsep=0pt]
	\item There are no repeated degrees of freedom inside the TPSS reference simplex element, and no coupling via numerical fluxes or SATs is required inside the simplices. Furthermore, the split subdomains are assembled using a method that preserves their high-order accuracy. 
	\item The TPSS operators are discretely stable by construction, \ie, the mapping from the quadrilateral or hexahedron reference element to the split quadrilateral or hexahedral subdomains in the reference triangle is handled in such a way that the SBP property is satisfied. 
	\item The simplex mesh in the computational domain does not need to be split. The operators on the reference elements are mapped to the physical elements in the same manner as any other multidimensional SBP operator; hence, the TPSS operators are truly multidimensional SBP operators with tensor-product structures. It is, however, possible to split and assemble TPSS operators on the physical mesh itself, but this approach has not be pursued in this work. 
\end{itemize}

We have established the accuracy of the TPSS operators theoretically, and demonstrated their efficiency through a number of numerical experiments. Compared to the existing dense SBP diagonal-$\E$ operators, the TPSS operators are substantially more accurate for a given number of degrees of freedom and require considerably less computational time to achieve the same error threshold. For the advection equation, the TPSS operators produce solutions that are 1 to 2 orders of magnitude more accurate for a given number of degrees of freedom. Similarly, for the three-dimensional isentropic vortex problem, they produce solutions that are greater than an order of magnitude more accurate for a fixed number of degrees of freedom. Furthermore, computational runtime comparisons on tetrahedral meshes show that the TPSS operators require about 20 times less computational time to achieve a given error threshold. For a fixed computational time, the TPSS operators produce errors that are about two orders of magnitude smaller than those of the dense SBP diagonal-$\E$ operators on tetrahedral meshes. For the viscous Taylor--Green vortex problem, TPSS operators with lower polynomial degrees and fewer degrees of freedom produce more accurate results than the dense SBP diagonal-$\E$ operators. Preliminary results show that on distorted meshes, the relative advantages of the TPSS operators in terms of accuracy and time-step restrictions decrease to some extent, but this trend does not appear to worsen substantially with severity of the distortion. In general, the TPSS operators still offer substantially more efficient discretizations compared to the existing dense SBP diagonal-$\E$ operators on distorted meshes. Further research is required to investigate the efficiency of the novel entropy-stable TPSS operators presented relative to that of entropy-stable schemes exploiting collapsed coordinates to achieve a tensor-product structure on simplices \cite{montoya2023efficient,montoya2024efficient} for problems of practical interest.

\section*{Declaration of competing interest}
The authors declare that they have no known competing financial interests or personal relationships that could have appeared to influence the work reported in this paper.

\section*{Acknowledgments}
The authors would like to thank Professor Masayuki Yano and his Aerospace Computational Engineering Lab at the University of Toronto for the use of their software, the Automated PDE Solver (APS). The first author would also like to thank Alex Bercik for the insightful discussions on the accuracy of high-order summation-by-parts operators under non-affine mappings. Computation were performed on the Niagara supercomputer at the SciNet HPC Consortium \cite{ponce2019deploying}. SciNet is funded by: the Canada Foundation for Innovation; the Government of Ontario; Ontario Research Fund - Research Excellence; and the University of Toronto.

\addcontentsline{toc}{section}{Acknowledgments}

\appendix
\newpage
\section*{Appendices}
\section{Grid convergence results for the advection problem}\label{sec:grid conv adv}
	\begin{table*} [h!]
	\footnotesize
	\caption{\label{tab:grid conv adv} Grid convergence study for the 2D and 3D advection problem with upwind SATs. The $\H$-norm errors are computed at $t=1$ for both the 2D and 3D cases.}
	\centering
	\setlength{\tabcolsep}{0.68em}
	\renewcommand*{\arraystretch}{1.2}
	\begin{tabular}{cllllllllllll}
		\toprule
		\multirow{3}{*}{$p$}&  \multicolumn{5}{c}{2D}&    \multicolumn{5}{c}{3D} \\ 
		\cmidrule(lr){2-6} \cmidrule(lr){7-11}
		&  \multirow{2}{*}{$ n_e $ }   & \multicolumn{2}{c}{SBP-$\E$} &  \multicolumn{2}{c}{TPSS} &   \multirow{2}{*}{$ n_e $ }   &    \multicolumn{2}{c}{SBP-$\E$} &  \multicolumn{2}{c}{TPSS} \\
		\cmidrule(lr){3-4} \cmidrule(lr){5-6} \cmidrule(lr){8-9} \cmidrule(lr){10-11}
		& & {$\H$-norm}  &{rate} & {$\H$-norm}  & {rate} &  & {$\H$-norm}  & {rate} &{$\H$-norm}  & {rate} \\
		\midrule
		\multirow[t]{4}{*}{$ 1 $}&     $ 35^2\times 2 $  	& $ 1.5886e-01 $ &     --    &   $ 2.9708e-03 $ &     --     & $ 5^3\times 6 $    &$  1.1687e-01 $  &  --     & $  1.8802e-03 $  &     --     \\  
		&   $ 46^2\times 2 $   &$ 8.4921e-02 $ &$ 2.29 $  & $ 1.5749e-03$ &$  2.32 $ & $ 10^3\times 6 $ & $  2.2786e-02$ & $ 2.36 $   & $ 2.8595e-04$ & $ 2.72 $     \\  
		&   $ 57^2\times 2 $   &$  4.8928e-02$ & $ 2.57 $   &$ 9.8235e-04$ & $ 2.20$ &  $ 15^3\times 6 $& $ 8.7703e-03$ &$  2.35 $  &  $ 1.2434e-04$ & $ 2.05  $   \\  
		&   $ 68^2\times 2 $  &$  3.0418e-02$ & $ 2.69 $  & $ 6.7403e-04$ &$  2.13$ &  $ 20^3\times 6 $& $ 4.6681e-03$ & $ 2.19 $  &  $ 6.9706e-05$ & $ 2.01  $    \\  
		\addlinespace
		\multirow[t]{4}{*}{$ 2 $}&     $ 30^2\times 2 $   &  $ 8.6415e-03$ &     --    & $  2.8152e-04$   &     --     & $ 5^3\times 6 $ & $ 4.3144e-02 $  &  --     &   $ 1.5247e-04 $ &     --     \\ 
		&    $ 40^2\times 2 $         & $ 2.9488e-03$ & $ 3.74 $  & $ 1.1773e-04$ & $ 3.03$ & $ 10^3\times 6 $ & $ 6.4558e-03 $ & $ 2.74 $ &  $ 1.4808e-05$ & $ 3.36$    \\  
		&    $ 50^2\times 2 $        &  $ 1.3883e-03$ &$  3.38 $     & $ 6.0029e-05$ & $ 3.02 $ & $15^3\times 6 $ & $ 2.1562e-03$  &$ 2.70  $  &  $ 3.9191e-06$ & $ 3.28 $      \\  
		&    $ 60^2\times 2 $       & $ 7.7438e-04$ & $ 3.20 $   & $ 3.4671e-05$ & $ 3.01 $ & $20^3\times 6 $ & $ 9.7182e-04$ &$ 2.77 $   &  $  1.5288e-06 $ &$ 3.27 $     \\  
		\addlinespace
		\multirow[t]{4}{*}{$ 3 $}&     $ 25^2\times 2 $  &  $ 1.4643e-03$ &     --    &$  2.7340e-05 $   &     --     & $ 5^3\times 6 $ &  $ 4.4746e-03$  &  --     &  $ 1.0791e-05 $   &     --     \\ 
		&    $ 34^2\times 2 $         &  $ 4.3908e-04$ &$  3.92 $  &$  7.9314e-06$ & $ 4.02 $ &  $ 10^3\times 6 $& $ 3.9061e-04$ & $ 3.52$   &  $ 5.0775e-07$ & $ 4.41 $     \\  
		&    $ 43^2\times 2 $        &$ 1.7419e-04$  &$  3.94 $   & $ 3.0910e-06$ &$  4.01$ & $ 15^3\times 6 $ & $ 9.0719e-05$ & $ 3.60 $ &  $ 1.1371e-07$ & $ 3.69 $  \\  
		&    $ 52^2\times 2 $        & $ 8.2126e-05$ & $ 3.96 $    & $ 1.4421e-06 $ &$ 4.01 $ & $ 20^3\times 6 $ & $ 3.0900e-05$ & $ 3.74    $ & $ 4.0697e-08$ & $ 3.57 $   \\  
		\addlinespace
		\multirow[t]{4}{*}{$ 4 $}&     $ 20^2\times 2 $ & $ 4.5390e-04$  &     --    &  $ 3.7776e-06$   &     --     &$ 5^3\times 6 $  &  $ 8.1064e-04$  &  --     &  $ 8.4527e-07$   &     --     \\ 
		&    $ 28^2\times 2 $        &$ 9.9500e-05$ & $ 4.51 $   & $ 6.5562e-07$ & $ 5.20  $  & $ 10^3\times 6 $ &$ 4.3935e-05$  &$ 4.21  $ & $  1.8421e-08$ & $ 5.52 $    \\  
		&    $ 36^2\times 2 $       & $ 3.2267e-05$ & $ 4.48   $ &$   1.9468e-07$ & $ 4.83 $  & $15^3\times 6 $ &  $ 6.3243e-06$ & $ 4.78  $  &  $  2.2792e-09$ &$  5.15   $  \\  
		&    $ 44^2\times 2 $        & $ 1.3100e-05$ & $ 4.49  $   &$  7.0252e-08$ & $ 5.08 $ & $ 20^3\times 6 $ &  $ 1.6100e-06$  &$ 4.76   $  &    $ 5.2483e-10$ & $ 5.10$     \\  
		\addlinespace
		\multirow[t]{4}{*}{$ 5$}&     $ 15^2\times 2 $   &  $ 2.7900e-04$ &     --    &  $ 9.3948e-07 $  &     --     & $ 5^3\times 6 $ & $ 1.4885e-04 $  &  --     &  $ 3.5014e-08$  &     --     \\ 
		&    $ 22^2\times 2 $       & $ 3.5564e-05$ & $ 5.34  $   &$  1.0399e-07 $ & $ 5.75 $ &  $ 10^3\times 6 $& $  3.9943e-06$ & $ 5.22$  & $ 5.3716e-10$ & $ 6.03  $ \\  
		&    $ 29^2\times 2 $        & $ 8.5520e-06$ &$  5.16 $  & $ 1.7328e-08 $ &$  6.49 $  & $ 15^3\times 6 $ & $ 4.8853e-07$ &$ 5.18 $  &  $ 5.7630e-11$ &$  5.51$     \\  
		&    $ 36^2\times 2 $        & $ 2.8587e-06$ & $ 5.07  $    & $ 5.1844e-09 $ & $ 5.58$   & $ 20^3\times 6 $ & $  1.1049e-07$ &$  5.17 $  &  $ 1.3352e-11 $  &$  5.08$ \\  
		
		\bottomrule
	\end{tabular}
\end{table*}
\newpage
\section{Grid convergence results for the isentropic vortex problem}\label{sec:grid conv ivp}
\begin{table*} [h!]
	\footnotesize
	\caption{\label{tab:grid conv ivp} Grid convergence study for the 2D and 3D advection problem discretized using the Hadamard-form scheme with the Ismail-Roe two-point fluxes and no interface dissipation. The $\H$-norm errors are computed at $t=2$ and $t=1$ for the 2D and 3D cases, respectively.}
	\centering
	\setlength{\tabcolsep}{0.68em}
	\renewcommand*{\arraystretch}{1.2}
	\begin{tabular}{cllllllllll}
		\toprule
		\multirow{3}{*}{$p$}&  \multicolumn{5}{c}{2D}&    \multicolumn{5}{c}{3D} \\ 
		\cmidrule(lr){2-6} \cmidrule(lr){7-11}
		&  \multirow{2}{*}{$ n_e $ }   & \multicolumn{2}{c}{SBP-$\E$} &  \multicolumn{2}{c}{TPSS} &   \multirow{2}{*}{$ n_e $ }   &    \multicolumn{2}{c}{SBP-$\E$} &  \multicolumn{2}{c}{TPSS} \\
		\cmidrule(lr){3-4} \cmidrule(lr){5-6} \cmidrule(lr){8-9} \cmidrule(lr){10-11}
		& & {$\H$-norm}  &{rate} & {$\H$-norm}  & {rate} &  & {$\H$-norm}  & {rate} &{$\H$-norm}  & {rate} \\
		\midrule
		\multirow[t]{4}{*}{$ 1 $}&     $ 20^2\times 2 $  	& $ 7.2383e-01$  &     --    &  $ 2.0116e-01$   &     --     & $ 5^3\times 6 $   &  $ 7.3473e+00$  &  --     &   $ 6.7330e-01$  &     --     \\  
		&   $ 30^2\times 2 $        & $ 4.1073e-01$  & $ 1.40$    &   $ 9.5821e-02$ & $ 1.83$     &  $ 10^3\times 6 $&  $ 3.7359e+00$ & $ 0.98$    & $ 1.3236e-01$ & $ 2.35$      \\  
		&    $ 40^2\times 2 $        & $ 2.5426e-01$ & $ 1.67$      & $ 5.2785e-02$  &$ 2.07$     &  $ 15^3\times 6 $& $ 2.1157e+00$ & $ 1.40$   & $ 4.0707e-02$ & $ 2.91$      \\  
		&    $ 50^2\times 2 $        & $ 1.7114e-01 $ & $ 1.77$      &   $ 3.4825e-02$ &$ 1.86 $     & $ 20^3\times 6 $ & $ 1.4422e+00$ & $ 1.33$   & $  2.1197e-02$ & $ 2.27$    \\  
		&       								 &  &     &    &      & $ 25^3\times 6 $ &  $ 1.0380e+00$  & $ 1.47$   &    &      \\  
		&       								 &  &     &    &      & $ 30^3\times 6 $ &  $ 7.9249e-01$  &  $ 1.48 $  &    &      \\  
		\addlinespace
		\multirow[t]{4}{*}{$ 2 $}&     $ 20^2\times 2 $   & $ 2.4048e-01 $ &     --    & $ 1.6436e-02$   &     --     &  $ 5^3\times 6 $& $ 5.2483e+00$   &  --     &  $ 2.0157e-01 $  &     --     \\ 
		&    $ 30^2\times 2 $        & $ 8.9582e-02$ & $ 2.44$  & $ 4.9432e-03$ & $ 2.96$      & $ 10^3\times 6 $ & $  1.7563e+00$  & $ 1.58 $   &  $ 2.1648e-02$ &$  3.22 $  \\  
		&    $ 40^2\times 2 $        & $ 4.1745e-02$ & $ 2.65 $  & $  2.2873e-03$ &$  2.68 $     & $ 15^3\times 6 $ & $  8.1358e-01$ & $ 1.90 $    & $ 5.3565e-03$ & $ 3.44 $    \\  
		&    $ 50^2\times 2 $       &  $ 2.3406e-02$ &$ 2.59$    &  $ 1.2988e-03$ & $ 2.54  $     &  $ 20^3\times 6 $&  $ 4.3663e-01$ &$ 2.16$  &   $ 2.1267e-03$ & $ 3.21 $    \\  
		&       								 &  &     &    &      & $ 25^3\times 6 $ &  $ 2.5071e-01 $ & $ 2.49 $  &    &      \\  
		&       								 &  &     &    &      &  $ 30^3\times 6 $& $  1.6515e-01 $ & $ 2.29  $    &    &      \\  
		\addlinespace
		\multirow[t]{4}{*}{$ 3 $}&     $ 20^2\times 2 $   & $  7.3899e-02$ &     --    &   $ 2.5877e-03 $ &     --     &  $ 5^3\times 6 $&  $ 3.1482e+00$  &  --     &  $ 6.6803e-02$   &     --     \\ 
		&    $ 30^2\times 2 $         &$ 1.9839e-02$ & $ 3.24 $ &  $ 6.8097e-04$ & $ 3.29 $  &  $ 10^3\times 6 $&  $ 7.0380e-01$ & $ 2.16  $ &   $  4.7270e-03$ & $ 3.82$      \\  
		&    $ 40^2\times 2 $        &$ 7.9698e-03$ &$  3.17$    &  $ 2.4565e-04$ & $ 3.54$     & $ 15^3\times 6 $ & $ 2.4374e-01$ & $ 2.62  $  &  $  7.7821e-04$ & $ 4.45$      \\  
		&   $ 50^2\times 2 $        & $ 3.5824e-03$ &$  3.58$    & $  1.1425e-04$ & $ 3.43 $     & $ 20^3\times 6 $ & $ 9.8728e-02$ & $ 3.14  $   & $   2.2369e-04$ &$  4.33 $     \\  
		&       								 &  &     &    &      & $ 25^3\times 6 $ &   $ 5.1393e-02$  &  $ 2.93$    &    &      \\  
		&       								 &  &     &    &      & $ 30^3\times 6 $ &   $ 2.8662e-02$  & $ 3.20  $  &    &      \\  
		\addlinespace
		\multirow[t]{4}{*}{$ 4 $}&     $ 20^2\times 2 $    &$ 2.0905e-02 $  &     --    &  $ 3.2374e-04 $  &     --     & $ 5^3\times 6 $ &  $ 2.1926e+00$  &  --     &  $ 2.3497e-02$   &     --     \\ 
		&    $ 30^2\times 2 $        & $ 3.7355e-03$ & $ 4.25 $   &  $ 4.3499e-05$ & $ 4.95$    & $ 10^3\times 6 $ &  $ 2.8815e-01$ & $ 2.93  $ &    $ 9.5171e-04$ & $ 4.63 $      \\  
		&    $ 40^2\times 2 $        & $ 1.0581e-03$ &$  4.38  $ &  $  1.1967e-05$ & $ 4.49  $  & $ 15^3\times 6 $ &   $ 7.5276e-02$ & $ 3.31 $   &    $ 9.1641e-05 $ & $ 5.77  $     \\  
		&    $ 50^2\times 2 $        &  $ 4.0162e-04$ & $ 4.34$    &  $ 4.4235e-06$ &$  4.46   $   &  $ 20^3\times 6 $ &  $ 2.6651e-02$ & $ 3.61 $ &  $ 2.1964e-05$ & $ 4.97 $  \\  
		&       								 &  &     &    &      & $ 25^3\times 6 $ &$    1.0084e-02$ &$  4.36 $    &    &      \\  
		&       								 &  &     &    &      & $ 30^3\times 6 $ &   $ 4.7041e-03 $ & $ 4.18 $     &    &      \\  
		\addlinespace
		\multirow[t]{4}{*}{$ 5 $}&     $ 20^2\times 2 $   & $ 5.5445e-03 $ &     --    &  $ 3.6064e-05$   &     --     & $ 5^3\times 6 $ &  $ 1.3629e+00 $ &  --     &  $ 8.5285e-03 $  &     --     \\ 
		&    $ 30^2\times 2 $        & $ 7.0851e-04$ & $ 5.07  $   &  $ 4.1256e-06$ & $ 5.35 $     & $ 10^3\times 6 $ &$ 1.1097e-01$ & $ 3.62  $  &  $  1.7507e-04$ &$  5.61 $  \\  
		&    $ 40^2\times 2 $        & $ 1.5876e-04$ & $ 5.20  $   & $  9.1075e-07$ &$  5.25  $    & $ 15^3\times 6 $ &  $ 2.2833e-02$ & $ 3.90  $  & $  1.3964e-05  $ & $  6.24 $    \\  
		&    $ 50^2\times 2 $        & $ 5.0034e-05$ & $ 5.17   $  & $  2.7645e-07 $ & $ 5.34 $  & $ 20^3\times 6 $ &$ 5.2878e-03$ & $ 5.08  $   &  $ 2.2985e-06$  &$   6.27$     \\  
		&       								 &  &     &    &      & $ 25^3\times 6 $ &  $1.8199e-03$ & $4.78$ &    &      \\  
		&       								 &  &     &    &      & $ 30^3\times 6 $ &  $6.8777e-04$& $5.34$    &    &      \\  
		\bottomrule
	\end{tabular}
\end{table*}
\newpage
\section{Results for the Taylor--Green vortex problem}\label{sec:tgv figures}
	\begin{figure}[h!]
	\centering
	\captionsetup{belowskip=0pt, aboveskip=0pt}
	\begin{subfigure}{0.30\textwidth}
		\centering
		\includegraphics[scale=0.285]{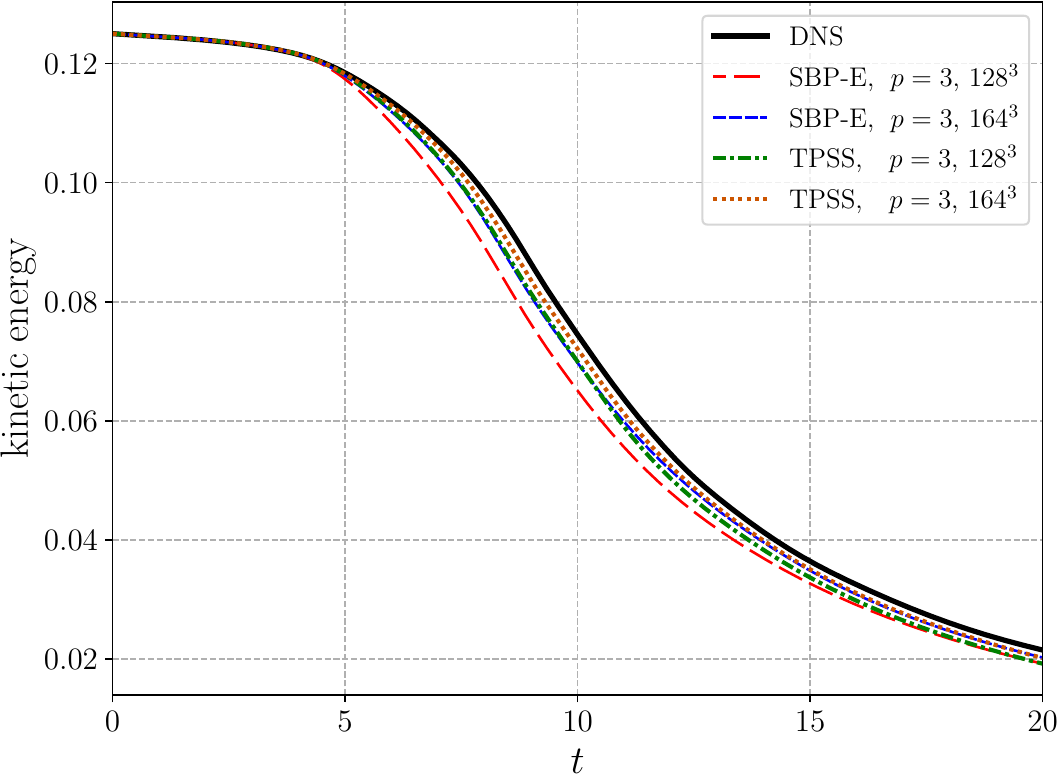}
	\end{subfigure}\hfill
	\begin{subfigure}{0.30\textwidth}
		\centering
		\includegraphics[scale=0.285]{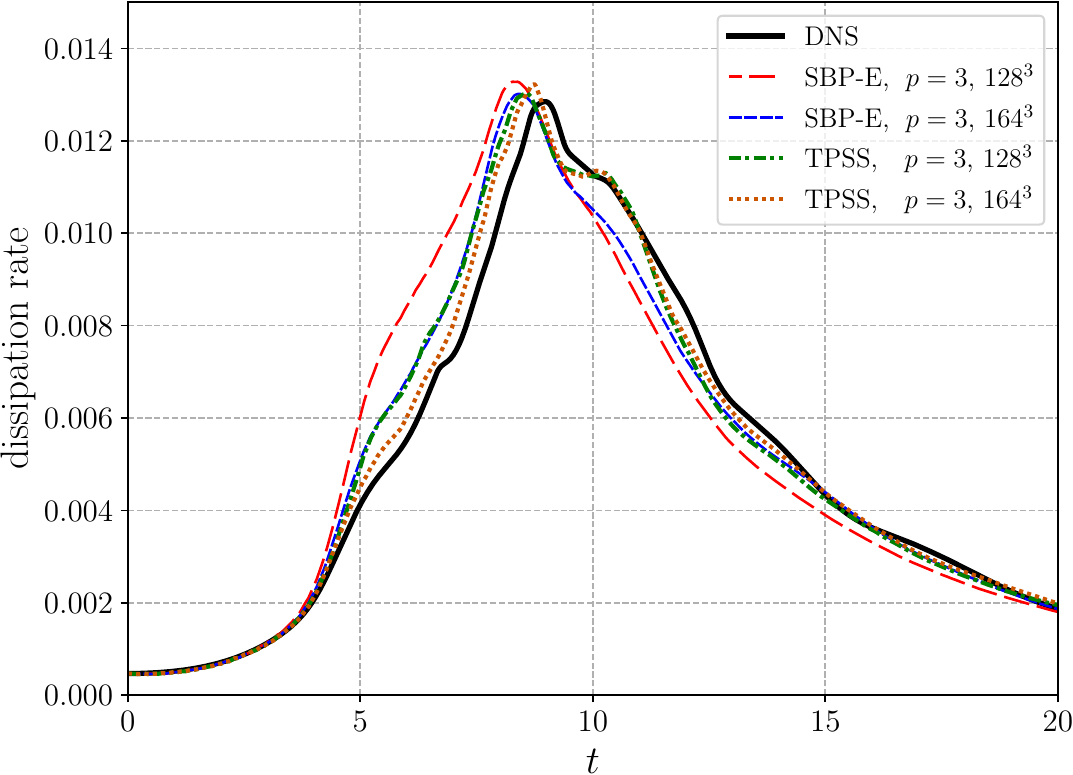}
	\end{subfigure} 
	\hfill
	\begin{subfigure}{0.30\textwidth}
		\centering
		\includegraphics[scale=0.285]{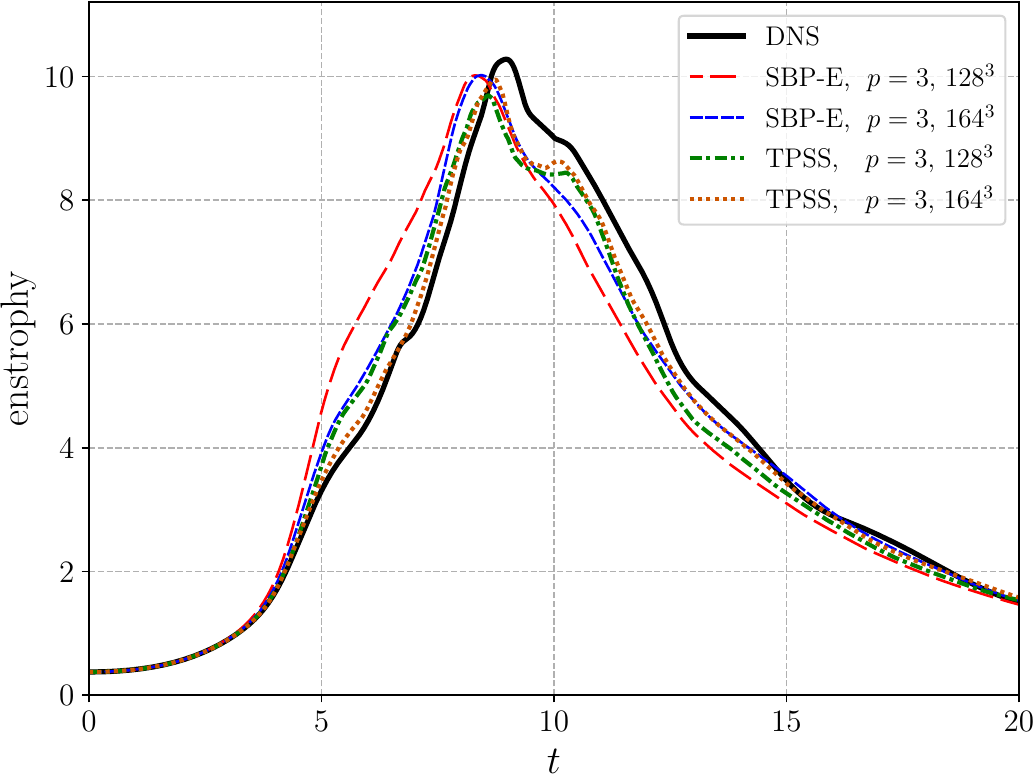}
	\end{subfigure} 
	\vspace{0.2cm}
	\\
	\begin{subfigure}{0.30\textwidth}
		\centering
		\includegraphics[scale=0.285]{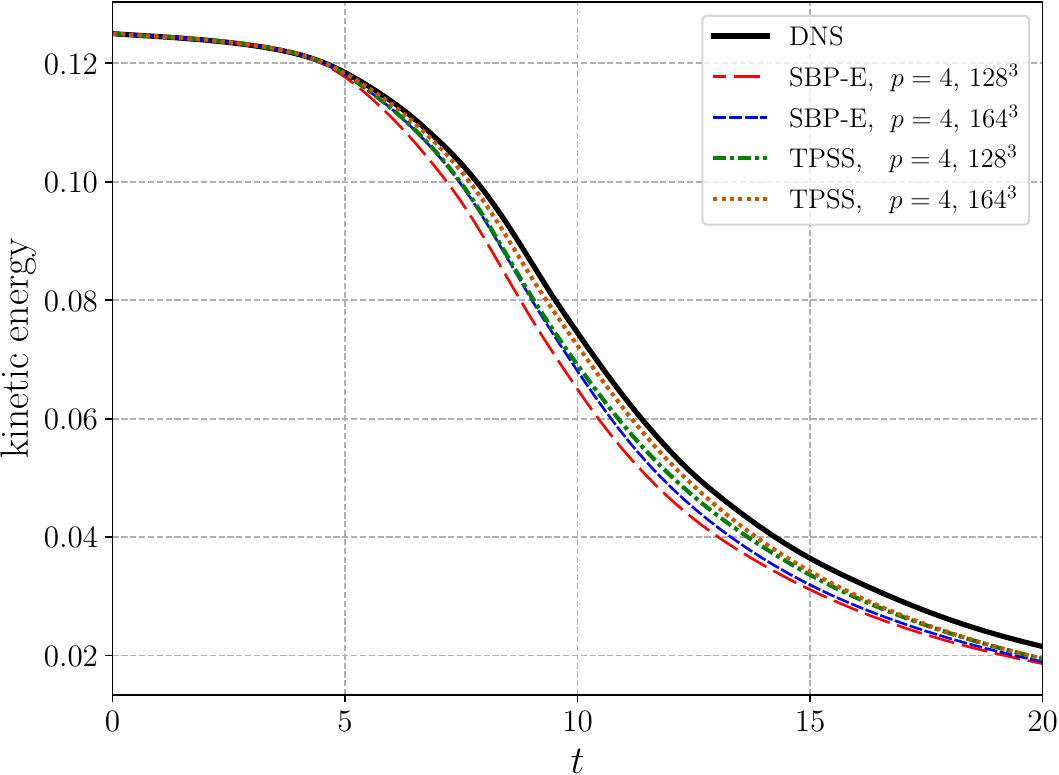}
	\end{subfigure}\hfill
	\begin{subfigure}{0.30\textwidth}
		\centering
		\includegraphics[scale=0.285]{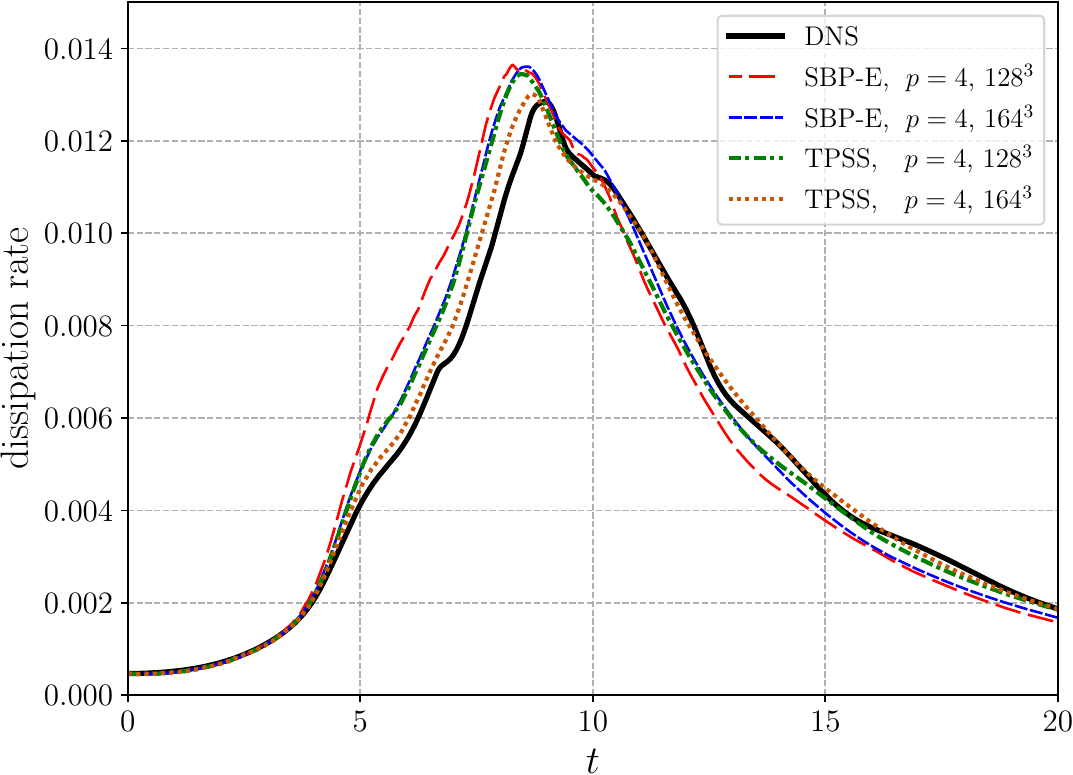}
	\end{subfigure} 
	\hfill
	\begin{subfigure}{0.30\textwidth}
		\centering
		\includegraphics[scale=0.285]{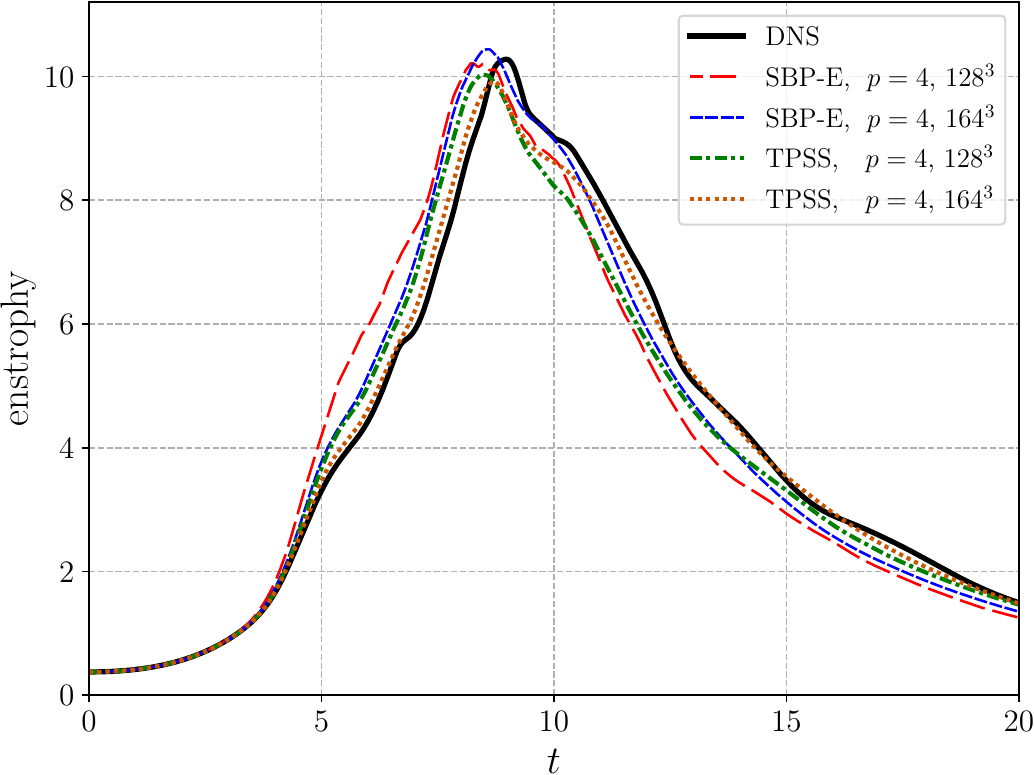}
	\end{subfigure} 
	\vspace{0.2cm}
	\\
	\begin{subfigure}{0.30\textwidth}
		\centering
		\includegraphics[scale=0.285]{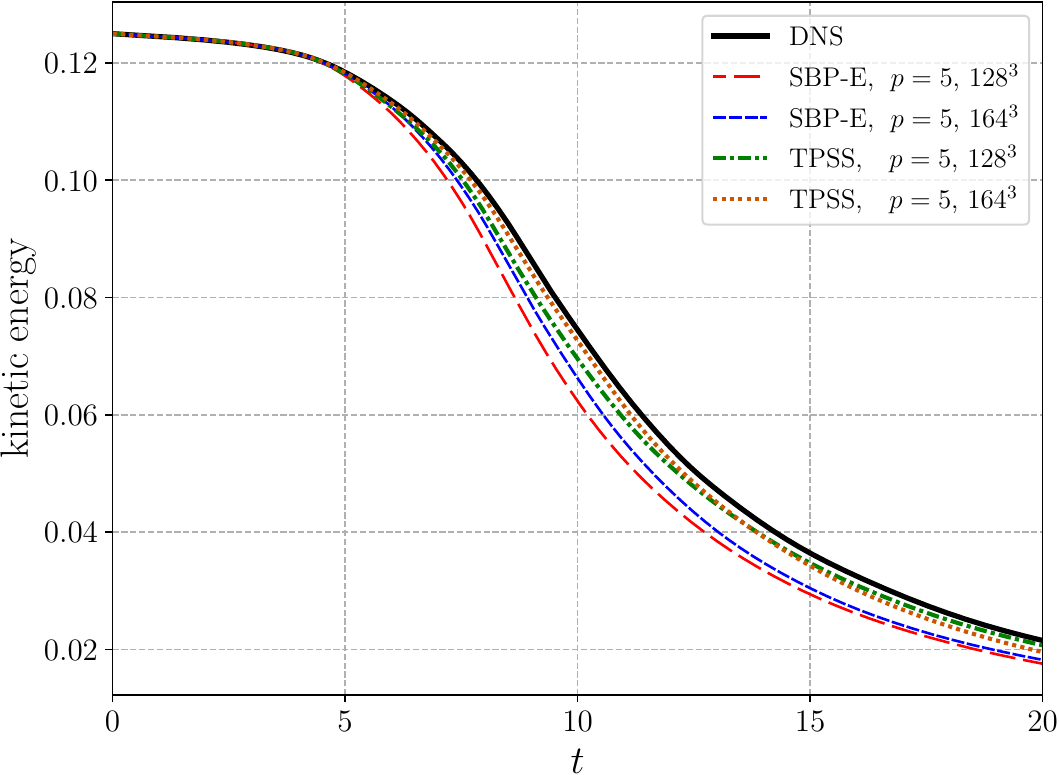}
	\end{subfigure}\hfill
	\begin{subfigure}{0.30\textwidth}
		\centering
		\includegraphics[scale=0.285]{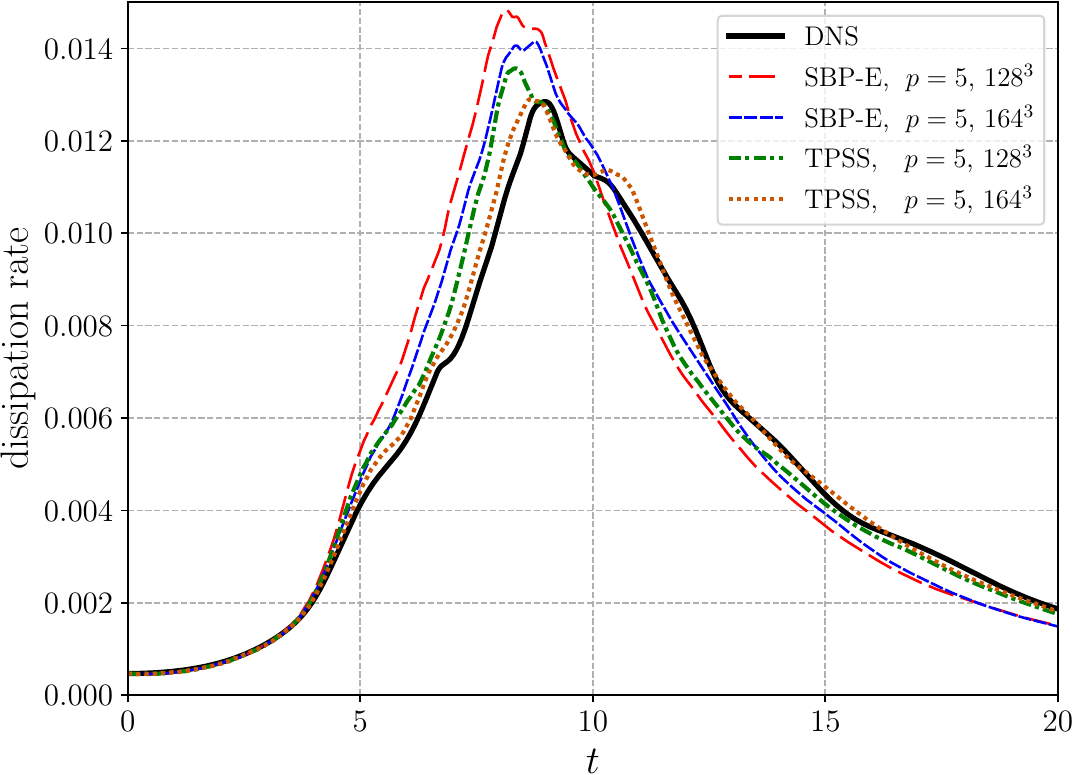}
	\end{subfigure} 
	\hfill
	\begin{subfigure}{0.30\textwidth}
		\centering
		\includegraphics[scale=0.285]{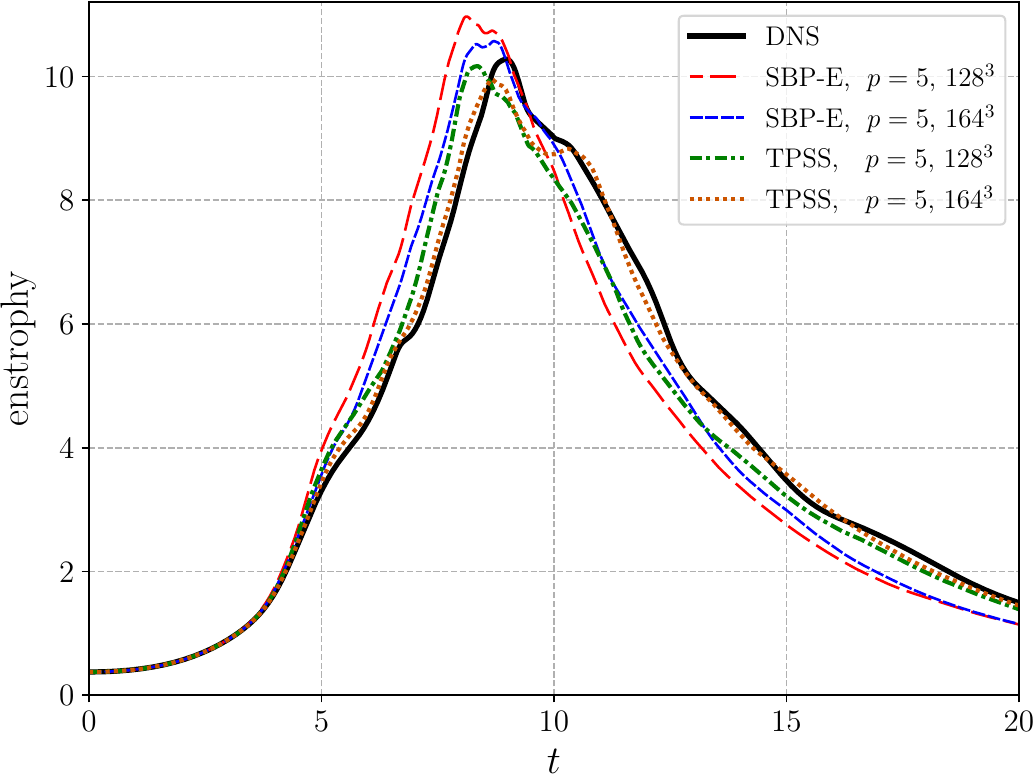}
	\end{subfigure} 
	\vspace{0.2cm}
	\\
	\begin{subfigure}{0.30\textwidth}
		\centering
		\includegraphics[scale=0.285]{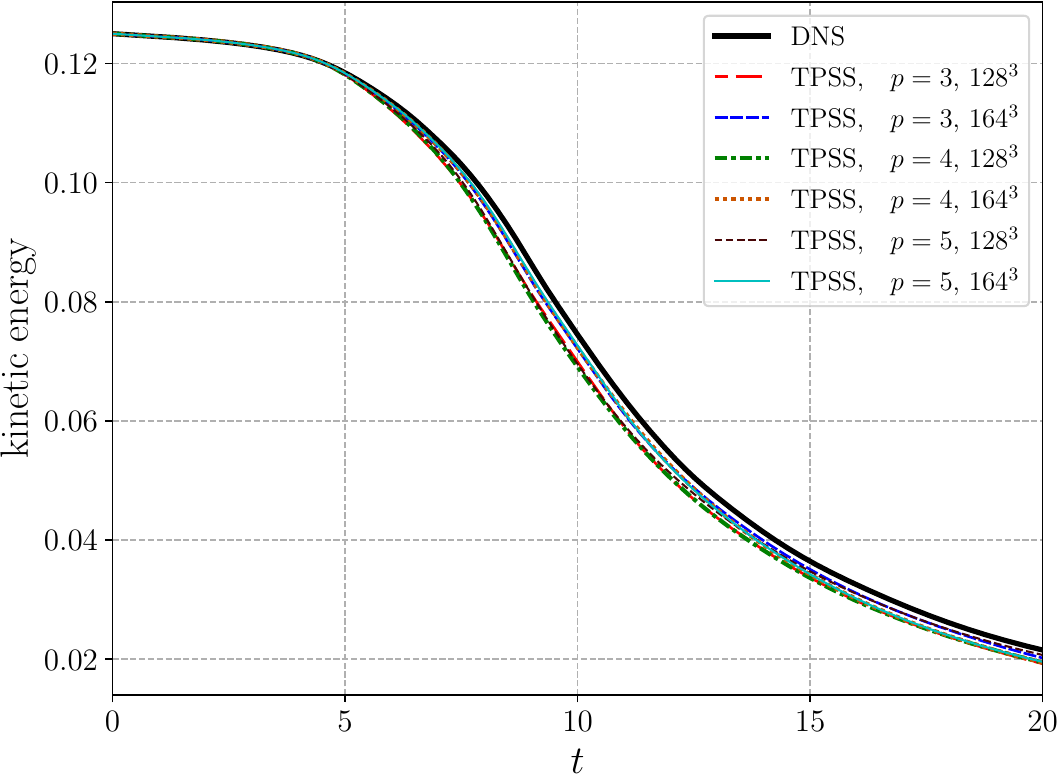}
	\end{subfigure}\hfill
	\begin{subfigure}{0.30\textwidth}
		\centering
		\includegraphics[scale=0.285]{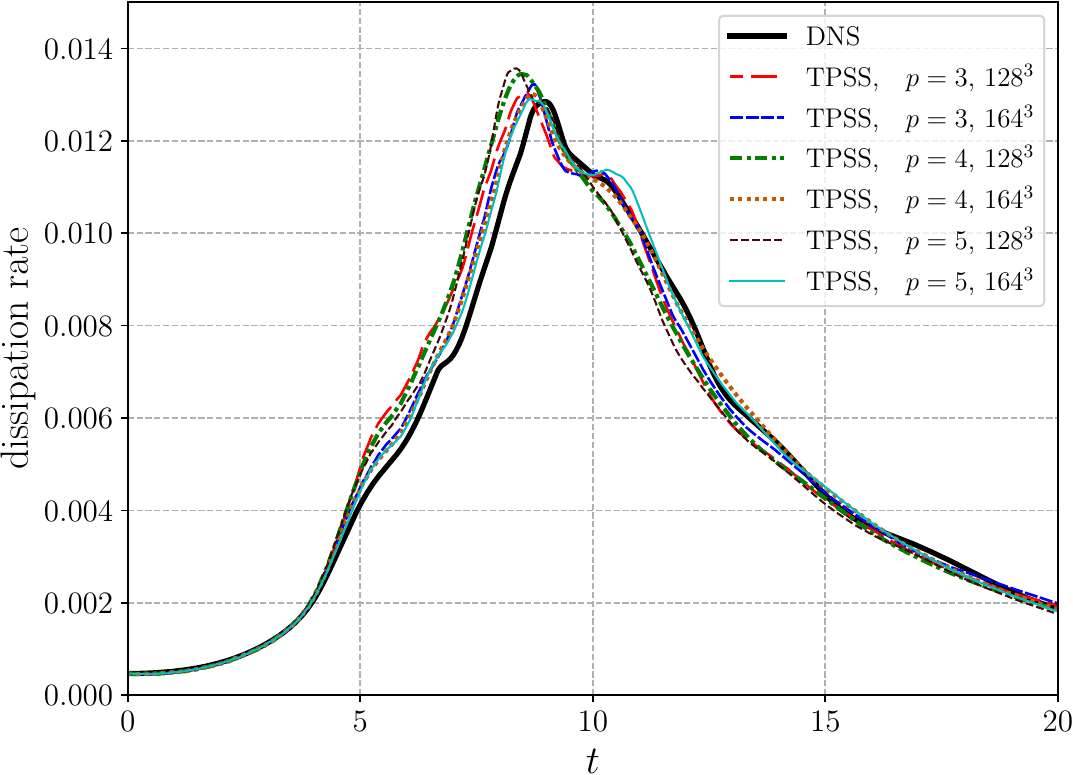}
	\end{subfigure} 
	\hfill
	\begin{subfigure}{0.30\textwidth}
		\centering
		\includegraphics[scale=0.285]{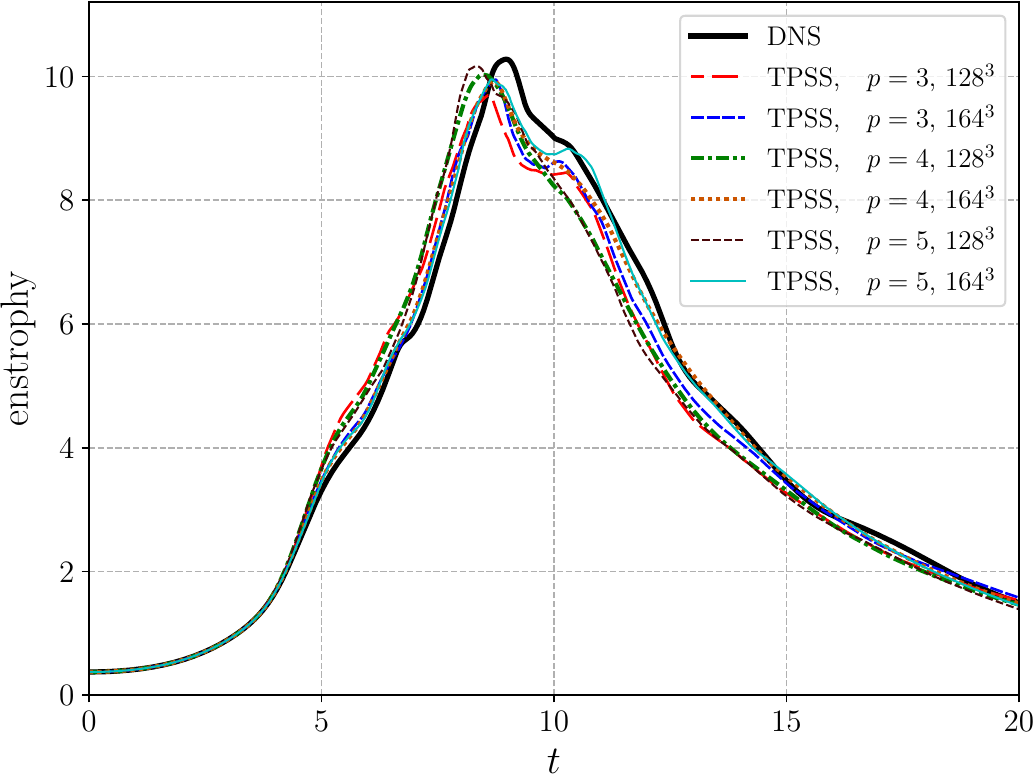}
	\end{subfigure} 
	\caption{\label{fig:tgv} Evolution of the kinetic energy, dissipation rate, and enstrophy of the Taylor--Green vortex problem at $Re=1600$. The cubed numbers in the labels indicate the nominal number of degrees of freedom in the mesh.}
\end{figure}

\bibliographystyle{model1-num-names}
{\small
	\bibliography{references}

\begin{thebibliography}{48}
\expandafter\ifx\csname natexlab\endcsname\relax\def\natexlab#1{#1}\fi
\providecommand{\url}[1]{\texttt{#1}}
\providecommand{\href}[2]{#2}
\providecommand{\path}[1]{#1}
\providecommand{\DOIprefix}{doi:}
\providecommand{\ArXivprefix}{arXiv:}
\providecommand{\URLprefix}{URL: }
\providecommand{\Pubmedprefix}{pmid:}
\providecommand{\doi}[1]{\href{http://dx.doi.org/#1}{\path{#1}}}
\providecommand{\Pubmed}[1]{\href{pmid:#1}{\path{#1}}}
\providecommand{\bibinfo}[2]{#2}
\ifx\xfnm\relax \def\xfnm[#1]{\unskip,\space#1}\fi
\bibitem[{Wang et~al.(2013)Wang, Fidkowski, Abgrall, Bassi, Caraeni, Cary,
  Deconinck, Hartmann, Hillewaert, Huynh et~al.}]{wang2013high}
\bibinfo{author}{Z.~J. Wang}, \bibinfo{author}{K.~Fidkowski},
  \bibinfo{author}{R.~Abgrall}, \bibinfo{author}{F.~Bassi},
  \bibinfo{author}{D.~Caraeni}, \bibinfo{author}{A.~Cary},
  \bibinfo{author}{H.~Deconinck}, \bibinfo{author}{R.~Hartmann},
  \bibinfo{author}{K.~Hillewaert}, \bibinfo{author}{H.~T. Huynh}, et~al.,
\newblock \bibinfo{title}{High-order {CFD} methods: current status and
  perspective},
\newblock \bibinfo{journal}{International Journal for Numerical Methods in
  Fluids} \bibinfo{volume}{72} (\bibinfo{year}{2013})
  \bibinfo{pages}{811--845}.
\bibitem[{Del Rey~Fern{\'a}ndez et~al.(2014)Del Rey~Fern{\'a}ndez, Hicken, and
  Zingg}]{fernandez2014review}
\bibinfo{author}{D.~C. Del Rey~Fern{\'a}ndez}, \bibinfo{author}{J.~E. Hicken},
  \bibinfo{author}{D.~W. Zingg},
\newblock \bibinfo{title}{Review of summation-by-parts operators with
  simultaneous approximation terms for the numerical solution of partial
  differential equations},
\newblock \bibinfo{journal}{Computers \& Fluids} \bibinfo{volume}{95}
  (\bibinfo{year}{2014}) \bibinfo{pages}{171--196}.
\bibitem[{Sv{\"a}rd and Nordstr{\"o}m(2014)}]{svard2014review}
\bibinfo{author}{M.~Sv{\"a}rd}, \bibinfo{author}{J.~Nordstr{\"o}m},
\newblock \bibinfo{title}{Review of summation-by-parts schemes for
  initial--boundary-value problems},
\newblock \bibinfo{journal}{Journal of Computational Physics}
  \bibinfo{volume}{268} (\bibinfo{year}{2014}) \bibinfo{pages}{17--38}.
\bibitem[{Hicken et~al.(2016)Hicken, Del Rey~Fern{\'a}ndez, and
  Zingg}]{hicken2016multidimensional}
\bibinfo{author}{J.~E. Hicken}, \bibinfo{author}{D.~C. Del Rey~Fern{\'a}ndez},
  \bibinfo{author}{D.~W. Zingg},
\newblock \bibinfo{title}{Multidimensional summation-by-parts operators:
  General theory and application to simplex elements},
\newblock \bibinfo{journal}{SIAM Journal on Scientific Computing}
  \bibinfo{volume}{38} (\bibinfo{year}{2016}) \bibinfo{pages}{A1935--A1958}.
\bibitem[{Del Rey~Fern{\'a}ndez et~al.(2018)Del Rey~Fern{\'a}ndez, Hicken, and
  Zingg}]{fernandez2018simultaneous}
\bibinfo{author}{D.~C. Del Rey~Fern{\'a}ndez}, \bibinfo{author}{J.~E. Hicken},
  \bibinfo{author}{D.~W. Zingg},
\newblock \bibinfo{title}{Simultaneous approximation terms for
  multi-dimensional summation-by-parts operators},
\newblock \bibinfo{journal}{Journal of Scientific Computing}
  \bibinfo{volume}{75} (\bibinfo{year}{2018}) \bibinfo{pages}{83--110}.
\bibitem[{Chen and Shu(2017)}]{chen2017entropy}
\bibinfo{author}{T.~Chen}, \bibinfo{author}{C.-W. Shu},
\newblock \bibinfo{title}{Entropy stable high order discontinuous {Galerkin}
  methods with suitable quadrature rules for hyperbolic conservation laws},
\newblock \bibinfo{journal}{Journal of Computational Physics}
  \bibinfo{volume}{345} (\bibinfo{year}{2017}) \bibinfo{pages}{427--461}.
\bibitem[{Crean et~al.(2018)Crean, Hicken, Del Rey~Fern{\'a}ndez, Zingg, and
  Carpenter}]{crean2018entropy}
\bibinfo{author}{J.~Crean}, \bibinfo{author}{J.~E. Hicken},
  \bibinfo{author}{D.~C. Del Rey~Fern{\'a}ndez}, \bibinfo{author}{D.~W. Zingg},
  \bibinfo{author}{M.~H. Carpenter},
\newblock \bibinfo{title}{Entropy-stable summation-by-parts discretization of
  the {E}uler equations on general curved elements},
\newblock \bibinfo{journal}{Journal of Computational Physics}
  \bibinfo{volume}{356} (\bibinfo{year}{2018}) \bibinfo{pages}{410--438}.
\bibitem[{Worku et~al.(2023)Worku, Hicken, and Zingg}]{worku2023quadrature}
\bibinfo{author}{Z.~A. Worku}, \bibinfo{author}{J.~E. Hicken},
  \bibinfo{author}{D.~W. Zingg},
\newblock \bibinfo{title}{Quadrature rules on triangles and tetrahedra for
  multidimensional summation-by-parts operators},
\newblock \bibinfo{journal}{Submitted to Journal of Scientific Computing, ArXiv
  Preprint arXiv:2311.15576}  (\bibinfo{year}{2023}).
\bibitem[{Sherwin and Karniadakis(1995)}]{sherwin1995triangular}
\bibinfo{author}{S.~J. Sherwin}, \bibinfo{author}{G.~E. Karniadakis},
\newblock \bibinfo{title}{A triangular spectral element method; applications to
  the incompressible {Navier-Stokes} equations},
\newblock \bibinfo{journal}{Computer Methods in Applied Mechanics and
  Engineering} \bibinfo{volume}{123} (\bibinfo{year}{1995})
  \bibinfo{pages}{189--229}.
\bibitem[{Sherwin and Karniadakis(1996)}]{sherwin1996tetrahedralhpfinite}
\bibinfo{author}{S.~J. Sherwin}, \bibinfo{author}{G.~E. Karniadakis},
\newblock \bibinfo{title}{Tetrahedral hp finite elements: {A}lgorithms and flow
  simulations},
\newblock \bibinfo{journal}{Journal of Computational Physics}
  \bibinfo{volume}{124} (\bibinfo{year}{1996}) \bibinfo{pages}{14--45}.
\bibitem[{Montoya and Zingg(2023{\natexlab{a}})}]{montoya2023efficient}
\bibinfo{author}{T.~Montoya}, \bibinfo{author}{D.~W. Zingg},
\newblock \bibinfo{title}{Efficient tensor-product spectral-element operators
  with the summation-by-parts property on curved triangles and tetrahedra},
\newblock \bibinfo{journal}{Journal of Scientific Computing (Accepted), ArXiv
  Preprint arXiv:2306.05975}  (\bibinfo{year}{2023}{\natexlab{a}}).
\bibitem[{Montoya and Zingg(2023{\natexlab{b}})}]{montoya2024efficient}
\bibinfo{author}{T.~Montoya}, \bibinfo{author}{D.~W. Zingg},
\newblock \bibinfo{title}{Efficient entropy-stable discontinuous
  spectral-element methods using tensor-product summation-by-parts operators on
  triangles and tetrahedra},
\newblock \bibinfo{journal}{Submitted to Journal of Computational Physics,
  ArXiv Preprint arXiv:2312.07874}  (\bibinfo{year}{2023}{\natexlab{b}}).
\bibitem[{Durufl{\'e} et~al.(2009)Durufl{\'e}, Grob, and
  Joly}]{durufle2009influence}
\bibinfo{author}{M.~Durufl{\'e}}, \bibinfo{author}{P.~Grob},
  \bibinfo{author}{P.~Joly},
\newblock \bibinfo{title}{Influence of {Gauss} and {Gauss-Lobatto} quadrature
  rules on the accuracy of a quadrilateral finite element method in the time
  domain},
\newblock \bibinfo{journal}{Numerical Methods for Partial Differential
  Equations: An International Journal} \bibinfo{volume}{25}
  (\bibinfo{year}{2009}) \bibinfo{pages}{526--551}.
\bibitem[{Owen(1998)}]{owen1998survey}
\bibinfo{author}{S.~J. Owen},
\newblock \bibinfo{title}{A survey of unstructured mesh generation technology},
\newblock in: \bibinfo{booktitle}{Proceedings of the International Meshing
  Roundtable}, \bibinfo{year}{1998}, pp. \bibinfo{pages}{239--267}.
\bibitem[{Castel et~al.(2009)Castel, Cohen, and
  Durufl{\'e}}]{castel2009application}
\bibinfo{author}{N.~Castel}, \bibinfo{author}{G.~Cohen},
  \bibinfo{author}{M.~Durufl{\'e}},
\newblock \bibinfo{title}{Application of discontinuous {Galerkin} spectral
  method on hexahedral elements for aeroacoustic},
\newblock \bibinfo{journal}{Journal of Computational Acoustics}
  \bibinfo{volume}{17} (\bibinfo{year}{2009}) \bibinfo{pages}{175--196}.
\bibitem[{Dalcin et~al.(2019)Dalcin, Rojas, Zampini, Fern{\'a}ndez, Carpenter,
  and Parsani}]{dalcin2019conservative}
\bibinfo{author}{L.~Dalcin}, \bibinfo{author}{D.~Rojas},
  \bibinfo{author}{S.~Zampini}, \bibinfo{author}{D.~C. D.~R. Fern{\'a}ndez},
  \bibinfo{author}{M.~H. Carpenter}, \bibinfo{author}{M.~Parsani},
\newblock \bibinfo{title}{Conservative and entropy stable solid wall boundary
  conditions for the compressible {Navier--Stokes} equations: {Adiabatic} wall
  and heat entropy transfer},
\newblock \bibinfo{journal}{Journal of Computational Physics}
  \bibinfo{volume}{397} (\bibinfo{year}{2019}) \bibinfo{pages}{108775}.
\bibitem[{Al~Jahdali et~al.(2023)Al~Jahdali, Kortas, Shaikh, Dalcin, and
  Parsani}]{al2023evaluation}
\bibinfo{author}{R.~Al~Jahdali}, \bibinfo{author}{S.~Kortas},
  \bibinfo{author}{M.~Shaikh}, \bibinfo{author}{L.~Dalcin},
  \bibinfo{author}{M.~Parsani},
\newblock \bibinfo{title}{Evaluation of next-generation high-order compressible
  fluid dynamic solver on cloud computing for complex industrial flows},
\newblock \bibinfo{journal}{Array} \bibinfo{volume}{17} (\bibinfo{year}{2023})
  \bibinfo{pages}{100268}.
\bibitem[{Del Rey~Fern{\'a}ndez et~al.(2014)Del Rey~Fern{\'a}ndez, Boom, and
  Zingg}]{fernandez2014generalized}
\bibinfo{author}{D.~C. Del Rey~Fern{\'a}ndez}, \bibinfo{author}{P.~D. Boom},
  \bibinfo{author}{D.~W. Zingg},
\newblock \bibinfo{title}{A generalized framework for nodal first derivative
  summation-by-parts operators},
\newblock \bibinfo{journal}{Journal of Computational Physics}
  \bibinfo{volume}{266} (\bibinfo{year}{2014}) \bibinfo{pages}{214--239}.
\bibitem[{Gassner(2013)}]{gassner2013skew}
\bibinfo{author}{G.~J. Gassner},
\newblock \bibinfo{title}{A skew-symmetric discontinuous {Galerkin} spectral
  element discretization and its relation to {SBP-SAT} finite difference
  methods},
\newblock \bibinfo{journal}{SIAM Journal on Scientific Computing}
  \bibinfo{volume}{35} (\bibinfo{year}{2013}) \bibinfo{pages}{A1233--A1253}.
\bibitem[{Glaubitz et~al.(2023)Glaubitz, Klein, Nordström, and
  Öffner}]{glaubitz2023multi}
\bibinfo{author}{J.~Glaubitz}, \bibinfo{author}{S.-C. Klein},
  \bibinfo{author}{J.~Nordström}, \bibinfo{author}{P.~Öffner},
\newblock \bibinfo{title}{Multi-dimensional summation-by-parts operators for
  general function spaces: {Theory} and construction},
\newblock \bibinfo{journal}{Journal of Computational Physics}
  \bibinfo{volume}{491} (\bibinfo{year}{2023}) \bibinfo{pages}{112370}.
\bibitem[{Vinokur(1974)}]{vinokur1974conservation}
\bibinfo{author}{M.~Vinokur},
\newblock \bibinfo{title}{Conservation equations of gasdynamics in curvilinear
  coordinate systems},
\newblock \bibinfo{journal}{Journal of Computational Physics}
  \bibinfo{volume}{14} (\bibinfo{year}{1974}) \bibinfo{pages}{105--125}.
\bibitem[{Thomas and Lombard(1979)}]{thomas1979geometric}
\bibinfo{author}{P.~D. Thomas}, \bibinfo{author}{C.~K. Lombard},
\newblock \bibinfo{title}{Geometric conservation law and its application to
  flow computations on moving grids},
\newblock \bibinfo{journal}{AIAA Journal} \bibinfo{volume}{17}
  (\bibinfo{year}{1979}) \bibinfo{pages}{1030--1037}.
\bibitem[{Kopriva(2006)}]{kopriva2006metric}
\bibinfo{author}{D.~A. Kopriva},
\newblock \bibinfo{title}{Metric identities and the discontinuous spectral
  element method on curvilinear meshes},
\newblock \bibinfo{journal}{Journal of Scientific Computing}
  \bibinfo{volume}{26} (\bibinfo{year}{2006}) \bibinfo{pages}{301--327}.
\bibitem[{Chan and Wilcox(2019)}]{chan2019discretely}
\bibinfo{author}{J.~Chan}, \bibinfo{author}{L.~C. Wilcox},
\newblock \bibinfo{title}{On discretely entropy stable weight-adjusted
  discontinuous {Galerkin} methods: curvilinear meshes},
\newblock \bibinfo{journal}{Journal of Computational Physics}
  \bibinfo{volume}{378} (\bibinfo{year}{2019}) \bibinfo{pages}{366--393}.
\bibitem[{Shadpey and Zingg(2020)}]{shadpey2020entropy}
\bibinfo{author}{S.~Shadpey}, \bibinfo{author}{D.~W. Zingg},
\newblock \bibinfo{title}{Entropy-stable multidimensional summation-by-parts
  discretizations on hp-adaptive curvilinear grids for hyperbolic conservation
  laws},
\newblock \bibinfo{journal}{Journal of Scientific Computing}
  \bibinfo{volume}{82} (\bibinfo{year}{2020}) \bibinfo{pages}{70}.
\bibitem[{Hicken(2020)}]{hicken2020entropy}
\bibinfo{author}{J.~E. Hicken},
\newblock \bibinfo{title}{Entropy-stable, high-order summation-by-parts
  discretizations without interface penalties},
\newblock \bibinfo{journal}{Journal of Scientific Computing}
  \bibinfo{volume}{82} (\bibinfo{year}{2020}) \bibinfo{pages}{50}.
\bibitem[{Hicken(2021)}]{hicken2021entropy}
\bibinfo{author}{J.~E. Hicken},
\newblock \bibinfo{title}{On entropy-stable discretizations and the entropy
  adjoint},
\newblock \bibinfo{journal}{Journal of Scientific Computing}
  \bibinfo{volume}{86} (\bibinfo{year}{2021}) \bibinfo{pages}{36}.
\bibitem[{Mattsson et~al.(2014)Mattsson, Almquist, and
  Carpenter}]{mattson2014optimal}
\bibinfo{author}{K.~Mattsson}, \bibinfo{author}{M.~Almquist},
  \bibinfo{author}{M.~H. Carpenter},
\newblock \bibinfo{title}{Optimal diagonal-norm {SBP} operators},
\newblock \bibinfo{journal}{Journal of Computational Physics}
  \bibinfo{volume}{264} (\bibinfo{year}{2014}) \bibinfo{pages}{91--111}.
\bibitem[{Diener et~al.(2007)Diener, Dorband, Schnetter, and
  Tiglio}]{diener2007optimized}
\bibinfo{author}{P.~Diener}, \bibinfo{author}{E.~N. Dorband},
  \bibinfo{author}{E.~Schnetter}, \bibinfo{author}{M.~Tiglio},
\newblock \bibinfo{title}{Optimized high-order derivative and dissipation
  operators satisfying summation by parts, and applications in
  three-dimensional multi-block evolutions},
\newblock \bibinfo{journal}{Journal of Scientific Computing}
  \bibinfo{volume}{32} (\bibinfo{year}{2007}) \bibinfo{pages}{109--145}.
\bibitem[{Glaubitz et~al.(2023)Glaubitz, Nordstr{\"o}m, and
  {\"O}ffner}]{glaubitz2023summation}
\bibinfo{author}{J.~Glaubitz}, \bibinfo{author}{J.~Nordstr{\"o}m},
  \bibinfo{author}{P.~{\"O}ffner},
\newblock \bibinfo{title}{Summation-by-parts operators for general function
  spaces},
\newblock \bibinfo{journal}{SIAM Journal on Numerical Analysis}
  \bibinfo{volume}{61} (\bibinfo{year}{2023}) \bibinfo{pages}{733--754}.
\bibitem[{Kreiss and Scherer(1974)}]{kreiss1974finite}
\bibinfo{author}{H.-O. Kreiss}, \bibinfo{author}{G.~Scherer},
\newblock \bibinfo{title}{Finite element and finite difference methods for
  hyperbolic partial differential equations},
\newblock in: \bibinfo{booktitle}{Mathematical Aspects of Finite Elements in
  Partial Differential Equations}, \bibinfo{publisher}{Elsevier},
  \bibinfo{year}{1974}, pp. \bibinfo{pages}{195--212}.
\bibitem[{Strand(1994)}]{strand1994summation}
\bibinfo{author}{B.~Strand},
\newblock \bibinfo{title}{Summation by parts for finite difference
  approximations for d/dx},
\newblock \bibinfo{journal}{Journal of Computational Physics}
  \bibinfo{volume}{110} (\bibinfo{year}{1994}) \bibinfo{pages}{47--67}.
\bibitem[{Pietroni et~al.(2022)Pietroni, Campen, Sheffer, Cherchi, Bommes, Gao,
  Scateni, Ledoux, Remacle, and Livesu}]{pietroni2022hex}
\bibinfo{author}{N.~Pietroni}, \bibinfo{author}{M.~Campen},
  \bibinfo{author}{A.~Sheffer}, \bibinfo{author}{G.~Cherchi},
  \bibinfo{author}{D.~Bommes}, \bibinfo{author}{X.~Gao},
  \bibinfo{author}{R.~Scateni}, \bibinfo{author}{F.~Ledoux},
  \bibinfo{author}{J.~Remacle}, \bibinfo{author}{M.~Livesu},
\newblock \bibinfo{title}{Hex-mesh generation and processing: a survey},
\newblock \bibinfo{journal}{ACM transactions on graphics} \bibinfo{volume}{42}
  (\bibinfo{year}{2022}) \bibinfo{pages}{1--44}.
\bibitem[{Lyness and Jespersen(1975)}]{lyness1975moderate}
\bibinfo{author}{J.~N. Lyness}, \bibinfo{author}{D.~Jespersen},
\newblock \bibinfo{title}{Moderate degree symmetric quadrature rules for the
  triangle},
\newblock \bibinfo{journal}{IMA Journal of Applied Mathematics}
  \bibinfo{volume}{15} (\bibinfo{year}{1975}) \bibinfo{pages}{19--32}.
\bibitem[{Wang and Papanicolopulos(2023)}]{wang2023explicit}
\bibinfo{author}{W.~Wang}, \bibinfo{author}{S.-A. Papanicolopulos},
\newblock \bibinfo{title}{Explicit consistency conditions for fully symmetric
  cubature on the tetrahedron},
\newblock \bibinfo{journal}{Engineering with Computers} \bibinfo{volume}{39}
  (\bibinfo{year}{2023}) \bibinfo{pages}{4013--4024}.
\bibitem[{Sanderson and Curtin(2016)}]{sanderson2016armadillo}
\bibinfo{author}{C.~Sanderson}, \bibinfo{author}{R.~Curtin},
\newblock \bibinfo{title}{Armadillo: a template-based {C++} library for linear
  algebra},
\newblock \bibinfo{journal}{Journal of Open Source Software}
  \bibinfo{volume}{1} (\bibinfo{year}{2016}) \bibinfo{pages}{26}.
\bibitem[{Baumann and Oden(1999)}]{baumann1999discontinuous}
\bibinfo{author}{C.~E. Baumann}, \bibinfo{author}{J.~T. Oden},
\newblock \bibinfo{title}{A discontinuous hp finite element method for
  convection--diffusion problems},
\newblock \bibinfo{journal}{Computer Methods in Applied Mechanics and
  Engineering} \bibinfo{volume}{175} (\bibinfo{year}{1999})
  \bibinfo{pages}{311--341}.
\bibitem[{Worku and Zingg(2021)}]{worku2021simultaneous}
\bibinfo{author}{Z.~A. Worku}, \bibinfo{author}{D.~W. Zingg},
\newblock \bibinfo{title}{Simultaneous approximation terms and functional
  accuracy for diffusion problems discretized with multidimensional
  summation-by-parts operators},
\newblock \bibinfo{journal}{Journal of Computational Physics}
  \bibinfo{volume}{445} (\bibinfo{year}{2021}) \bibinfo{pages}{110634}.
\bibitem[{Zingg and Lomax(1993)}]{zingg1993finite}
\bibinfo{author}{D.~W. Zingg}, \bibinfo{author}{H.~Lomax},
\newblock \bibinfo{title}{Finite-difference schemes on regular triangular
  grids},
\newblock \bibinfo{journal}{Journal of Computational Physics}
  \bibinfo{volume}{108} (\bibinfo{year}{1993}) \bibinfo{pages}{306--313}.
\bibitem[{Chang et~al.(2017)Chang, Chang, and Venkatachari}]{chang2017cause}
\bibinfo{author}{S.-C. Chang}, \bibinfo{author}{C.-L. Chang},
  \bibinfo{author}{B.~S. Venkatachari},
\newblock \bibinfo{title}{Cause and cure-deterioration in accuracy of {CFD}
  simulations with use of high-aspect-ratio triangular/tetrahedral grids},
\newblock in: \bibinfo{booktitle}{23rd AIAA Computational Fluid Dynamics
  Conference, AIAA 2017-4293}, \bibinfo{year}{2017}.
\bibitem[{Sun et~al.(2012)Sun, Darmofal, and Haimes}]{sun2012impact}
\bibinfo{author}{H.~Sun}, \bibinfo{author}{D.~L. Darmofal},
  \bibinfo{author}{R.~Haimes},
\newblock \bibinfo{title}{On the impact of triangle shapes for boundary layer
  problems using high-order finite element discretization},
\newblock \bibinfo{journal}{Journal of Computational Physics}
  \bibinfo{volume}{231} (\bibinfo{year}{2012}) \bibinfo{pages}{541--557}.
\bibitem[{Blacker(2000)}]{blacker2000meeting}
\bibinfo{author}{T.~Blacker},
\newblock \bibinfo{title}{Meeting the challenge for automated conformal
  hexahedral meshing},
\newblock in: \bibinfo{booktitle}{9th International Meshing Roundtable},
  \bibinfo{year}{2000}, pp. \bibinfo{pages}{11--20}.
\bibitem[{Erlebacher et~al.(1997)Erlebacher, Hussaini, and
  Shu}]{erlebacher1997interaction}
\bibinfo{author}{G.~Erlebacher}, \bibinfo{author}{M.~Y. Hussaini},
  \bibinfo{author}{C.-W. Shu},
\newblock \bibinfo{title}{Interaction of a shock with a longitudinal vortex},
\newblock \bibinfo{journal}{Journal of Fluid Mechanics} \bibinfo{volume}{337}
  (\bibinfo{year}{1997}) \bibinfo{pages}{129--153}.
\bibitem[{Chan et~al.(2019)Chan, Del Rey~Fern{\'a}ndez, and
  Carpenter}]{chan2019efficient}
\bibinfo{author}{J.~Chan}, \bibinfo{author}{D.~C. Del Rey~Fern{\'a}ndez},
  \bibinfo{author}{M.~H. Carpenter},
\newblock \bibinfo{title}{Efficient entropy stable {Gauss} collocation
  methods},
\newblock \bibinfo{journal}{SIAM Journal on Scientific Computing}
  \bibinfo{volume}{41} (\bibinfo{year}{2019}) \bibinfo{pages}{A2938--A2966}.
\bibitem[{Williams and Jameson(2013)}]{williams2013nodal}
\bibinfo{author}{D.~M. Williams}, \bibinfo{author}{A.~Jameson},
\newblock \bibinfo{title}{Nodal points and the nonlinear stability of
  high-order methods for unsteady flow problems on tetrahedral meshes},
\newblock in: \bibinfo{booktitle}{21st AIAA Computational Fluid Dynamics
  Conference, AIAA 2013-2830}, \bibinfo{year}{2013}.
\bibitem[{DeBonis(2013)}]{debonis2013solutions}
\bibinfo{author}{J.~DeBonis},
\newblock \bibinfo{title}{Solutions of the {Taylor-Green} vortex problem using
  high-resolution explicit finite difference methods},
\newblock in: \bibinfo{booktitle}{51st AIAA Aerospace Sciences Meeting
  Including the New Horizons Forum and Aerospace Exposition, AIAA 2013-0382},
  \bibinfo{year}{2013}.
\bibitem[{Dairay et~al.(2017)Dairay, Lamballais, Laizet, and
  Vassilicos}]{dairay2017numerical}
\bibinfo{author}{T.~Dairay}, \bibinfo{author}{E.~Lamballais},
  \bibinfo{author}{S.~Laizet}, \bibinfo{author}{J.~C. Vassilicos},
\newblock \bibinfo{title}{Numerical dissipation vs. subgrid-scale modelling for
  large eddy simulation},
\newblock \bibinfo{journal}{Journal of Computational Physics}
  \bibinfo{volume}{337} (\bibinfo{year}{2017}) \bibinfo{pages}{252--274}.
\bibitem[{Ponce et~al.(2019)Ponce, van Zon, Northrup, Gruner, Chen, Ertinaz,
  Fedoseev, Groer, Mao, Mundim et~al.}]{ponce2019deploying}
\bibinfo{author}{M.~Ponce}, \bibinfo{author}{R.~van Zon},
  \bibinfo{author}{S.~Northrup}, \bibinfo{author}{D.~Gruner},
  \bibinfo{author}{J.~Chen}, \bibinfo{author}{F.~Ertinaz},
  \bibinfo{author}{A.~Fedoseev}, \bibinfo{author}{L.~Groer},
  \bibinfo{author}{F.~Mao}, \bibinfo{author}{B.~C. Mundim}, et~al.,
\newblock \bibinfo{title}{Deploying a top-100 supercomputer for large parallel
  workloads: The {Niagara} supercomputer},
\newblock in: \bibinfo{booktitle}{Proceedings of the Practice and Experience in
  Advanced Research Computing on Rise of the Machines (learning)},
  \bibinfo{year}{2019}, pp. \bibinfo{pages}{1--8}.

\end{thebibliography}
}
\addcontentsline{toc}{section}{\refname}

\end{document}